\documentclass[a4j,10pt,showkeys]{article}

\usepackage{amsmath,amsthm,amssymb}
\usepackage{geometry}
\usepackage{hyperref}
\usepackage{cancel}
\usepackage{mathtools}
\usepackage{bm}
\usepackage{color}
\usepackage{mathrsfs}

\newtheorem{theo}{Theorem}[section]
\newtheorem{lemm}[theo]{Lemma}
\newtheorem{prop}[theo]{Proposition}
\newtheorem{cor}[theo]{Corollary}

\theoremstyle{definition}
\newtheorem{defi}[theo]{Definition}
\newtheorem{rem}[theo]{Remark}

\newcommand{\p}{\mathbb{P}}
\newcommand{\e}{\mathbb{E}}
\newcommand{\real}{\mathbb{R}}
\newcommand{\n}{\mathbb{N}}

\newcommand{\rd}{\,\mathrm{d}}

\newcommand{\cY}{\mathcal{Y}}
\newcommand{\cZ}{\mathcal{Z}}
\newcommand{\cH}{\mathcal{H}}
\newcommand{\cG}{\mathcal{G}}
\newcommand{\relmiddle}[1]{\mathrel{}\middle#1\mathrel{}}
\newcommand{\1}{\mbox{\rm{1}}\hspace{-0.25em}\mbox{\rm{l}}}
\providecommand{\keywords}[1]{\textbf{Keywords:} #1}

\allowdisplaybreaks[0]

\makeatletter

\@addtoreset{equation}{section}
\makeatother
\def\widebar{\accentset{{\cc@style\underline{\mskip10mu}}}}
\numberwithin{equation}{section}
\def\rnum#1{\expandafter{\romannumeral #1}} 
\def\Rnum#1{\uppercase\expandafter{\romannumeral #1}}

\allowdisplaybreaks

\title{
Approximations for adapted M-solutions of Type-\Rnum{2} backward stochastic Volterra integral equations
}
\author{
	Yushi Hamaguchi\thanks{
	Graduate School of Engineering Science,
	Osaka University.
	1-3, Machikaneyama-cho, Toyonaka,
	Osaka,
	560-8531,
	Japan.
	\href{mailto:hmgch2950@gmail.com}{hmgch2950@gmail.com}}
	\quad and \quad
	Dai Taguchi\thanks{
	Research Institute for Interdisciplinary Science,
	Department of mathematics,
	Okayama University.
	3-1-1
	Tsushima-naka,
	Kita-ku
	Okayama,
	700-8530,
	Japan.
	\href{mailto:dai.taguchi.dai@gmail.com}{dai.taguchi.dai@gmail.com}}
}

\begin{document}
\maketitle

\begin{abstract}
In this paper, we study a class of Type-\Rnum{2} backward stochastic Volterra integral equations (BSVIEs). For the adapted M-solutions, we obtain two approximation results, namely, a BSDE approximation and a numerical approximation. The BSDE approximation means that the solution of a finite system of backward stochastic differential equations (BSDEs) converges to the adapted M-solution of the original equation. As a consequence of the BSDE approximation, we obtain an estimate for the $L^2$-time regularity of the adapted M-solutions of Type-\Rnum{2} BSVIEs. For the numerical approximation, we provide a backward Euler--Maruyama scheme, and show that the scheme converges in the strong $L^2$-sense with the convergence speed of order $1/2$. These results hold true without any differentiability conditions for the coefficients.
\end{abstract}
\keywords
Backward stochastic Volterra integral equations;
adapted M-solutions;
BSDE approximations;
Euler--Maruyama scheme;
$L^{2}$-time regularity.
\\
\textbf{2020 Mathematics Subject Classification}: 60H20; 65C30; 60H07.

%65C30 Numerical solutions to stochastic differential and integral equations
%60H20 Stochastic integral equations
%60H07 Stochastic calculus of variations and the Malliavin calculus

%65C05 Monte Carlo methods
%65C30:Numerical solutions to stochastic differential and integral equations
%60H35 Computational methods for stochastic equations
%41A25 Rate of convergence, degree of approximation
%91G60 Numerical methods (including Monte Carlo methods)
%65Cxx Probabilistic methods, simulation and stochastic differential equations For theoretical aspects, see 68U20 and 60H35
%60H10 Stochastic ordinary dierential equations (aspects of stochastic analysis)
%45D05 Volterra integral equations

%\tableofcontents

\section{Introduction}\label{sec_Intro}

Backward stochastic differential equations (BSDEs) and backward stochastic Volterra integral equations (BSVIEs) have been studied extensively and applied to many areas including stochastic control, PDE theory, mathematical finance and economics.
Linear BSDEs were first introduced by Bismut \cite{Bi78} as adjoint equations by means of the Pontryagin maximum principle for stochastic control problems of stochastic differential equations (SDEs).
Later, Pardoux and Peng \cite{PaPe90} developed systematic treatments of general nonlinear BSDEs of the following form:
\begin{equation}\label{BSDE}
Y(t)
=
\Psi
+
\int_{t}^{T}
	G
	(
		s,
		Y(s),
		Z(s)
	)
\rd s
-
\int_{t}^{T}
	Z(s)
\rd W(s),
~t\in[0,T].
\end{equation}
Here, $W(\cdot)$ is a standard Brownian motion on a complete probability space $(\Omega,\mathcal{F},\mathbb{P})$ with the filtlation $\mathbb{F}=(\mathcal{F}_t)_{t\geq0}$ generated by $W(\cdot)$, $\Psi$ is an $\mathcal{F}_T$-measurable random variable called the terminal condition, and $G$ is a progressively measurable function called the driver.
The adapted solution of BSDE~\eqref{BSDE} is the pair $(Y(\cdot),Z(\cdot))$ of adapted processes satisfying \eqref{BSDE}.
Specifically, the second component $Z(\cdot)$ of the adapted solution is called the martingale integrand.
Such an equation has been found useful in applications to, for example, recursive utilities, dynamic risk measures, nonlinear Feynman--Kac formula and path-dependent PDEs.
We refer the readers to the textbook of Zhang \cite{Zh17} and the survey paper of El Karoui, Peng and Quenez \cite{ElKaPeQu97} for the detailed account of theory and applications of BSDEs.

As a natural extension of BSDEs, BSVIEs of the form
\begin{equation}\label{Type-1}
Y(t)
=
\Psi(t)
+
\int_{t}^{T}
	G(t,s,Y(s),Z(t,s))
\rd s
-
\int_{t}^{T}
	Z(t,s)
\rd W(s),
~t\in[0,T],
\end{equation}
were introduced by Lin \cite{Li02} and Yong \cite{Yo06} and further studied in \cite{Yo07,Yo08,WaT12,ShWa12,ShWaYo13,WaTYo19,Po20} among others.
The stochastic process $\Psi(\cdot)$, which is called the free term in the literature of BSVIEs, is a family of $\mathcal{F}_T$-measurable random variables $\Psi(t)$, $t\in[0,T]$ (not necessarily $\mathbb{F}$-adapted), and the driver $G$ is progressively measurable with respect to the time parameter $s$ for each fixed $t$.
The unknown we are looking for is the pair $(Y(\cdot),Z(\cdot,\cdot))$, where $Y(\cdot)$ and $Z(t,\cdot)$ are adapted for each $t\in[0,T]$.
Yong \cite{Yo06,Yo08} also considered the following form of BSVIEs:
\begin{equation}\label{BSVIE_II_1}
Y(t)
=
\Psi(t)
+
\int_{t}^{T}
	G(t,s,Y(s),Z(t,s),Z(s,t))
\rd s
-
\int_{t}^{T}
	Z(t,s)
\rd W(s),
~t\in[0,T].
\end{equation}
In the literature, \eqref{Type-1} and \eqref{BSVIE_II_1} are referred as Type-\Rnum{1} and Type-\Rnum{2} BSVIEs, respectively. Unlike Type-\Rnum{1} BSVIE~\eqref{Type-1}, the solution of Type-\Rnum{2} BSVIE~\eqref{BSVIE_II_1} needs an additional constraint on the term $Z(t,s)$, $0\leq s\leq t\leq T$, for the well-posedness of the equation.
Inspired by the duality principle appearing in stochastic control problems of (forward) stochastic Volterra integral equations (SVIEs), the so-called \emph{adapted M-solution} was introduced and studied in \cite{Yo08} (see Definition \ref{Def_M_sol} of the present paper).
Both Type-\Rnum{1} and Type-\Rnum{2} BSVIEs have become important tools to study some problems in stochastic control and mathematical finance.
For example, Yong \cite{Yo07}, Wang, Sun and Yong \cite{WaSuYo19} and Agram \cite{Ag19} applied BSVIEs to dynamic risk measures for the so-called position processes.
Kromer and Overbeck \cite{KrOv17} investigated dynamic capital allocations via BSVIEs.
Beissner and Rosazza Gianin \cite{BeRG21} applied BSVIEs to arbitrage-free asset pricing via a path of EMMs, called an EMM-string.
Stochastic control problems for systems of SVIEs and BSVIEs were studied by Shi, Wang and Yong~\cite{ShWaYo15} and Wang and Zhang \cite{WaTZh17}.
Also, it is worth to mention that BSVIEs have a strong connection to time-inconsistent stochastic control problems.
For example, time-inconsistent recursive utility processes of general discounting can be modelled by the solutions of BSVIEs (see \cite{WaSuYo19}).
Wang and Yong \cite{WaYo21} and Hamaguchi \cite{Ha21} studied time-inconsistent problems where the cost functionals were defined by the solutions of Type-\Rnum{1} BSVIEs.
In \cite{WaYo21}, generalizing the earlier study of Yong \cite{Yo12}, they derived the so-called equilibrium HJB equation which characterizes the closed-loop equilibrium strategy.
In \cite{Ha21}, the author characterized the open-loop equilibrium controls by variational methods, where the adjoint equations turned out to be Type-\Rnum{1} BSVIEs of an extended form.

For relationships between BSVIEs and PDEs, we note that Wang and Yong \cite{WaTYo19} derived the representation PDEs for both Type-\Rnum{1} and Type-\Rnum{2} BSVIEs.
Further studies on this topic were developed by Wang \cite{Wa20} and Wang, Yong and Zhang \cite{WaYoZh20}.
It is remarkable that the representation PDEs of BSVIEs have the same structure as the equilibrium HJB equations appearing in time-inconsistent stochastic control problems obtained by \cite{Yo12,WaYo21}.
Here let us briefly recall the results of \cite{WaTYo19}.
They considered Type-\Rnum{1} and Type-\Rnum{2} BSVIEs of the following forms:
\begin{equation}\label{Type-1 Markov}
\begin{split}
Y(t)
=
\psi(t,X(t),X(T))
+
\int_{t}^{T}
	g(t,s,X(t),X(s),Y(s),Z(t,s))
\rd s
-
\int_{t}^{T}
	Z(t,s)
\rd W(s),
~t\in[0,T],
\end{split}
\end{equation}
and
\begin{equation}\label{Type-2 Markov}
\begin{split}
Y(t)
=
\psi(t,X(t),X(T))
+
\int_{t}^{T}
	g(t,s,X(t),X(s),Y(s),Z(t,s),Z(s,t))
\rd s
-
\int_{t}^{T}
	Z(t,s)
\rd W(s),
~t\in[0,T],
\end{split}
\end{equation}
respectively, where $\psi$ and $g$ are given deterministic functions, and $X(\cdot)$ is the solution of the SDE
\begin{equation}\label{SDE}
X(t)
=
x
+
\int_{0}^{t}
	b(s,X(s))
\rd s
+
\int_{0}^{t}
	\sigma(s,X(s))
\rd W(s),
~t \in[0,T],
\end{equation}
with given deterministic functions $b,\,\sigma$ and a given initial condition $x$.
First, they showed that Type-\Rnum{1} BSVIE~\eqref{Type-1 Markov} can be approximated by finite systems of BSDEs.
Then, based on this approximation result and the well-known representation PDEs for BSDEs, they derived the representation PDE of Type-\Rnum{1} BSVIE~\eqref{Type-1 Markov}.
Finally, they obtained the representation PDE of Type-\Rnum{2} BSVIE~\eqref{Type-2 Markov} by connecting that of a Type-\Rnum{1} BSVIE and the usual linear PDE induced by the martingale representation theorem.
We remark that, unlike the case of Type-\Rnum{1} BSVIE~\eqref{Type-1 Markov}, the well-posedness of the representation PDE of Type-\Rnum{2} BSVIE~\eqref{Type-2 Markov} was proved only in the framework of the mild solution, not the classical (smooth) solution.
We also remark that, unlike Type-\Rnum{1} BSVIE~\eqref{Type-1 Markov}, there have not been any approximation results for Type-\Rnum{2} BSVIE~\eqref{Type-2 Markov} by systems of BSDEs.
We guess that this is an important problem since the BSDE approximation of BSVIEs would give some probabilistic interpretations to the corresponding (non-standard) PDE systems.

In general, it is difficult to obtain the explicit forms of the solutions of nonlinear BSDEs and nonlinear BSVIEs.
To make matter worse, even in the linear case, the explicit solutions of Type-\Rnum{2} BSVIEs have not been obtained yet to the best of our knowledge.
Hence, in order to calculate the solutions, it is important to consider numerical approximations of BSDEs and BSVIEs.
For BSDEs, there have been many attempts to provide numerical approximations, see for example \cite{DoMaPr96,Na02,BaPa03,BoTo04,Zh04,GoLeWa05,DeMe06,LeGoWa06,BeDe07,GoLa07,GoMa10,BrDeMe11,PeXu11,GoLe17,LiDRSz18} among others.
However, for BSVIEs, the numerical method is quit limited.
Here we mention \cite{BePo13,WaY18}.
On one hand, Bender and Pokalyuk \cite{BePo13} studied a numerical method for the following Type-\Rnum{1} BSVIE:
\begin{equation*}
Y(t)
=
\psi(t,W)
+
\int_{t}^{T}
	g(s,Y(s))
\rd s
-
\int_{t}^{T}
	Z(t,s)
\rd W(s),
~t\in[0,T],
\end{equation*}
which is weakly approximated by a sequence of discrete BSVIEs driven by a binary random walk.
On the other hand, Wang \cite{WaY18} constructed a kind of backward Euler--Maruyama schemes for Type-\Rnum{1} BSVIEs of the form
\begin{equation}\label{Type-1 numeric}
Y(t)
=
\psi(t,X(T))
+
\int_{t}^{T}
	g(t,s,X(s),Y(s),Z(t,s))
\rd s
-
\int_{t}^{T}
	Z(t,s)
\rd W(s),
~t\in[0,T],
\end{equation}
with $X(\cdot)$ being the solution of an SVIE.
He showed that, in the cases of $g(t,s,x,y,z)=g(t,s,x,y)$ or $g(t,s,x,y,z)=g(t,s,x,z)$, the scheme converges in the strong $L^2$-sense to the solution of \eqref{Type-1 numeric} with the convergence speed of order $1/2$.
We note that in \cite{WaY18} the coefficients were assumed to be smooth. Also, to the best of our knowledge, the numerical method for Type-\Rnum{1} BSVIE~\eqref{Type-1 numeric} (or more generally \eqref{Type-1 Markov}) of the general form of $g$ was not obtained in the literature.
Moreover, the problem of numerical approximations for Type-\Rnum{2} BSVIEs has been completely open.

The purpose of this paper is to show two approximation results for Type-\Rnum{2} BSVIEs, that is, a \emph{BSDE approximation} and a \emph{numerical approximation}.
The BSDE approximation means that the solution of a finite system of standard BSDEs converges to the adapted M-solution of the original Type-\Rnum{2} BSVIE.
For the numerical approximation, we provide a backward Euler--Maruyama scheme, and show that the scheme converges in the strong $L^2$-sense with the convergence speed of order $1/2$.
Our main results are the following:
\begin{itemize}
\item
%Under assumptions $(\mathrm{H}_{\psi,g})$ and $(\mathrm{H}_{b,\sigma})$ for the coefficients (see Section \ref{sec_L2_reg}), the term
For the general Type-\Rnum{2} BSVIE \eqref{BSVIE_II_1} with stochastic coefficients $\Psi$ and $G$, the term
\begin{equation}\label{error 1}
\sum^{N-1}_{k=0}
\int^{t_{k+1}}_{t_k}
	\mathbb{E}
	\Bigl[
		\big|Y(t)-\mathscr{Y}^\pi(t_k,t)\big|^2
		+
		\int^T_0
			\big|Z(t,s)-\mathscr{Z}^\pi(t_k,s)\big|^2
		\rd s
	\Bigr]\rd t,
\end{equation}
which is defined for each time mesh $\pi=\{t_{0},t_{1},\dots,t_{N}\}$ of $[0,T]$,  tends to zero as the mesh size $|\pi|$ of $\pi$ tends to zero, where $(\mathscr{Y}^{\pi},\mathscr{Z}^{\pi})$ is the solution of a BSDE system corresponding to $\pi$ (see Theorem \ref{Theo_BSDE_0}).

\item
For Type-\Rnum{2} BSVIE \eqref{Type-2 Markov} with $X(\cdot)$ being the solution of SDE \eqref{SDE}, the term
\begin{equation}\label{error 2}
\sum_{k=0}^{N-1}
\e
\Big[
	\int_{t_{k}}^{t_{k+1}}
		\big|
			Y(t)
			-
			Y^{\pi}(t_k,t_k)
		\big|^2
	\rd t
\Big]
+
\sum_{k=0}^{N-1}
\sum_{\ell=0}^{N-1}
\e
\Big[
	\int_{t_{k}}^{t_{k+1}}
		\int_{t_\ell}^{t_{\ell+1}}
			\big|
				Z(t,s)
				-
				Z^{\pi}(t_{k},t_{\ell})
			\big|^{2}
		\rd s
	\rd t
\Big],
\end{equation}
which is defined for each time mesh $\pi=\{t_{0},t_{1},\dots,t_{N}\}$ of $[0,T]$, is estimated by a constant times $|\pi|$, where $(Y^{\pi},Z^{\pi})$ is a backward Euler--Maruyama scheme corresponding to $\pi$ (see Theorem \ref{Thm_2}).
\end{itemize}
We emphasize that the above approximation results hold true without any differentiability or structural conditions of the coefficients as assumed in \cite{WaY18}.

For BSDE \eqref{BSDE}, it is well-known that the so-called \emph{$L^2$-time regularity} of the martingale integrand $Z(\cdot)$ plays a central role in the study of numerical approximations (see the textbook \cite{Zh17}).
In this paper, we first consider the general Type-\Rnum{2} BSVIE~\eqref{BSVIE_II_1} with stochastic coefficients $\Psi$ and $G$, and estimate the error \eqref{error 2} in terms of the modulus of the $L^2$-time regularity of the martingale integrand $\mathscr{Z}^{\pi}$ of the approximation BSDE system (see Proposition \ref{Prop_EM_1}).
This result is important on its own right since we can apply it to equation~\eqref{Type-2 Markov} under more general settings where, for example, $X(\cdot)$ is the solution of an SVIE, as well as where the coefficients $\psi$ and $g$ are ``irregular'' (in the sense of \cite{GoMa10}) in terms of $X(\cdot)$. Such generalizations are, however, beyond the scope of this paper.
In order to investigate the $L^2$-time regularity of $\mathscr{Z}^{\pi}$, we consider Type-\Rnum{2} BSVIE~\eqref{Type-2 Markov} with $X(\cdot)$ being the solution of SDE~\eqref{SDE}.
Under this setting, by using the Malliavin calculus technique, we first represent $\mathscr{Z}^{\pi}$ in terms of the solution of a \emph{variational BSDE system} (see \eqref{Prop_L2_2}), which is new even in the case of Type-\Rnum{1} BSVIEs.
Then we provide some key $L^p$-estimates ($p\geq 2$) for the variational BSDE system.
Finally, we provide an estimate for the modulus of the $L^2$-time regularity of $\mathscr{Z}^{\pi}$, which concludes our second main result mentioned above.

Compared with Type-\Rnum{1} BSVIEs, the treatment of the time regularity of Type-\Rnum{2} BSVIEs is difficult due to the dependency on $Z(s,t)$ of the driver.
On one hand, in \cite{Yo08} the continuity (in the strong $L^2$-sense) of the adapted M-solution of the general Type-\Rnum{2} BSVIE~\eqref{BSVIE_II_1} was proved under technical differentiability assumptions for the coefficients.
On the other hand, in this paper, as a corollary of the BSDE approximation we obtain a quantitative estimate for the modulus of the $L^2$-time regularity of the adapted M-solution of Type-\Rnum{2} BSVIE \eqref{Type-2 Markov} with $X(\cdot)$ being the solution of SDE \eqref{SDE}, without smoothness of the coefficients (see Theorem \ref{Thm_1}).
This kind of regularity estimate for adapted M-solutions appears for the first time in the literature of BSVIEs.
Also, we provide $L^p$-a priori estimate ($p\geq2$) for adapted M-solutions of the general Type-\Rnum{2} BSVIE~\eqref{BSVIE_II_1} with stochastic coefficients $\Psi$ and $G$ (see Theorem~\ref{Lem_Lp} which is proved in \hyperref[sec_Appendix]{Appendix}).
This $L^p$-estimate is also new, and it is found useful for the analysis of the $L^2$-time regularity of $\mathscr{Z}^\pi$.

The paper is organized as follows:
In Section \ref{sec_Preli}, we introduce some notation and prove fundamental inequalities which we use throughout this paper.
In Section \ref{sec_BSDE}, we construct BSDE systems with stochastic coefficients $\Psi$ and $G$, and prove the BSDE approximation.
In Section \ref{sec_Disc_TypeII}, we construct the backward Euler--Maruyama scheme for the general Type-\Rnum{2} BSVIE~\eqref{BSVIE_II_1} with stochastic coefficients $\Psi$ and $G$.
Under this general setting, we estimate the error \eqref{error 2} in terms of the modulus of the $L^2$-time regularity of the martingale integrand $\mathscr{Z}^{\pi}$ of the approximation BSDE system.
In order to estimate the modulus of the $L^{2}$-time regularity of $\mathscr{Z}^{\pi}$, in Section \ref{sec_L2_reg}, we consider the case of Type-\Rnum{2} BSVIE~\eqref{Type-2 Markov} with $X(\cdot)$ being the solution of SDE~\eqref{SDE}.
We give precise statements of our main theorems.
In \hyperref[sec_Appendix]{Appendix}, we provide a proof of $L^p$-a priori estimate for Type-\Rnum{2} BSVIEs.

%%%%%%%%%%%%%%%
%% Section
%%%%%%%%%%%%%%%

\section{Preliminaries}\label{sec_Preli}

\subsection{Notation}\label{sec_func}
Let $W(\cdot)$ be a $d$-dimensional standard Brownian motion on a complete probability space $(\Omega,\mathcal{F},\p)$.
$\mathbb{F}=(\mathcal{F}_{t})_{t \geq 0}$ denotes the augmentation of the filtration generated by $W(\cdot)$.
$\1_A$ denotes the indicator function for a given set $A$, $\mathbb{E}[\cdot]$ denotes the expectation, and $\mathbb{E}_t[\cdot]:=\mathbb{E}[\cdot|\mathcal{F}_t]$ denotes the conditional expectation with respect to $\mathcal{F}_t$ for each $t\geq0$.

For each $d_{1},d_{2} \in \n$, we denote the space of $(d_{1} \times d_{2})$-matrices by $\real^{d_{1} \times d_{2}}$, which is endowed with the Frobenius norm denoted by $|\cdot|$.
We define $\real^{d_{1}}:=\real^{d_{1} \times 1}$, that is, each element of $\real^{d_{1}}$ is understood as a column vector.
For each  matrix $A$, $A^{\top}$ denotes the transpose of $A$.

Throughout this paper, we fix $T \in (0,\infty)$.
We define $\Delta [0,T]:=\{(t,s) \in [0,T]^{2}~|~0 \leq t \leq s \leq T\}$ and $\Delta^{\mathrm{c}} [0,T]:=\{(t,s) \in [0,T]^{2}~|~0 \leq s < t \leq T\}$.
$\Pi[0,T]$ denotes the set of all time meshes $\pi=\{t_{0},t_{1},\ldots,t_{N}\}$ of $[0,T]$ with $N \geq 2$ and $0=t_{0}<t_{1}<\cdots<t_{N}=T$.
For each $\pi=\{t_{0},t_{1},\ldots,t_{N}\} \in \Pi[0,T]$, we define $\Delta t_{k}:=t_{k+1}-t_{k}$, $|\pi|:=\max_{k=0,\ldots, N-1} \Delta t_{k}$ and $\Delta W_{k}:=W(t_{k+1})-W(t_{k})$.
We define $\tau:[0,T) \to \{t_{0},\ldots,t_{N-1}\}$ and $\tau^{*}:[0,T) \to \{t_{1},\ldots,t_{N}\}$ by $\tau(t):=t_{k}$ and $\tau^{*}(t):=t_{k+1}$, respectively, for each $t \in [t_{k},t_{k+1})$ with $k=0,\ldots,N-1$.
Also, we define $\Delta \tau(t):=\tau^{*}(t)-\tau(t)$ for $t \in [0,T)$.

Throughout this paper, $C>0$ denotes a generic constant depending only on $L$ and $T$, where $L$ is the constant appearing in the assumptions $(\mathrm{H}_{\Psi,G})$, $(\mathrm{H}_{\Psi,G})'$, $(\mathrm{H}_{\psi,g})$, $(\mathrm{H}_{\psi,g})'$, $(\mathrm{H}_{b,\sigma})$ or $(\mathrm{H}_{b,\sigma})'$ which will be introduced later.
For a given parameter $\mu$, $C_{\mu}>0$ denotes a generic constant depending only on $\mu$, $L$ and $T$.
$C$ and $C_{\mu}$ may change from line to line.

Fix $0\leq T_{0}<T_{1}<\infty$ and $\mathcal{G} \subset \mathcal{F}$, and let $\mathbb{H}$ be a Euclidean space.
We define the following spaces:
\begin{align*}
&L^{2}_{\mathcal{G}}(\Omega;\mathbb{H})
:=
\{
	\varphi:\Omega \to \mathbb{H}
	\,|\,
	\varphi\ \text{is}\ \mathcal{G}\text{-measurable and }
	\mathbb{E}[|\varphi|^{2}]
	<\infty
\},
\\
&L^{2}_{\mathcal{G}}(T_{0},T_{1};\mathbb{H})
:=
\left\{
	\varphi:\Omega \times [T_{0},T_{1}] \to \mathbb{H}
	\relmiddle|
	\varphi(\cdot)\ \text{is}\ \mathcal{G} \otimes \mathcal{B}([T_{0},T_{1}])\text{-measurable and }
	\mathbb{E}\Big[\int_{T_{0}}^{T_{1}}\big|\varphi(t)\big|^{2} \rd t\Big]
	<\infty
\right\},\\
&L^{2}_{\mathbb{F}}(T_{0},T_{1};\mathbb{H})
:=
\left\{
	\varphi:\Omega\times[T_{0},T_{1}]\to\mathbb{H}
	\relmiddle|
	\varphi(\cdot)\ \text{is progressively measurable and }
	\mathbb{E}
	\Bigl[
		\int^{T_{1}}_{T_{0}}
			\big|\varphi(s)\big|^{2}
		\rd s
	\Bigr]
	<\infty
\right\},\\
&L^2_\mathbb{F}(\Omega;C([T_{0},T_{1}];\mathbb{H}))
:=
\left\{
	\varphi:\Omega\times[T_{0},T_{1}]\to\mathbb{H}
	\relmiddle|
	\begin{aligned}
		&\varphi(\cdot)\ \text{is progressively measurable, has continuous paths}\\
		&\text{and satisfies }
		\mathbb{E}
		\Bigl[
			\sup_{s\in[T_{0},T_{1}]}\big|\varphi(s)\big|^2
		\Bigr]
		<\infty
	\end{aligned}
\right\}.
\end{align*}
Also, we define the spaces of stochastic processes with two time parameters:
\begin{align*}
L^{2}(T_{0},T_{1};L_{\mathbb{F}}^{2}(\Omega;C([T_{0},T_{1}];\mathbb{H})))
\mathalpha{:=}
\left\{
	\varphi: \Omega \mathalpha{\times} [T_{0},T_{1}]^{2} \mathalpha{\to} \mathbb{H}
	\relmiddle|
	\begin{aligned}
		&\varphi(\cdot,\cdot)\ \text{is measurable,}\\
		&\varphi(t,\cdot) \in L_{\mathbb{F}}^{2}(\Omega;C([T_{0},T_{1}];\mathbb{H}))~\text{for a.e.\ }t \in [T_{0},T_{1}]\\
		&\text{and satisfies }
		\mathbb{E}
		\Bigl[
			\int_{T_{0}}^{T_{1}}
				\sup_{s\in[T_{0},T_{1}]}
				\big|
					\varphi(t,s)
				\big|^{2}
			\rd t
		\Bigr]
		<\infty
	\end{aligned}
\right\}
\end{align*}
and
\begin{align*}
L^{2}(T_{0},T_{1};L_{\mathbb{F}}^{2}(T_{0},T_{1};\mathbb{H}))
:=
\left\{
	\varphi: \Omega \times [T_{0},T_{1}]^{2} \to \mathbb{H}
	\relmiddle|
	\begin{aligned}
		&\varphi(\cdot,\cdot)\ \text{is measurable,}\\
		&\varphi(t,\cdot) \in L_{\mathbb{F}}^{2}(T_{0},T_{1};\mathbb{H})~\text{for a.e.\ }t \in [T_{0},T_{1}]\\
		&\text{and satisfies }
		\mathbb{E}
		\Bigl[
			\int_{T_{0}}^{T_{1}}
				\int_{T_{0}}^{T_{1}}
					\big|
						\varphi(t,s)
					\big|^{2}
				\rd s
			\rd t
		\Bigr]
		<\infty
	\end{aligned}
\right\}.
\end{align*}

\subsection{A priori estimate for BSDEs}\label{sec_apriori}
For each $0\leq T_{0}<T_{1} \leq T$, consider the following BSDE on $[T_{0},T_{1}]$:
\begin{equation}\label{BSDE_0}
Y(t)
=
\Psi
+
\int_{t}^{T_{1}}
	G(s,Y(s),Z(s))
\rd s
-
\int_{t}^{T_{1}}
	Z(s)
\rd W(s),
~t\in[T_{0},T_{1}].
\end{equation}
We say that a pair $(Y(\cdot),Z(\cdot))$ is an adapted solution of BSDE \eqref{BSDE_0} if $(Y(\cdot),Z(\cdot))\in L^{2}_{\mathbb{F}}(\Omega, C([T_{0},T_{1}];\real^{m})) \times L^{2}_{\mathbb{F}}(T_{0},T_{1};\real^{m \times d})$ and the equality \eqref{BSDE_0} holds a.s.\ for any $t\in[T_{0},T_{1}]$.
The following fact is well-known, see for example \cite{Zh17}.

\begin{lemm}\label{apriori_0}
Fix $p\geq2$.
Let $\Psi:\Omega \to \real^{m}$ and $G:\Omega\times[T_{0},T_{1}] \times \mathbb{R}^{m}\times \mathbb{R}^{m\times d}\to\mathbb{R}^{m}$ be measurable maps such that
\begin{itemize}
\item
$\Psi$ is $\mathcal{F}_{T_{1}}$-measurable, and the process $(G(s,y,z))_{s\in[T_{0},T_{1}]}$ is progressively measurable for each $y\in\mathbb{R}^{m}$ and $z\in\mathbb{R}^{m\times d}$;
\item
$\e[|\Psi|^{p}+(\int_{T_{0}}^{T_{1}}|G(s,0,0)| \rd s)^{p}]<\infty$;
\item
There exists a constant $L>0$ such that, for a.e.\,$s \in [T_{0},T_{1}]$, a.s., it holds that
\begin{equation*}
	|G(s,y_{1},z_{1})-G(s,y_{2},z_{2})|
	\leq
	L\big\{|y_{1}-y_{2}|+|z_{1}-z_{2}|\big\}
\end{equation*}
for any $y_{1},y_{2}\in\mathbb{R}^{m}$ and $z_{1},z_{2}\in\mathbb{R}^{m\times d}$.
\end{itemize}
Then there exists a unique adapted solution $(Y(\cdot),Z(\cdot))$ of BSDE \eqref{BSDE_0}, and the following estimate holds:
\begin{equation*}
\mathbb{E}
\Big[
	\sup_{s\in[T_{0},T_{1}]}
	\big|Y(s)\big|^{p}
	+
	\Big(
		\int_{T_{0}}^{T_{1}}
			\big|Z(s)\big|^{2}
		\rd s
	\Big)^{p/2}
\Big]
\leq
C_{p}
\mathbb{E}
\Big[
	\big|\Psi\big|^{p}
	+
	\Big(
		\int_{T_{0}}^{T_{1}}
			\big|G(s,0,0)\big|
		\rd s
	\Big)^{p}
\Big].
\end{equation*}
For $i=1,2$, let $(\Psi_{i},G_{i})$ satisfy the above conditions and let $(Y_{i}(\cdot),Z_{i}(\cdot))$ be the unique adapted solution of BSDE \eqref{BSDE_0} corresponding to $(\Psi_{i},G_{i})$.
Then it holds that
\begin{equation*}
\begin{split}
&\mathbb{E}
\Big[
	\sup_{s\in[T_{0},T_{1}]}
	\big|
		Y_{1}(s)
		-
		Y_{2}(s)
	\big|^{p}
	+
	\Big(
		\int_{T_{0}}^{T_{1}}
			\big|
				Z_{1}(s)
				-
				Z_{2}(s)
			\big|^{2}
		\rd s
	\Big)^{p/2}
\Big]\\
&\leq
C_{p}\mathbb{E}
\Big[
	\big|
		\Psi_{1}
		-
		\Psi_{2}
	\big|^{p}
	+
	\Big(
		\int_{T_{0}}^{T_{1}}
			\big|
				G_{1}(s,Y_{1}(s),Z_{1}(s))
				-
				G_{2}(s,Y_{1}(s),Z_{1}(s))
			\big|
		\rd s
	\Big)^{p}
\Big].
\end{split}
\end{equation*}
\end{lemm}

\subsection{Gronwall-like inequalities}\label{Sec_Gron}

In this subsection, we provide Gronwall-like inequalities which are frequently used in this paper.
The idea of the proof is inspired by \cite{ShWaYo15}, where the authors treated a weighted norm of the adapted M-solution of a BSVIE.

Let $(S,\Sigma, \mu)$ be a measure space.
We first provide a continuous version of the inequality.
\begin{lemm}\label{Lem_Gron_cont}
Let $a,b,c:[0,T] \to \real$ be nonnegative integrable functions and let $\zeta:\Delta^{\mathrm{c}}[0,T] \times S \to \real$ be a nonnegative jointly measurable function.
Assume that there exists a constant $K>0$ such that
\begin{align}
\label{Gron_0}
\displaystyle
&a(t)
\leq
K
\Big\{
	b(t)
	+
	\int_{t}^{T}
		a(s)
	\rd s
	+
	\int_{S}
		\Big(
			\int_{t}^{T}
				\zeta(s,t,x)
			\rd s
		\Big)^{2}
	\rd \mu(x)
\Big\},~\text{a.e.}~t \in [0,T],
\\
&
\displaystyle
\label{Gron_00}
\int_{S}
	\int_{0}^{t}
		\zeta(t,s,x)^{2}
	\rd s
\rd \mu(x)
\leq
K
\bigl\{
	a(t)
	+
	c(t)
\bigr\},~
\text{a.e.}~t \in [0,T].
\end{align}
Then for any $\gamma \geq 2K(1+K)$, it holds that
\begin{align*}
\int_{0}^{T}
	e^{\gamma t}
	a(t)
\rd t
\leq
2K
\int_{0}^{T}
	e^{\gamma t}
	b(t)
\rd t
+
\int_{0}^{T}
	e^{\gamma t}
	c(t)
\rd t.
\end{align*}
In particular, we have
\begin{align*}
\int_{0}^{T}
	a(t)
\rd t
\leq
(2K+1)e^{2K(1+K)T}
\int_{0}^{T}
	\big\{
		b(t)
		+
		c(t)
	\big\}
\rd t.
\end{align*}
\end{lemm}
\begin{proof}
Let $\gamma >0$ be fixed.
By the inequality \eqref{Gron_0}, we have
\begin{align*}%\label{Gron_1}
\int_{0}^{T}
	e^{\gamma t}
	a(t)
\rd t
\leq
K
\int_{0}^{T}
	e^{\gamma t}
	b(t)
\rd t
+
K
\int_{0}^{T}
	e^{\gamma t}
	\int_{t}^{T}
		a(s)
	\rd s
\rd t
+
K
\int_{0}^{T}
	e^{\gamma t}
	\int_{S}
		\Big(
			\int_{t}^{T}
				\zeta(s,t,x)
			\rd s
		\Big)^{2}
	\rd \mu(x)
	\rd t.
\end{align*}
By Fubini's theorem, we have
\begin{align*}
\int_{0}^{T}
	e^{\gamma t}
	\int_{t}^{T}
		a(s)
	\rd s
\rd t
=
\int_{0}^{T}
	a(t)
	\int_{0}^{t}
		e^{\gamma s}
	\rd s
\rd t
\leq
\frac{1}{\gamma}
\int_{0}^{T}
	e^{\gamma t}
	a(t)
\rd t.
\end{align*}
Furthermore, by using Fubini's theorem, H\"older's inequality and the inequality \eqref{Gron_00}, we have
\begin{align*}
\int_{0}^{T}
	e^{\gamma t}
	\int_{S}
		\Big(
			\int_{t}^{T}
				\zeta(s,t,x)
			\rd s
		\Big)^{2}
	\rd \mu(x)
\rd t
&=
\int_{S}
	\int_{0}^{T}
		e^{\gamma t}
		\Big(
			\int_{t}^{T}
				e^{-\frac{\gamma}{2}s}
				e^{\frac{\gamma}{2}s}
				\zeta(s,t,x)
			\rd s
		\Big)^{2}
	\rd t
\rd \mu(x)
\\&\leq
\int_{S}
	\int_{0}^{T}
		e^{\gamma t}
		\Big(
			\int_{t}^{T}
				e^{-\gamma s}
			\rd s
		\Big)
		\Big(
			\int_{t}^{T}
				e^{\gamma s}
					\zeta(s,t,x)^{2}
			\rd s
		\Big)
	\rd t
\rd \mu(x)
\\&\leq
\frac{1}{\gamma}
\int_{S}
	\int_{0}^{T}
		\int_{t}^{T}
			e^{\gamma s}
				\zeta(s,t,x)^{2}
		\rd s
	\rd t
\rd \mu(x)
\\&=
\frac{1}{\gamma}
\int_{0}^{T}
	e^{\gamma t}
	\int_{S}
		\int_{0}^{t}
				\zeta(t,s,x)^{2}
		\rd s
	\rd \mu(x)
\rd t
\\&\leq
\frac{K}{\gamma}
\int_{0}^{T}
	e^{\gamma t}
	\big\{
		a(t)
		+
		c(t)
	\big\}
\rd t.
\end{align*}
Hence we obtain
\begin{align*}
\int_{0}^{T}
	e^{\gamma t}
	a(t)
\rd t
\leq
\frac{K(1+K)}{\gamma}
\int_{0}^{T}
	e^{\gamma t}
	a(t)
\rd t
+
K
\int_{0}^{T}
	e^{\gamma t}
	b(t)
\rd t
+
\frac{K^{2}}{\gamma}
\int_{0}^{T}
	e^{\gamma t}
	c(t)
\rd t.
\end{align*}
Therefore, by choosing $\gamma \geq 2K(1+K)$, we get the assertion.
\end{proof}

The following is the discrete version of the above inequality.
The proof is similar to that of the above lemma, but we prove it for the sake of self-containedness.

\begin{lemm}\label{Lem_Gron_disc}
Let $\pi=\{t_{0},t_{1},\ldots,t_{N}\} \in \Pi[0,T]$ be fixed.
Let $\{a_{k}\}_{k=0}^{N-1}$, $\{b_{k}\}_{k=0}^{N-1}$, $\{c_{k}\}_{k=0}^{N-1}$ be nonnegative real-valued sequences, and for each $k=1,\ldots,N-1$ and $\ell=0,\ldots,k-1$, let $\zeta_{k,\ell}:S \to \real$ be a nonnegative measurable function.
Assume that there exists a constant $K>0$ such that
\begin{align}
\label{Gron_000}
&\displaystyle
a_{k}
\leq
K
\Big\{
	b_{k}
	+
	\sum_{\ell=k+1}^{N-1}
		\Delta t_{\ell}
		a_{\ell}
	+
	\int_{S}
		\Big(
			\sum_{\ell=k+1}^{N-1}
				\Delta t_{\ell}
				\zeta_{\ell,k}(x)
		\Big)^{2}
	\rd \mu(x)
\Big\},~k=0,\ldots,N-1,
\\
&
\displaystyle
\label{Gron_0000}
\int_{S}
	\sum_{\ell=0}^{k-1}
		\Delta t_{\ell}
			\zeta_{k,\ell}(x)^{2}
\rd \mu(x)
\leq
K
\big\{
	a_{k}
	+
	c_{k}
\big\},~
k=1,\ldots,N-1.
\end{align}
Then for any $\gamma \geq 2K(1+K)$, it holds that
\begin{align*}
\sum_{k=0}^{N-1}
\Gamma_{k}
a_{k}
\leq
2K
\sum_{k=0}^{N-1}
\Gamma_{k}
b_{k}
+
\sum_{k=0}^{N-1}
\Gamma_{k}
c_{k},
\end{align*}
where $\Gamma_{k}:=\int_{t_{k}}^{t_{k+1}} e^{\gamma t} \rd t$, $k=0,\ldots,N-1$.
In particular, we have
\begin{align*}
\sum_{k=0}^{N-1}
	\Delta t_{k}
	a_{k}
\leq
(2K+1)e^{2K(1+K)T}
\sum_{k=0}^{N-1}
	\Delta t_{k}
	\big\{
		b_{k}
		+
		c_{k}
	\big\}.
\end{align*}
\end{lemm}
\begin{proof}
Let $\gamma >0$ be fixed.
By the inequality \eqref{Gron_000}, we have
\begin{align*}%\label{Gron_2}
\sum_{k=0}^{N-1}
\Gamma_{k}
a_{k}
\leq
K
\sum_{k=0}^{N-1}
\Gamma_{k}
b_{k}
+
K
\sum_{k=0}^{N-2}
\Gamma_{k}
\sum_{\ell=k+1}^{N-1}
\Delta t_{\ell}
a_{\ell}
+
K
\sum_{k=0}^{N-2}
\Gamma_{k}
\int_{S}
	\Big(
		\sum_{\ell=k+1}^{N-1}
			\Delta t_{\ell}
			\zeta_{\ell,k}(x)
	\Big)^{2}
\rd \mu(x).
\end{align*}
Noting that $\Delta t_{k} e^{\gamma t_{k}} \leq \Gamma_{k} \leq \Delta t_{k} e^{\gamma t_{k+1}}$ for each $k=0,\ldots,N-1$, we have
\begin{align*}
\sum_{k=0}^{N-2}
\Gamma_{k}
\sum_{\ell=k+1}^{N-1}
\Delta t_{\ell}
a_{\ell}
=
\sum_{k=1}^{N-1}
\Delta t_{k}
a_{k}
\int_{0}^{t_{k}}
	e^{\gamma s}
\rd s
\leq
\frac{1}{\gamma}
\sum_{k=1}^{N-1}
\Delta t_{k} e^{\gamma t_{k}}
a_{k}
\leq
\frac{1}{\gamma}
\sum_{k=0}^{N-1}
\Gamma_{k}
a_{k}.
\end{align*}
Furthermore, by using H\"older's inequality and the inequality \eqref{Gron_0000}, we have
\begin{align*}
\sum_{k=0}^{N-2}
\Gamma_{k}
\int_{S}
	\Big(
		\sum_{\ell=k+1}^{N-1}
			\Delta t_{\ell}
			\zeta_{\ell,k}(x)
	\Big)^{2}
\rd \mu(x)
&=
\int_{S}
	\sum_{k=0}^{N-2}
	\Gamma_{k}
	\Big(
		\sum_{\ell=k+1}^{N-1}
			\zeta_{\ell,k}(x)
			\int_{t_{\ell}}^{t_{\ell+1}}
				e^{-\frac{\gamma}{2}s}
				e^{\frac{\gamma}{2}s}
			\rd s
	\Big)^{2}
\rd \mu(x)
\\&\leq
\int_{S}
	\sum_{k=0}^{N-2}
	\Gamma_{k}
	\Big(
		\sum_{\ell=k+1}^{N-1}
		\int_{t_{\ell}}^{t_{\ell+1}}
			e^{-\gamma s}
		\rd s
	\Big)
	\Big(
		\sum_{\ell=k+1}^{N-1}
				\zeta_{\ell,k}(x)^{2}
			\int_{t_{\ell}}^{t_{\ell+1}}
				e^{\gamma s}
			\rd s
	\Big)
\rd \mu(x)
\\&\leq
\frac{1}{\gamma}
\int_{S}
	\sum_{k=0}^{N-2}
	\Gamma_{k}
	e^{-\gamma t_{k+1}}
	\sum_{\ell=k+1}^{N-1}
	\Gamma_{\ell}
		\zeta_{\ell,k}(x)^{2}
\rd \mu(x)
\\&\leq
\frac{1}{\gamma}
\int_{S}
	\sum_{k=0}^{N-2}
		\Delta t_{k}
		\sum_{\ell=k+1}^{N-1}
		\Gamma_{\ell}
			\zeta_{\ell,k}(x)^{2}
\rd \mu(x)
\\&=
\frac{1}{\gamma}
\sum_{k=1}^{N-1}
	\Gamma_{k}
	\int_{S}
		\sum_{\ell=0}^{k-1}
		\Delta t_{\ell}
			\zeta_{k,\ell}(x)^{2}
	\rd \mu(x)
\\&\leq
\frac{K}{\gamma}
\sum_{k=1}^{N-1}
\Gamma_{k}
\big\{
	a_{k}
	+
	c_{k}
\big\}.
\end{align*}
Hence we obtain
\begin{align*}
\sum_{k=0}^{N-1}
\Gamma_{k}
a_{k}
\leq
\frac{K(1+K)}{\gamma}
\sum_{k=0}^{N-1}
\Gamma_{k}
a_{k}
+
K
\sum_{k=0}^{N-1}
\Gamma_{k}
b_{k}
+
\frac{K^{2}}{\gamma}
\sum_{k=0}^{N-1}
\Gamma_{k}
c_{k}.
\end{align*}
Therefore, by choosing $\gamma \geq 2K(1+K)$, we get the assertion.
\end{proof}

\section{BSDE approximations for  general Type-\Rnum{2} BSVIEs}\label{sec_BSDE}

In this and the next sections, we consider the general Type-\Rnum{2} BSVIE \eqref{BSVIE_II_1} with the stochastic coefficients $(\Psi,G)$.
We first give basic assumptions for the free term $\Psi$ and the driver $G$.

\begin{itemize}
\item[$(\mathrm{H}_{\Psi,G})$]
\begin{itemize}
\item[(i)]
$\Psi:\Omega \times [0,T] \to \real^{m}$ is measurable, and $\Psi(t)$ is $\mathcal{F}_{T}$-measurable for any $t \in [0,T]$;
\item[(ii)]
$G: \Omega \times \Delta[0,T] \times \real^{m} \times \real^{m \times d} \times \real^{m \times d} \to \real^{m}$ is measurable, and the process $(G(t,s,y,z_{1},z_{2}))_{s \in [t,T]}$ is progressively measurable for any $(t,y,z_{1},z_{2}) \in [0,T] \times \real^{m} \times \real^{m \times d} \times \real^{m \times d}$;
\item[(iii)]
There exists a constant $L>0$ such that, for any $(t,s) \in \Delta [0,T]$ and $(y,z_{1},z_{2}), (y',z_{1}',z_{2}') \in \real^{m} \times \real^{m \times d} \times \real^{m \times d}$,
\begin{align*}
|G(t,s,y,z_{1},z_{2})-G(t,s,y',z_{1}',z_{2}')|
\leq
L
\big\{
	|y-y'|
	+
	|z_{1}-z_{1}'|
	+
	|z_{2}-z_{2}'|
\big\};
\end{align*}
\item[(iv)]
There exists a constant $M \geq 1$ such that
\begin{align*}
&
\e\Big[
	\int_{0}^{T}
		\big|\Psi(t)\big|^{2}
	\rd t
	+
	\int_{0}^{T}
		\int_{t}^{T}
			\big|G(t,s,0,0,0)\big|^{2}
		\rd s
	\rd t
\Big]^{1/2}
\leq
M;
\end{align*}

\item[(v)]
There exists an increasing and continuous function $\rho_{\Psi,G}:[0,\infty) \to [0,\infty)$ with $\rho_{\Psi,G}(0)=0$ such that, for any $0 \leq t,t' \leq s \leq T$ and $(Y,Z_{1},Z_{2}) \in L^{2}_{\mathcal{F}}(\Omega;\real^{m}) \times (L^{2}_{\mathcal{F}}(\Omega;\real^{m \times d}))^{2}$,
\begin{align*}
&\e
\Big[
	\big|\Psi(t)-\Psi(t')\big|^{2}
	+
	\big|G(t,s,Y,Z_{1},Z_{2})-G(t',s,Y,Z_{1},Z_{2})\big|^{2}
\Big]^{1/2}
\\&
\leq
\rho_{\Psi,G}(|t-t'|)
\Big\{
	M
	+
	\e\big[
		\big|Y\big|^{2}
	\big]^{1/2}
	+
	\e\big[
		\big|Z_{1}\big|^{2}
	\big]^{1/2}
	+
	\e\big[
		\big|Z_{2}\big|^{2}
	\big]^{1/2}
\Big\}.
\end{align*}

\end{itemize}
\end{itemize}

\begin{defi}\label{Def_M_sol}
A pair $(Y(\cdot),Z(\cdot,\cdot))$ is  called an adapted M-solution of BSVIE \eqref{BSVIE_II_1} if it satisfies the following conditions:
\begin{itemize}
\item[(i)]
$Y(\cdot) \in L^{2}_{\mathbb{F}}(0,T;\real^{m})$ and $Z(\cdot,\cdot) \in L^{2}(0,T;L^{2}_{\mathbb{F}}(0,T;\real^{m \times d}))$;
\item[(ii)]
It holds that
\begin{align*}
Y(t)
=
\e[Y(t)]
+
\int_{0}^{t}
	Z(t,s)
\rd W(s),~\text{a.e.}~t \in [0,T],~\text{a.s.};
\end{align*}
\item[(iii)]
$(Y(\cdot),Z(\cdot,\cdot))$ satisfies \eqref{BSVIE_II_1} for a.e.\ $t \in [0,T]$, a.s.
\end{itemize}
Here, the M-solution is named after the martingale representation theorem which determines the values of $Z(t,s)$ for $(t,s) \in \Delta^{\mathrm{c}}[0,T]$ by (ii) (see \cite{Yo08}).
\end{defi}

The following lemma shows the well-posedness of Type-\Rnum{2} BSVIE \eqref{BSVIE_II_1} (cf.\ Theorem 3.7 in \cite{Yo08}).

\begin{lemm}\label{Prop_BSVIE_0}
Under the conditions (i), (ii), (iii) and (iv) in $(\mathrm{H}_{\Psi,G})$, BSVIE \eqref{BSVIE_II_1} has a unique adapted M-solution.
Furthermore, it holds that
\begin{align}
\label{Prop_BSVIE_0_1}
\e\Big[
	\int_{0}^{T}
		\big|
			Y(t)
		\big|^{2}
	\rd t
	+
	\int_{0}^{T}
		\int_{0}^{T}
			\big|
				Z(t,s)
			\big|^{2}
		\rd s
	\rd t
\Big]
\leq
CM^{2}.
\end{align}
\end{lemm}

Fix $\pi=\{t_{0},t_{1},\ldots,t_{N}\} \in \Pi[0,T]$ and let $\mathbb{H}$ be a Euclidean space.
For each $\theta \in [0,T]$ and $\zeta(\cdot) \in L^{2}_{\mathbb{F}}(0,T;\mathbb{H})$, define 
\begin{align*}
\mathcal{I}^{\pi,\theta}[\zeta(\cdot)](t_{k})
:=
\frac{1}{\Delta t_{k}}
\int_{t_{k} \vee \theta}^{t_{k+1} \vee \theta}
	\zeta(s)
\rd s,~
k=0,\ldots, N-1.
\end{align*}
Note that $\mathcal{I}^{\pi,\theta}[\zeta(\cdot)](t_{k})$ is $\mathcal{F}_{t_{k+1}}$-measurable.
Denote $\mathcal{I}^{\pi,0}[\zeta(\cdot)](t_{k})$ by $\mathcal{I}^{\pi}[\zeta(\cdot)](t_{k})$, that is,
\begin{align*}
\mathcal{I}^{\pi}[\zeta(\cdot)](t_{k})
=
\frac{1}{\Delta t_{k}}
\int_{t_{k}}^{t_{k+1}}
	\zeta(s)
\rd s,~
k=0,\ldots, N-1.
\end{align*}

We introduce the following BSDE system with parameter $\theta \in [0,T]$:
\begin{align}\label{BSDE_sys_0}
\begin{cases}
\mathscr{Y}^{\pi,\theta}(t_{k},s)
=
\displaystyle
\mathscr{Y}^{\pi,\theta}(t_{k},t_{\ell+1})
+
\int_{s}^{t_{\ell+1}}
	G
	(
		t_{k},
		r,
		\mathscr{Y}^{\pi,\theta}(t_{\ell},r),
		\mathscr{Z}^{\pi,\theta}(t_{k},r),
		\mathcal{I}^{\pi,\theta}[\mathscr{Z}^{\pi,\theta}(t_{\ell},\cdot)](t_{k})
	)
	\1_{\{k<\ell\}}
\rd r
\\\displaystyle \quad\quad\quad\quad\quad\quad
-
\int_{s}^{t_{\ell+1}}
	\mathscr{Z}^{\pi,\theta}(t_{k},r)
\rd W(r),~
s \in [t_{\ell},t_{\ell+1}],~k,\ell=0,\ldots, N-1,
\\
\mathscr{Y}^{\pi,\theta}(t_{k},t_{N})
=
\Psi(t_{k}),~
k=0,\ldots,N-1.
\end{cases}
\end{align}
The above BSDE system is equivalent to the following:
\begin{align}\label{BSDE_sys_1}
\begin{split}
\mathscr{Y}^{\pi,\theta}(t_{k},s)
&=
\Psi(t_{k})
+
\int_{s}^{T}
	G
	(
		t_{k},
		r,
		\mathscr{Y}^{\pi,\theta}(\tau(r),r),
		\mathscr{Z}^{\pi,\theta}(t_{k},r),
		\mathcal{I}^{\pi,\theta}[\mathscr{Z}^{\pi,\theta}(\tau(r),\cdot)](t_{k})
	)
	\1_{[t_{k+1},T)}(r)
\rd r
\\&\quad
-
\int_{s}^{T}
	\mathscr{Z}^{\pi,\theta}(t_{k},r)
\rd W(r),~s \in [0,T],~k=0,\ldots,N-1.
\end{split}
\end{align}
We note that, under $(\mathrm{H}_{\Psi,G})$, $\Psi(t_{k})$ and $G(t_{k},s,y,z_{1},z_{2})$ are well-defined for each $k=0,\ldots,N-1$ and $(s,y,z_{1},z_{2}) \in [t_{k},T] \times \real^{m} \times \real^{m \times d} \times \real^{m \times d} $, since $\Psi$ and $G$ are continuous with respect to $t \in [0,T]$ in the sense of (v).
Furthermore, it holds that
\begin{align*}
&\sum_{k=0}^{N-1}
\Delta t_{k}
\e\Big[
	\big|\Psi(t_{k})\big|^{2}
	+
	\int_{t_{k+1}}^{T}
		\big|G(t_{k},s,0,0,0)\big|^{2}
	\rd s
\Big]
\leq
CM^{2}(1+\rho_{\Psi,G}(|\pi|)^{2}).
\end{align*}

First, we show fundamental properties of BSDE system \eqref{BSDE_sys_0}.
\begin{lemm}\label{Lem_BSDE_sys_0}
Under $(\mathrm{H}_{\Psi,G})$, for any $\pi=\{t_{0},t_{1},\ldots,t_{N}\} \in \Pi[0,T]$ and $\theta \in [0,T]$, there exists a unique solution
$
\{
	(\mathscr{Y}^{\pi,\theta}(t_{k},\cdot),
	\mathscr{Z}^{\pi,\theta}(t_{k},\cdot))
\}_{k=0}^{N-1}
\in
(
	L^{2}_{\mathbb{F}}(\Omega, C([0,T];\real^{m}))
	\times
	L^{2}_{\mathbb{F}}(0,T;\real^{m \times d})
)^{N}
$
of BSDE system \eqref{BSDE_sys_0}.
Furthermore, the following estimate holds:
\begin{align}\label{Lem_3_1_2}
\begin{split}
&\sum_{k=0}^{N-1}
\Delta t_{k}
\e\Big[
	\sup_{s \in [0,T]}
	\big|
		\mathscr{Y}^{\pi,\theta}(t_{k},s)
	\big|^{2}
	+
	\int_{0}^{T}
		\big|
			\mathscr{Z}^{\pi,\theta}(t_{k},s)
		\big|^{2}
	\rd s
\Big]
\\&\leq
C\sum_{k=0}^{N-1}
\Delta t_{k}
\e\Big[
	\big|\Psi(t_{k})\big|^{2}
	+
	\Big(
		\int_{t_{k+1}}^{T}
			\big|G(t_{k},s,0,0,0)\big|
		\rd s
	\Big)^{2}
\Big].
\end{split}
\end{align}
For $i=1,2$, let $(\Psi_{i},G_{i})$ satisfy $(\mathrm{H}_{\Psi,G})$ for $(\Psi,G)=(\Psi_{i},G_{i})$ with the same constants $L$ and $M$, and let
$
\{
	(\mathscr{Y}_{i}^{\pi,\theta}(t_{k},\cdot),
	\mathscr{Z}_{i}^{\pi,\theta}(t_{k},\cdot))
\}_{k=0}^{N-1}
\in
(
	L^{2}_{\mathbb{F}}(\Omega, C([0,T];\real^{m}))
	\times
	L^{2}_{\mathbb{F}}(0,T;\real^{m \times d})
)^{N}
$
be the unique solution of BSDE system \eqref{BSDE_sys_0} corresponding to $(\Psi_{i},G_{i})$.
Then it holds that
\begin{align}\label{Lem_3_1_1}
\begin{split}
&\sum_{k=0}^{N-1}
\Delta t_{k}
\e\Big[
	\sup_{s \in [0,T]}
	\big|
		\mathscr{Y}_{1}^{\pi,\theta}(t_{k},s)
		-
		\mathscr{Y}_{2}^{\pi,\theta}(t_{k},s)
	\big|^{2}
	+
	\int_{0}^{T}
		\big|
			\mathscr{Z}_{1}^{\pi,\theta}(t_{k},s)
			-
			\mathscr{Z}_{2}^{\pi,\theta}(t_{k},s)
		\big|^{2}
	\rd s
\Big]
\\&
\leq
C
\sum_{k=0}^{N-1}
\Delta t_{k}
\e\Big[
	\big|
		\Psi_{1}(t_{k})
		-
		\Psi_{2}(t_{k})
	\big|^{2}
	\\&\quad\quad\quad\quad\quad\quad\quad
	+
	\Big(
		\int_{t_{k+1}}^{T}
			\big|
				G_{1}
				(
					t_{k},
					s,
					\mathscr{Y}_{1}^{\pi,\theta}(\tau(s),s),
					\mathscr{Z}_{1}^{\pi,\theta}(t_{k},s),
					\mathcal{I}^{\pi,\theta}[\mathscr{Z}_{1}^{\pi,\theta}(\tau(s),\cdot)](t_{k})
				)
				\\&\quad\quad\quad\quad\quad\quad\quad\quad\quad\quad\quad\quad
				-
				G_{2}
				(
					t_{k},
					s,
					\mathscr{Y}_{1}^{\pi,\theta}(\tau(s),s),
					\mathscr{Z}_{1}^{\pi,\theta}(t_{k},s),
					\mathcal{I}^{\pi,\theta}[\mathscr{Z}_{1}^{\pi,\theta}(\tau(s),\cdot)](t_{k})
				)
			\big|
		\rd s
	\Big)^{2}
\Big].
\end{split}
\end{align}
\end{lemm}
\begin{proof}
Fix $\pi=\{t_{0},t_{1},\ldots,t_{N}\} \in \Pi[0,T]$ and $\theta \in [0,T]$.
We show that there exists a unique solution 
$
\{
	(\mathscr{Y}^{\pi,\theta}(t_{k},\cdot),
	\mathscr{Z}^{\pi,\theta}(t_{k},\cdot))
\}_{k=0}^{N-1}
\in
(
	L^{2}_{\mathbb{F}}(\Omega, C([0,T];\real^{m}))
	\times
	L^{2}_{\mathbb{F}}(0,T;\real^{m \times d})
)^{N}
$
of BSDE system \eqref{BSDE_sys_0} by a backward induction with respect to $k=0,\ldots,N-1$.

For $k=N-1$, the equation \eqref{BSDE_sys_1} for $(\mathscr{Y}^{\pi,\theta}(t_{N-1},\cdot),\mathscr{Z}^{\pi,\theta}(t_{N-1},\cdot))$ becomes
\begin{align*}
\mathscr{Y}^{\pi,\theta}(t_{N-1},s)
=
\Psi(t_{N-1})
-
\int_{s}^{T}
	\mathscr{Z}^{\pi,\theta}(t_{N-1},r)
\rd W(r),~s \in [0,T],
\end{align*}
which admits a unique adapted solution in $L^{2}_{\mathbb{F}}(\Omega, C([0,T];\real^{m})) \times L^{2}_{\mathbb{F}}(0,T;\real^{m \times d})$.

Let $k' \in \{0,\ldots, N-2\}$ be fixed and assume that, for each $k=k'+1,\ldots,N-1$, the equation \eqref{BSDE_sys_1} admits a unique solution
\begin{align*}
(
	\mathscr{Y}^{\pi,\theta}(t_{k},\cdot),
	\mathscr{Z}^{\pi,\theta}(t_{k},\cdot)
)
\in
L^{2}_{\mathbb{F}}(\Omega, C([0,T];\real^{m}))
\times
L^{2}_{\mathbb{F}}(0,T;\real^{m \times d}).
\end{align*}
Then we know that
\begin{align*}
&(\mathscr{Y}^{\pi,\theta}(\tau(r),r))_{r \in [t_{k'+1},T)} \in L^{2}_{\mathbb{F}}(t_{k'+1},T;\real^{m})
\shortintertext{and}
&(\mathcal{I}^{\pi,\theta}[\mathscr{Z}^{\pi,\theta}(\tau(r),\cdot)](t_{k'}))_{r \in [t_{k'+1},T)} \in L^{2}_{\mathcal{F}_{t_{k'+1}}}(t_{k'+1},T;\real^{m \times d}).
\end{align*}
Hence the BSDE
\begin{align*}
\mathscr{Y}^{\pi,\theta}(t_{k'},s)
&=
\Psi(t_{k'})
+
\int_{s}^{T}
	G
	(
		t_{k'},
		r,
		\mathscr{Y}^{\pi,\theta}(\tau(r),r),
		\mathscr{Z}^{\pi,\theta}(t_{k'},r),
		\mathcal{I}^{\pi,\theta}[\mathscr{Z}^{\pi,\theta}(\tau(r),\cdot)](t_{k'})
	)
	\1_{[t_{k'+1},T)}(r)
\rd r
\\&\quad
-
\int_{s}^{T}
	\mathscr{Z}^{\pi,\theta}(t_{k'},r)
\rd W(r),
~s \in [0,T],
\end{align*}
admits a unique adapted solution $(\mathscr{Y}^{\pi,\theta}(t_{k'},\cdot),\mathscr{Z}^{\pi,\theta}(t_{k'},\cdot)) \in L^{2}_{\mathbb{F}}(\Omega, C([0,T];\real^{m})) \times L^{2}_{\mathbb{F}}(0,T;\real^{m \times d})$.
Thus we complete the proof of the first assertion by the backward induction.

The estimate \eqref{Lem_3_1_2} follows from \eqref{Lem_3_1_1} by choosing $(\Psi_{1},G_{1})=(0,0)$ and $(\Psi_{2},G_{2})=(\Psi,G)$.
We prove \eqref{Lem_3_1_1}.
For each fixed $k=0,\ldots,N-1$ and $i=1,2$, $(\mathscr{Y}_{i}^{\pi,\theta}(t_{k},\cdot),\mathscr{Z}_{i}^{\pi,\theta}(t_{k},\cdot))$ solves the BSDE
\begin{align*}
\begin{split}
\mathscr{Y}_{i}^{\pi,\theta}(t_{k},s)
&=
\Psi_{i}(t_{k})
+
\int_{s}^{T}
	G_{i}
	(
		t_{k},
		r,
		\mathscr{Y}_{i}^{\pi,\theta}(\tau(r),r),
		\mathscr{Z}_{i}^{\pi,\theta}(t_{k},r),
		\mathcal{I}^{\pi,\theta}[\mathscr{Z}_{i}^{\pi,\theta}(\tau(r),\cdot)](t_{k})
	)
	\1_{[t_{k+1},T)}(r)
\rd r
\\&\quad
-
\int_{s}^{T}
	\mathscr{Z}_{i}^{\pi,\theta}(t_{k},r)
\rd W(r),~s \in [0,T].
\end{split}
\end{align*}
By the stability estimate for BSDEs (see  Lemma \ref{apriori_0}), we have
\begin{align*}
&
\e\Big[
	\sup_{s \in [0,T]}
	\big|
		\mathscr{Y}_{1}^{\pi,\theta}(t_{k},s)
		-
		\mathscr{Y}_{2}^{\pi,\theta}(t_{k},s)
	\big|^{2}
	+
	\int_{0}^{T}
		\big|
			\mathscr{Z}_{1}^{\pi, \theta}(t_{k},s)
			-
			\mathscr{Z}_{2}^{\pi, \theta}(t_{k},s)
		\big|^{2}
	\rd s
\Big]
\\*&\leq
C
\e\Big[
	\big|
		\Psi_{1}(t_{k})
		-
		\Psi_{2}(t_{k})
	\big|^{2}
	\\*&\quad\quad\quad
	+
	\Big(
		\int_{t_{k+1}}^{T}
			\big|
				G_{1}
				(
					t_{k},
					s,
					\mathscr{Y}_{1}^{\pi, \theta}(\tau(s),s),
					\mathscr{Z}_{1}^{\pi, \theta}(t_{k},s),
					\mathcal{I}^{\pi,\theta}[\mathscr{Z}_{1}^{\pi,\theta}(\tau(s),\cdot)](t_{k})
				)
				\\*&\quad\quad\quad\quad\quad\quad\quad\quad
				-
				G_{2}
				(
					t_{k},
					s,
					\mathscr{Y}_{2}^{\pi, \theta}(\tau(s),s),
					\mathscr{Z}_{1}^{\pi, \theta}(t_{k},s),
					\mathcal{I}^{\pi,\theta}[\mathscr{Z}_{2}^{\pi,\theta}(\tau(s),\cdot)](t_{k})
				)
			\big|
		\rd s
	\Big)^{2}
\Big]
\\&\leq
C
\e\Big[
	\big|
		\Psi_{1}(t_{k})
		-
		\Psi_{2}(t_{k})
	\big|^{2}
	\\&\quad\quad\quad
	+
	\Big(
		\int_{t_{k+1}}^{T}
			\big|
				G_{1}
				(
					t_{k},
					s,
					\mathscr{Y}_{1}^{\pi, \theta}(\tau(s),s),
					\mathscr{Z}_{1}^{\pi, \theta}(t_{k},s),
					\mathcal{I}^{\pi,\theta}[\mathscr{Z}_{1}^{\pi,\theta}(\tau(s),\cdot)](t_{k})
				)
				\\&\quad\quad\quad\quad\quad\quad\quad\quad
				-
				G_{2}
				(
					t_{k},
					s,
					\mathscr{Y}_{1}^{\pi, \theta}(\tau(s),s),
					\mathscr{Z}_{1}^{\pi, \theta}(t_{k},s),
					\mathcal{I}^{\pi,\theta}[\mathscr{Z}_{1}^{\pi,\theta}(\tau(s),\cdot)](t_{k})
				)
			\big|
		\rd s
	\Big)^{2}
\Big]
\\&\quad
+
C
\sum_{\ell=k+1}^{N-1}
\Delta t_{\ell}
\e\Big[
	\sup_{s \in [0,T]}
	\big|
		\mathscr{Y}_{1}^{\pi, \theta}(t_{\ell},s)
		-
		\mathscr{Y}_{2}^{\pi, \theta}(t_{\ell},s)
	\big|^{2}
\Big]
\\&\quad
+
C\e\Big[
	\Big(
		\sum_{\ell=k+1}^{N-1}
		\Delta t_{\ell}
		\big|
			\mathcal{I}^{\pi,\theta}[\mathscr{Z}_{1}^{\pi,\theta}(t_{\ell},\cdot)](t_{k})
			-
			\mathcal{I}^{\pi,\theta}[\mathscr{Z}_{2}^{\pi,\theta}(t_{\ell},\cdot)](t_{k})
		\big|
	\Big)^{2}
\Big].
\end{align*}
Moreover, we have for each $k=1,\ldots,N-1$,
\begin{align*}
\mathscr{Y}_{i}^{\pi, \theta}(t_{k},t_{k})
=
\e
\Big[
	\mathscr{Y}_{i}^{\pi, \theta}(t_{k},t_{k})
\Big]
+
\int_{0}^{t_{k}}
	\mathscr{Z}_{i}^{\pi, \theta}(t_{k},s)
\rd W(s),
~\text{a.s.},
~i=1,2,
\end{align*}
and thus
\begin{align*}%\label{system_1}
&
\sum_{\ell=0}^{k-1}
\Delta t_{\ell}
\e\Big[
	\big|
		\mathcal{I}^{\pi,\theta}[\mathscr{Z}_{1}^{\pi,\theta}(t_{k},\cdot)](t_{\ell})
		-
		\mathcal{I}^{\pi,\theta}[\mathscr{Z}_{2}^{\pi,\theta}(t_{k},\cdot)](t_{\ell})
	\big|^{2}
\Big]
\\&=
\sum_{\ell=0}^{k-1}
\Delta t_{\ell}
\e\Big[
	\Big|
		\frac{1}{\Delta t_{\ell}}
		\int_{t_{\ell}\vee \theta}^{t_{\ell+1} \vee \theta}
			\big\{
				\mathscr{Z}_{1}^{\pi, \theta}(t_{k},s)
				-
				\mathscr{Z}_{2}^{\pi, \theta}(t_{k},s)
			\big\}
		\rd s
	\Big|^{2}
\Big]
\\&
\leq
\sum_{\ell=0}^{k-1}
\e\Big[
	\int_{t_{\ell}}^{t_{\ell+1}}
		|
			\mathscr{Z}_{1}^{\pi, \theta}(t_{k},s)
			-
			\mathscr{Z}_{2}^{\pi, \theta}(t_{k},s)
		|^{2}
	\rd s
\Big]
=
\e\Big[
	\int_{0}^{t_{k}}
		|
			\mathscr{Z}_{1}^{\pi, \theta}(t_{k},s)
			-
			\mathscr{Z}_{2}^{\pi, \theta}(t_{k},s)
		|^{2}
	\rd s
\Big]
\\&=
\e\Big[
	\Big|
		\int_{0}^{t_{k}}
			\big\{
				\mathscr{Z}_{1}^{\pi, \theta}(t_{k},s)
				-
				\mathscr{Z}_{2}^{\pi, \theta}(t_{k},s)
			\big\}
		\rd W(s)
	\Big|^{2}
\Big]
\\&=
\e\Big[
	\Big|
		\mathscr{Y}_{1}^{\pi, \theta}(t_{k},t_{k})
		-
		\mathscr{Y}_{2}^{\pi, \theta}(t_{k},t_{k})
		-
		\e
		\Big[
			\mathscr{Y}_{1}^{\pi, \theta}(t_{k},t_{k})
			-
			\mathscr{Y}_{2}^{\pi, \theta}(t_{k},t_{k})
		\Big]
	\Big|^{2}
\Big]
\\&\leq
\e\Big[
	\big|
		\mathscr{Y}_{1}^{\pi, \theta}(t_{k},t_{k})
		-
		\mathscr{Y}_{2}^{\pi, \theta}(t_{k},t_{k})
	\big|^{2}
\Big].
\end{align*}
Hence, by using the discrete Gronwall-like inequality (see Lemma \ref{Lem_Gron_disc}) with
\begin{align*}
a_{k}
&=
\e\Big[
	\sup_{s \in [0,T]}
	\big|
		\mathscr{Y}_{1}^{\pi, \theta}(t_{k},s)
		-
		\mathscr{Y}_{2}^{\pi, \theta}(t_{k},s)
	\big|^{2}
	+
	\int_{0}^{T}
		\big|
			\mathscr{Z}_{1}^{\pi, \theta}(t_{k},s)
			-
			\mathscr{Z}_{2}^{\pi, \theta}(t_{k},s)
		\big|^{2}
	\rd s
\Big],
\\
b_{k}
&=
\e\Big[
	\big|
		\Psi_{1}(t_{k})
		-
		\Psi_{2}(t_{k})
	\big|^{2}
	\\&\quad\quad\quad
	+
	\Big(
		\int_{t_{k+1}}^{T}
			\big|
				G_{1}
				(
					t_{k},
					s,
					\mathscr{Y}_{1}^{\pi, \theta}(\tau(s),s),
					\mathscr{Z}_{1}^{\pi, \theta}(t_{k},s),
					\mathcal{I}^{\pi,\theta}[\mathscr{Z}_{1}^{\pi,\theta}(\tau(s),\cdot)](t_{k})
				)
				\\*&\quad\quad\quad\quad\quad\quad\quad\quad
				-
				G_{2}
				(
					t_{k},
					s,
					\mathscr{Y}_{1}^{\pi, \theta}(\tau(s),s),
					\mathscr{Z}_{1}^{\pi, \theta}(t_{k},s),
					\mathcal{I}^{\pi,\theta}[\mathscr{Z}_{1}^{\pi,\theta}(\tau(s),\cdot)](t_{k})
				)
			\big|
		\rd s
	\Big)^{2}
\Big],~
\\
c_{k}
&=
0,
~
\zeta_{k,\ell}(\omega)
=
\big|
	\mathcal{I}^{\pi,\theta}[\mathscr{Z}_{1}^{\pi,\theta}(t_{k},\cdot)](t_{\ell})
	-
	\mathcal{I}^{\pi,\theta}[\mathscr{Z}_{2}^{\pi,\theta}(t_{k},\cdot)](t_{\ell})
\big|
~\text{and}~
(S,\Sigma,\mu)
=
(\Omega,\mathcal{F},\p),
\end{align*}
we get the estimate \eqref{Lem_3_1_1}.
\end{proof}

Next, we provide $L^p$-a priori estimates for solutions of Type-\Rnum{2} BSVIE~\eqref{BSVIE_II_1} and BSDE system~\eqref{BSDE_sys_0} for $p\geq2$.
The $L^p$-estimates for $p>2$ are needed for the analysis in Section~\ref{sec_L2_reg}.

\begin{theo}\phantomsection\label{Lem_Lp}
\begin{itemize}
\item[(i)]
Suppose that $(\Psi,G)$ satisfies the conditions (i), (ii) and (iii) in $(\mathrm{H}_{\Psi,G})$, and
\begin{align*}
\e\Big[
	\Big(
		\int_{0}^{T}
			\big|
				\Psi(t)
			\big|^{2}
		\rd t
	\Big)^{p/2}
	+
	\Big(
		\int_{0}^{T}
			\int_{t}^{T}
				\big|
					G(t,s,0,0,0)
				\big|^{2}
			\rd s
		\rd t
	\Big)^{p/2}
\Big]
<\infty,
\end{align*}
for some $p \geq 2$.
Let $(Y(\cdot), Z(\cdot,\cdot)) \in L^{2}_{\mathbb{F}}(0,T;\real^{m}) \times L^{2}(0,T;L^{2}_{\mathbb{F}}(0,T;\real^{m \times d}))$ be the adapted M-solution of BSVIE \eqref{BSVIE_II_1}.
For a.e.\ $t \in [0,T]$, define
\begin{align*}
Y(t,s)
:=
\e_{s}
\Big[
	\Psi(t)
	+
	\int_{s}^{T}
		G(t,r,Y(r),Z(t,r),Z(r,t))
		\1_{[t,T]}(r)
	\rd r
\Big],~
s \in [0,T].
\end{align*}
Then it holds that
\begin{align*}%\label{Lp_BSVIE}
&\e\Big[
	\sup_{s \in [0,T]}
	\Big(
		\int_{0}^{T}
			\big|
				Y(t,s)
			\big|^{2}
		\rd t
	\Big)^{p/2}
	+
	\Big(
		\int_{0}^{T}
			\big|
				Y(t)
			\big|^{2}
		\rd t
	\Big)^{p/2}
	+
	\Big(
		\int_{0}^{T}
			\int_{0}^{T}
				\big|
					Z(t,s)
				\big|^{2}
			\rd s
		\rd t
	\Big)^{p/2}
\Big]
\\
&\leq
C_{p}
\e\Big[
	\Big(
		\int_{0}^{T}
			\big|
				\Psi(t)
			\big|^{2}
		\rd t
	\Big)^{p/2}
	+
	\Big(
		\int_{0}^{T}
			\int_{t}^{T}
				\big|
					G(t,s,0,0,0)
				\big|^{2}
			\rd s
		\rd t
	\Big)^{p/2}
\Big].
\end{align*}

\item[(ii)]
Let $\pi=\{t_{0},t_{1},\ldots,t_{N}\} \in \Pi[0,T]$ and $\theta \in [0,T]$ be fixed.
Suppose that $(\Psi,G)$ satisfies $(\mathrm{H}_{\Psi,G})$ and
\begin{align*}
\e\Big[
	\Big(
		\sum_{k=0}^{N-1}
		\Delta t_{k}
		\big|
			\Psi(t_{k})
		\big|^{2}
	\Big)^{p/2}
	+
	\Big(
		\sum_{k=0}^{N-1}
		\Delta t_{k}
		\int_{t_{k+1}}^{T}
			\big|
				G(t_{k},s,0,0,0)
			\big|^{2}
		\rd s
	\Big)^{p/2}
\Big]
<\infty,
\end{align*}
for some $p \geq 2$.
Let $
\{
	(\mathscr{Y}^{\pi,\theta}(t_{k},\cdot),
	\mathscr{Z}^{\pi,\theta}(t_{k},\cdot))
\}_{k=0}^{N-1}
\in
(
	L^{2}_{\mathbb{F}}(\Omega, C([0,T];\real^{m}))
	\times
	L^{2}_{\mathbb{F}}(0,T;\real^{m \times d})
)^{N}
$
be the solution of BSDE system \eqref{BSDE_sys_0}.
Then it holds that
\begin{align*}%\label{Lp_BSDE_sys}
&\e\Big[
	\sup_{s \in [0,T]}
	\Big(
		\sum_{k=0}^{N-1}
		\Delta t_{k}
		\big|
			\mathscr{Y}^{\pi,\theta}(t_{k},s)
		\big|^{2}
	\Big)^{p/2}
	+
	\Big(
		\sum_{k=0}^{N-1}
		\int_{t_{k}}^{t_{k+1}}
			\big|
				\mathscr{Y}^{\pi,\theta}(t_{k},t)
			\big|^{2}
		\rd t
	\Big)^{p/2}
	\\&\quad\quad\quad
	+
	\Big(
		\sum_{k=0}^{N-1}
		\Delta t_{k}
		\int_{0}^{T}
			\big|
				\mathscr{Z}^{\pi,\theta}(t_{k},s)
			\big|^{2}
		\rd s
	\Big)^{p/2}
\Big]
\\
&\leq
C_{p}
\e\Big[
	\Big(
		\sum_{k=0}^{N-1}
		\Delta t_{k}
		\big|
			\Psi(t_{k})
		\big|^{2}
	\Big)^{p/2}
	+
	\Big(
		\sum_{k=0}^{N-1}
		\Delta t_{k}
		\int_{t_{k+1}}^{T}
			\big|
				G(t_{k},s,0,0,0)
			\big|^{2}
		\rd s
	\Big)^{p/2}
\Big].
\end{align*}
\end{itemize}
\end{theo}

\begin{proof}
See \hyperref[sec_Appendix]{Appendix}.
\end{proof}

\begin{rem}
\begin{itemize}
\item[(\rnum{1})]
We remark that $L^p$-estimates for the solution of Type-\Rnum{1} BSVIE~\eqref{Type-1} for $p>1$ have been investigated in many papers (see for example \cite{Yo08,WaT12,Wa20,Ha21,Po20}).
Also, for Type-\Rnum{2} BSVIE~\eqref{BSVIE_II_1}, Yong~\cite{Yo08} showed the well-posedness and a priori estimate for the adapted M-solution in the $L^2$-space, and Wang~\cite{WaT12} showed the well-posedness in the $L^p$-space with $p\in(1,2)$.
However, $L^p$-estimates for the adapted M-solution of Type-\Rnum{2} BSVIE~\eqref{BSVIE_II_1} for $p>2$ have not been studied in the literature.
Indeed, Popier \cite{Po20} mentioned that the $L^{p}$-estimates for Type-\Rnum{2} BSVIEs has been an open problem.
Theorem~\ref{Lem_Lp} gives an answer to the open problem mentioned in the literature, and thus we guess that this theorem itself is an important theoretical result.
The main idea of the proof is to consider the dynamics of the ``integrated process'' $s \mapsto \int_{0}^{T}|Y(t,s)|^{2} \rd t$, which turns out to be an It\^o process.
For more detailed discussions, see \hyperref[sec_Appendix]{Appendix}.

\item[(\rnum{2})]
Theorem~\ref{Lem_Lp} does not give any estimates for the term $\mathbb{E}[\int^T_0|Y(t)|^p\rd t]$ with $p>2$, which is finite in the case of Type-\Rnum{1} BSVIE~\eqref{Type-1} under appropriate assumptions for $\Psi$ and $G$ (see \cite{Ha21,Po20}).
We note that the martingale integrand $Z(\cdot,\cdot)$, which stems from the martingale representation theorem, is just locally square integrable with respect to the second time parameter, and the integral $\int^T_0|\int^T_tZ(s,t)\rd s|^p\rd t$ is not finite in general.
For this reason, for Type-\Rnum{2} BSVIE~\eqref{BSVIE_II_1} whose driver $G$ depends on $Z(s,t)$, the term $\mathbb{E}[\int^T_0|Y(t)|^p\rd t]$ is difficult to estimate, and we can guess that it is not finite in general.
\end{itemize}
\end{rem}

Denote the solution
$(\mathscr{Y}^{\pi,0}, \mathscr{Z}^{\pi,0})$ of BSDE system \eqref{BSDE_sys_0} with $\theta =0 $ by $(\mathscr{Y}^{\pi}, \mathscr{Z}^{\pi})$.
The following theorem is the main result of this section, which we call a {\it BSDE approximation}.

\begin{theo}\label{Theo_BSDE_0}
Suppose $(\mathrm{H}_{\Psi,G})$ holds.
Then it holds that
\begin{align*}%\label{Thm_BSDE_1}
&
\lim_{|\pi| \to 0}
\e
\Big[
	\int_{0}^{T}
		\big|
			Y(t)
			-
			\mathscr{Y}^{\pi}(\tau(t),t)
		\big|^{2}
	\rd t
	+
	\int_{0}^{T}
		\int_{0}^{T}
			\big|
				Z(t,s)
				-
				\mathscr{Z}^{\pi}(\tau(t),s)
			\big|^{2}
		\rd s
	\rd t
\Big]
= 0.
\end{align*}
\end{theo}
\begin{proof}
For a.e.\ $t\in[0,T]$, define $Y(t,\cdot)=(Y(t,s))_{s\in[0,T]}$ by
\begin{equation*}
	Y(t,s):=\e_s\Bigl[\Psi(t)+\int^T_sG(t,r,Y(r),Z(t,r),Z(r,t))\1_{[t,T]}(r)\rd r\Bigr],\ s\in[0,T].
\end{equation*}
Then $(Y(t,\cdot),Z(t,\cdot))\in L^2_\mathbb{F}(\Omega;C([0,T];\real^m))\times L^2_\mathbb{F}(0,T;\real^{m\times d})$ solves the following BSDE parametrized by $t$:
\begin{equation*}
	Y(t,s)=\Psi(t)+\int^T_sG(t,r,Y(r),Z(t,r),Z(r,t))\1_{[t,T]}(r)\rd r-\int^T_sZ(t,r)\rd W(r),\ s\in[0,T].
\end{equation*}
Indeed, for each $s\in[t,T]$, by taking the conditional expectation $\e_s[\cdot]$ on both sides of \eqref{BSVIE_II_1}, we have
\begin{align*}
	Y(t)&=\e_s\Bigl[\Psi(t)+\int^T_tG(t,r,Y(r),Z(t,r),Z(r,t))\rd r\Bigr]-\int^s_tZ(t,r)\rd W(r)\\
	&=Y(t,s)+\int^s_tG(t,r,Y(r),Z(t,r),Z(r,t))\rd r-\int^s_tZ(t,r)\rd W(r),
\end{align*}
and hence
\begin{equation*}
	Y(t,s)=\Psi(t)+\int^T_sG(t,r,Y(r),Z(t,r),Z(r,t))\rd r-\int^T_sZ(t,r)\rd W(r).
\end{equation*}
Furthermore, for each $s\in[0,t]$, by the definition of the adapted M-solution, it holds that
\begin{align*}
	Y(t,s)&=\e_s\Bigl[\Psi(t)+\int^T_tG(t,r,Y(r),Z(t,r),Z(r,t))\rd r\Bigr]=\e_s\bigl[Y(t)\bigr]=Y(t)-\int^t_sZ(t,r)\rd W(r)\\
	&=\Psi(t)+\int^T_tG(t,r,Y(r),Z(t,r),Z(r,t))\rd r-\int^T_sZ(t,r)\rd W(r).
\end{align*}
We also note that $Y(t,t)=Y(t)$, a.e.\ $t\in[0,T]$, a.s.
On the other hand, for each fixed $t \in [0,T)$, $(\mathscr{Y}^{\pi}(\tau(t),\cdot), \mathscr{Z}^{\pi}(\tau(t),\cdot))$ solves the BSDE
\begin{align*}
\mathscr{Y}^{\pi}(\tau(t),s)
&=
\Psi(\tau(t))
+
\int_{s}^{T}
	G
	(
		\tau(t),
		r,
		\mathscr{Y}^{\pi}(\tau(r),r),
		\mathscr{Z}^{\pi}(\tau(t),r),
		\mathcal{I}^{\pi}[\mathscr{Z}^{\pi}(\tau(r),\cdot)](\tau(t))
	)
	\1_{[\tau^{*}(t),T)}(r)
\rd r
\\*&\quad
-
\int_{s}^{T}
	\mathscr{Z}^{\pi}(\tau(t),r)
\rd W(r),~s \in [0,T].
\end{align*}
Thus by  the stability estimate for BSDEs (see Lemma \ref{apriori_0}), we have, for a.e.\ $t \in [0,T]$,
\begin{align*}
&
\e
\Big[
	\big|
		Y(t)
		-
		\mathscr{Y}^{\pi}(\tau(t),t)
	\big|^{2}
	+
	\int_{0}^{T}
		\big|
			Z(t,s)
			-
			\mathscr{Z}^{\pi}(\tau(t),s)
		\big|^{2}
	\rd s
\Big]
\\*&\leq
\e\Big[
	\sup_{s \in [0,T]}
	\big|
		Y(t,s)
		-
		\mathscr{Y}^{\pi}(\tau(t),s)
	\big|^{2}
	+
	\int_{0}^{T}
		\big|
			Z(t,s)
			-
			\mathscr{Z}^{\pi}(\tau(t),s)
		\big|^{2}
	\rd s
\Big]
\\&\leq
C
\e
\Big[
	\big|
		\Psi(t)
		-
		\Psi(\tau(t))
	\big|^{2}
	+
	\Big(
		\int_{0}^{T}
			\big|
				G
				(
					t,
					s,
					Y(s),
					Z(t,s),
					Z(s,t)
				)
				\1_{[t,T)}(s)
				\\&\quad\quad\quad\quad\quad\quad\quad\quad\quad\quad\quad\quad\quad\quad\quad
				-
				G
				(
					\tau(t),
					s,
					\mathscr{Y}^{\pi}(\tau(s),s),
					Z(t,s),
					\mathcal{I}^{\pi}[\mathscr{Z}^{\pi}(\tau(s),\cdot)](\tau(t))
				)
				\1_{[\tau^{*}(t),T)}(s)
			\big|
		\rd s
	\Big)^{2}
\Big]
\\&\leq
C
\Big\{
	\rho_{\Psi,G}(|\pi|)^{2}
	\Big(
		M^{2}
		+
		\e\Big[
			\int_{\tau^{*}(t)}^{T}
				\big|
					Y(s)
				\big|^{2}
				+
				\big|
					Z(t,s)
				\big|^{2}
				+
				\big|
					Z(s,t)
				\big|^{2}
			\rd s
		\Big]
	\Big)
	\\&\quad\quad\quad
	+
	|\pi|
	\e\Big[
		\int_{t}^{\tau^{*}(t)}
			\Big\{
				\big|
					G
					(
						t,
						s,
						0,
						0,
						0
					)
				\big|^{2}
				+
				\big|
					Y(s)
				\big|^{2}
				+
				\big|
					Z(t,s)
				\big|^{2}
				+
				\big|
					Z(s,t)
				\big|^{2}
			\Big\}
		\rd s
	\Big]
	\\&\quad\quad\quad
	+
	\int_{\tau^{*}(t)}^{T}
		\e
		\Big[
			\big|
				Y(s)
				-
				\mathscr{Y}^{\pi}(\tau(s),s)
			\big|^{2}
		\Big]
	\rd s
	+
	\e
	\Big[
		\Big(
			\int_{\tau^{*}(t)}^{T}
				\big|
					Z(s,t)
					-
					\mathcal{I}^{\pi}[\mathscr{Z}^{\pi}(\tau(s),\cdot)](\tau(t))
				\big|
			\rd s
		\Big)^{2}
	\Big]
\Big\},
\end{align*}
and thus,
\begin{align}\label{eq_3.6_0}
\begin{split}
&
\e
\Big[
	\big|
		Y(t)
		-
		\mathscr{Y}^{\pi}(\tau(t),t)
	\big|^{2}
	+
	\int_{0}^{T}
		\big|
			Z(t,s)
			-
			\mathscr{Z}^{\pi}(\tau(t),s)
		\big|^{2}
	\rd s
\Big]
\\&\leq
C
\Big\{
	\big(
		|\pi|
		+
		\rho_{\Psi,G}(|\pi|)^{2}
	\big)
	\Big(
		M^{2}
		+
		\e\Big[
			\int_{t}^{T}
				\big\{
					\big|G(t,s,0,0,0)\big|^{2}
					+
					\big|Y(s)\big|^{2}
					+
					\big|Z(t,s)\big|^{2}
					+
					\big|Z(s,t)\big|^{2}
				\big\}
			\rd s
		\Big]
	\Big)
	\\&\quad\quad\quad
	+
	\int_{t}^{T}
		\e
		\Big[
			\big|
				Y(s)
				-
				\mathscr{Y}^{\pi}(\tau(s),s)
			\big|^{2}
		\Big]
	\rd s
	+
	\e
	\Big[
		\Big(
			\int_{t}^{T}
				\big|
					Z(s,t)
					-
					\mathcal{I}^{\pi}[\mathscr{Z}^{\pi}(\tau(s),\cdot)](\tau(t))
				\big|
			\1_{[\tau^{*}(t),T)}(s)
			\rd s
		\Big)^{2}
	\Big]
\Big\}.
\end{split}
\end{align}
Observe that, for a.e.\ $t \in [0,T)$,
\begin{align*}
&\e
\Big[
	\int_{0}^{t}
		\big|
			Z(t,s)
			-
			\mathcal{I}^{\pi}[\mathscr{Z}^{\pi}(\tau(t),\cdot)](\tau(s))
		\big|^{2}
		\1_{[\tau^{*}(s),T)}(t)
	\rd s
\Big]
=
\e
\Big[
	\int_{0}^{\tau(t)}
		\big|
			Z(t,s)
			-
			\mathcal{I}^{\pi}[\mathscr{Z}^{\pi}(\tau(t),\cdot)](\tau(s))
		\big|^{2}
	\rd s
\Big]
\\&\leq
2
\Big\{
\e
\Big[
	\int_{0}^{\tau(t)}
		\big|
			Z(t,s)
			-
			\mathcal{I}^{\pi}[Z(t,\cdot)](\tau(s))
		\big|^{2}
	\rd s
\Big]
+
\e
\Big[
	\int_{0}^{\tau(t)}
		\big|
			\mathcal{I}^{\pi}[Z(t,\cdot)-\mathscr{Z}^{\pi}(\tau(t),\cdot)](\tau(s))
		\big|^{2}
	\rd s
\Big]
\Big\}.
\end{align*}
The second term in the right-hand side can be estimated as follows:
\begin{align*}
\e
\Big[
	\int_{0}^{\tau(t)}
		\big|
			\mathcal{I}^{\pi}[Z(t,\cdot)-\mathscr{Z}^{\pi}(\tau(t),\cdot)](\tau(s))
		\big|^{2}
	\rd s
\Big]
&\leq
\e
\Big[
	\int_{0}^{t}
		\big|
			Z(t,s)
			-
			\mathscr{Z}^{\pi}(\tau(t),s)
		\big|^{2}
	\rd s
\Big]
\\&=
\e
\Big[
	\Big|
		\int_{0}^{t}
			\big\{
				Z(t,s)
				-
				\mathscr{Z}^{\pi}(\tau(t),s)
			\big\}
		\rd W(s)
	\Big|^{2}
\Big]
\\&\leq
\e
\Big[
	\big|
		Y(t)
		-
		\mathscr{Y}^{\pi}(\tau(t),t)
	\big|^{2}
\Big].
\end{align*}
Thus we obtain
\begin{align*}
&\e
\Big[
	\int_{0}^{t}
		\big|
			Z(t,s)
			-
			\mathcal{I}^{\pi}[\mathscr{Z}^{\pi}(\tau(t),\cdot)](\tau(s))
		\big|^{2}
		\1_{[\tau^{*}(s),T)}(t)
	\rd s
\Big]
\\&\leq
2
\Big\{
\e
\Big[
	\int_{0}^{\tau(t)}
		\big|
			Z(t,s)
			-
			\mathcal{I}^{\pi}[Z(t,\cdot)](\tau(s))
		\big|^{2}
	\rd s
\Big]
+
\e
\Big[
	\big|
		Y(t)
		-
		\mathscr{Y}^{\pi}(\tau(t),t)
	\big|^{2}
\Big]
\Big\}.
\end{align*}
From this inequality and \eqref{eq_3.6_0}, by using the Gronwall-like inequality (see Lemma \ref{Lem_Gron_cont}) with
\begin{align*}
a(t)
&=
\e
\Big[
	\big|
		Y(t)
		-
		\mathscr{Y}^{\pi}(\tau(t),t)
	\big|^{2}
	+
	\int_{0}^{T}
		\big|
			Z(t,s)
			-
			\mathscr{Z}^{\pi}(\tau(t),s)
		\big|^{2}
	\rd s
\Big],
\\
b(t)
&=
\big(
	|\pi|
	+
	\rho_{\Psi,G}(|\pi|)^{2}
\big)
\Big(
	M^{2}
	+
	\e\Big[
		\int_{t}^{T}
			\big\{
				\big|G(t,s,0,0,0)\big|^{2}
				+
				\big|Y(s)\big|^{2}
				+
				\big|Z(t,s)\big|^{2}
				+
				\big|Z(s,t)\big|^{2}
			\big\}
		\rd s
	\Big]
\Big),
\\
c(t)
&=
\e
\Big[
	\int_{0}^{\tau(t)}
		\big|
			Z(t,s)
			-
			\mathcal{I}^{\pi}[Z(t,\cdot)](\tau(s))
		\big|^{2}
	\rd s
\Big],
\\
\zeta(t,s,\omega)
&=
\big|
	Z(t,s)
	-
	\mathcal{I}^{\pi}[\mathscr{Z}^{\pi}(\tau(t),\cdot)](\tau(s))
\big|
\1_{[\tau^{*}(s),T)}(t)
~\text{and}~
(S,\Sigma,\mu)
=
(\Omega,\mathcal{F},\p),
\end{align*}
we see that
\begin{align*}
&\int_{0}^{T}
	\e
	\Big[
		\big|
			Y(t)
			-
			\mathscr{Y}^{\pi}(\tau(t),t)
		\big|^{2}
		+
		\int_{0}^{T}
			\big|
				Z(t,s)
				-
				\mathscr{Z}^{\pi}(\tau(t),s)
			\big|^{2}
		\rd s
	\Big]
\rd t
\\&\leq
C
\Big\{
	\big(
		|\pi|
		+
		\rho_{\Psi,G}(|\pi|)^{2}
	\big)
	\Big(
		M^{2}
		+
		\int_{0}^{T}
			\e\Big[
				\int_{t}^{T}
					\big\{
						\big|G(t,s,0,0,0)\big|^{2}
						+
						\big|Y(s)\big|^{2}
						+
						\big|Z(t,s)\big|^{2}
						+
						\big|Z(s,t)\big|^{2}
					\big\}
				\rd s
			\Big]
		\rd t
	\Big)
	\\&\quad\quad\quad+
	\int_{0}^{T}
		\e
		\Big[
			\int_{0}^{\tau(t)}
				\big|
					Z(t,s)
					-
					\mathcal{I}^{\pi}[Z(t,\cdot)](\tau(s))
				\big|^{2}
			\rd s
		\Big]
	\rd t
\Big\}
\\&\leq
C
\Big\{
	M^{2}
	\big(
		|\pi|
		+
		\rho_{\Psi,G}(|\pi|)^{2}
	\big)
	+
	\int_{0}^{T}
		\e\Big[
			\int_{0}^{T}
				\big|
					Z(t,s)
					-
					\mathcal{I}^{\pi}[Z(t,\cdot)](\tau(s))
				\big|^{2}
			\rd t
		\Big]
	\rd s
\Big\},
\end{align*}
where in the last inequality, we used the assumption (iv) in $(\mathrm{H}_{\Psi,G})$ and the estimate \eqref{Prop_BSVIE_0_1}.
%Since $s \mapsto Z(\cdot,\cdot,s)$ is a Bochner integrable function on the Hilbert space $L^{2}(\Omega \times [0,T];\real^{m \times d})$, we see that the second term in the right-hand side above converges to zero as $|\pi| \downarrow 0$.
Since $s \mapsto Z(\cdot,\cdot,s)$ is a square-integrable function with values in the Hilbert space $L^{2}(\Omega \times [0,T];\real^{m \times d})$, we see that the second term in the right-hand side above converges to zero as $|\pi| \downarrow 0$.
This completes the proof.
\end{proof}

\begin{rem}
In the case of Type-\Rnum{1} BSVIEs, a similar convergence result was shown by Wang \cite{WaY18} and Wang and Yong \cite{WaTYo19}.
Theorem \ref{Theo_BSDE_0} extends their result to Type-\Rnum{2} BSVIEs.
We emphasize that our result holds true for general (stochastic) coefficients $\Psi$ and $G$.
\end{rem}

For further analysis, we need a quantitative estimate for the $L^{2}$-error between the solution $(\mathscr{Y}^{\pi}, \mathscr{Z}^{\pi})$ of BSDE system \eqref{BSDE_sys_0}  with $\theta =0$ and the adapted M-solution $(Y(\cdot),Z(\cdot,\cdot))$ of BSVIE \eqref{BSVIE_II_1}.
For this purpose, we define
\begin{align*}
\overline{\mathscr{Z}}^{\pi}(t_{k},t_{\ell})
:=
\frac{1}{\Delta t_{\ell}}
\e_{t_{\ell}}
\Big[
	\int_{t_{\ell}}^{t_{\ell+1}}
		\mathscr{Z}^{\pi}(t_{k},s)
	\rd s
\Big]
=
\e_{t_{\ell}}
\Big[
	\mathcal{I}^{\pi}[\mathscr{Z}^{\pi}(t_{k},\cdot)](t_{\ell})
\Big],~
k,\ell=0,\ldots,N-1.
\end{align*}

\begin{prop}\label{Prop_BSDE_0}
Suppose $(\mathrm{H}_{\Psi,G})$ holds.
Then for any $\pi=\{t_{0},t_{1},\ldots,t_{N}\} \in \Pi[0,T]$, it holds that
\begin{align*}%\label{Thm_BSDE_2}
&
\e
\Big[
	\int_{0}^{T}
		\big|
			Y(t)
			-
			\mathscr{Y}^{\pi}(\tau(t),t)
		\big|^{2}
	\rd t
	+
	\int_{0}^{T}
		\int_{0}^{T}
			\big|
				Z(t,s)
				-
				\mathscr{Z}^{\pi}(\tau(t),s)
			\big|^{2}
		\rd s
	\rd t
\Big]
\\&\leq
C
\Big\{
	M^{2}
	\big(
		|\pi|
		+
		\rho_{\Psi,G}(|\pi|)^{2}
	\big)
	+
	\sum_{k=1}^{N-1}
		\Delta t_{k}
		\sum_{\ell=0}^{k-1}
			\e\Big[
				\int_{t_{\ell}}^{t_{\ell+1}}
					\big|
						\mathscr{Z}^{\pi}(t_{k},s)
						-
						\overline{\mathscr{Z}}^{\pi}(t_{k},t_{\ell})
					\big|^{2}
				\rd s
			\Big]
\Big\}.
\end{align*}
\end{prop}
\begin{proof}
As before, we see that the estimate \eqref{eq_3.6_0} holds.
Noting that 
\begin{align*}
&\e
\Big[
	\int_{0}^{t}
		\big|
			Z(t,s)
			-
			\mathcal{I}^{\pi}[\mathscr{Z}^{\pi}(\tau(t),\cdot)](\tau(s))
		\big|^{2}
		\1_{[\tau^{*}(s),T)}(t)
	\rd s
\Big]
=
\e
\Big[
	\int_{0}^{\tau(t)}
		\big|
			Z(t,s)
			-
			\mathcal{I}^{\pi}[\mathscr{Z}^{\pi}(\tau(t),\cdot)](\tau(s))
		\big|^{2}
	\rd s
\Big]
\\&\leq
3
\Big\{
\e
\Big[
	\int_{0}^{t}
		\big|
			Z(t,s)
			-
			\mathscr{Z}^{\pi}(\tau(t),s)
		\big|^{2}
	\rd s
\Big]
+
\e
\Big[
	\int_{0}^{\tau(t)}
		\big|
			\mathscr{Z}^{\pi}(\tau(t),s)
			-
			\overline{\mathscr{Z}}^{\pi}(\tau(t),\tau(s))
		\big|^{2}
	\rd s
\Big]
\\&\quad\quad\quad
+
\e
\Big[
	\int_{0}^{\tau(t)}
		\big|
			\mathcal{I}^{\pi}[\mathscr{Z}^{\pi}(\tau(t),\cdot)](\tau(s))
			-
			\overline{\mathscr{Z}}^{\pi}(\tau(t),\tau(s))
		\big|^{2}
	\rd s
\Big]
\Big\}
\\&\leq
3
\Big\{
\e
\Big[
	\big|
		Y(t)
		-
		\mathscr{Y}^{\pi}(\tau(t),t)
	\big|^{2}
\Big]
+
\e
\Big[
	\int_{0}^{\tau(t)}
		\big|
			\mathscr{Z}^{\pi}(\tau(t),s)
			-
			\overline{\mathscr{Z}}^{\pi}(\tau(t),\tau(s))
		\big|^{2}
	\rd s
\Big]
\\&\quad\quad\quad
+
\e
\Big[
	\int_{0}^{\tau(t)}
		\big|
			\mathcal{I}^{\pi}[\mathscr{Z}^{\pi}(\tau(t),\cdot)](\tau(s))
			-
			\overline{\mathscr{Z}}^{\pi}(\tau(t),\tau(s))
		\big|^{2}
	\rd s
\Big]
\Big\}
\\&\leq
3
\Big\{
\e
\Big[
	\big|
		Y(t)
		-
		\mathscr{Y}^{\pi}(\tau(t),t)
	\big|^{2}
\Big]
+
2
\e
\Big[
	\int_{0}^{\tau(t)}
		\big|
			\mathscr{Z}^{\pi}(\tau(t),s)
			-
			\overline{\mathscr{Z}}^{\pi}(\tau(t),\tau(s))
		\big|^{2}
	\rd s
\Big]
\Big\},
\end{align*}
for a.e.\ $t \in [0,T)$, by the Gronwall-like inequality (see Lemma \ref{Lem_Gron_cont}) with
\begin{align*}
a(t)
&=
\e
\Big[
	\big|
		Y(t)
		-
		\mathscr{Y}^{\pi}(\tau(t),t)
	\big|^{2}
	+
	\int_{0}^{T}
		\big|
			Z(t,s)
			-
			\mathscr{Z}^{\pi}(\tau(t),s)
		\big|^{2}
	\rd s
\Big],
\\
b(t)
&=
\big(
	|\pi|
	+
	\rho_{\Psi,G}(|\pi|)^{2}
\big)
\Big(
	M^{2}
	+
	\e\Big[
		\int_{t}^{T}
			\big\{
				\big|G(t,s,0,0,0)\big|^{2}
				+
				\big|Y(s)\big|^{2}
				+
				\big|Z(t,s)\big|^{2}
				+
				\big|Z(s,t)\big|^{2}
			\big\}
		\rd s
	\Big]
\Big),
\\
c(t)
&=
\e
\Big[
	\int_{0}^{\tau(t)}
		\big|
			\mathscr{Z}^{\pi}(\tau(t),s)
			-
			\overline{\mathscr{Z}}^{\pi}(\tau(t),\tau(s))
		\big|^{2}
	\rd s
\Big],
\\
\zeta(t,s,\omega)
&=
\big|
	Z(t,s)
	-
	\mathcal{I}^{\pi}[\mathscr{Z}^{\pi}(\tau(t),\cdot)](\tau(s))
\big|
\1_{[\tau^{*}(s),T)}(t)
~\text{and}~
(S,\Sigma,\mu)
=
(\Omega,\mathcal{F},\p),
\end{align*}
we see that
\begin{align*}
&\int_{0}^{T}
	\e
	\Big[
		\big|
			Y(t)
			-
			\mathscr{Y}^{\pi}(\tau(t),t)
		\big|^{2}
		+
		\int_{0}^{T}
			\big|
				Z(t,s)
				-
				\mathscr{Z}^{\pi}(\tau(t),s)
			\big|^{2}
		\rd s
	\Big]
\rd t
\\&\leq
C
\Big\{
	\big(
		|\pi|
		+
		\rho_{\Psi,G}(|\pi|)^{2}
	\big)
	\Big(
		M^{2}
		+
		\int_{0}^{T}
			\e\Big[
				\int_{t}^{T}
					\big\{
						\big|G(t,s,0,0,0)\big|^{2}
						+
						\big|Y(s)\big|^{2}
						+
						\big|Z(t,s)\big|^{2}
						+
						\big|Z(s,t)\big|^{2}
					\big\}
				\rd s
			\Big]
		\rd t
	\Big)
	\\&\quad\quad\quad+
	\int_{0}^{T}
		\e
		\Big[
			\int_{0}^{\tau(t)}
				\big|
					\mathscr{Z}^{\pi}(\tau(t),s)
					-
					\overline{\mathscr{Z}}^{\pi}(\tau(t),\tau(s))
				\big|^{2}
			\rd s
		\Big]
	\rd t
\Big\}
\\&\leq
C
\Big\{
	M^{2}
	\big(
		|\pi|
		+
		\rho_{\Psi,G}(|\pi|)^{2}
	\big)
	+
	\sum_{k=1}^{N-1}
		\Delta t_{k}
		\sum_{\ell=0}^{k-1}
			\e\Big[
				\int_{t_{\ell}}^{t_{\ell+1}}
					\big|
						\mathscr{Z}^{\pi}(t_{k},s)
						-
						\overline{\mathscr{Z}}^{\pi}(t_{k},t_{\ell})
					\big|^{2}
				\rd s
			\Big]
\Big\},
\end{align*}
where in the last inequality, we used the assumption (iv) in $(\mathrm{H}_{\Psi,G})$ and the estimate \eqref{Prop_BSVIE_0_1}.
This completes the proof.
\end{proof}

As a corollary of the above proposition, we provide an estimate for the $L^{2}$-time regularity of $(Y(\cdot),Z(\cdot,\cdot))$.
For this purpose, for each $\pi=\{t_{0},t_{1},\ldots,t_{N}\} \in \Pi[0,T]$, we introduce
\begin{align}\label{Reg_YZ}
\begin{split}
&
\mathcal{E}(Y;\pi)
:=
\sum^{N-1}_{k=0}
\e \Big[
	\int^{t_{k+1}}_{t_k}
		\big|
			Y(t)
			-
			\overline{Y}^{\pi}(t_k)
		\big|^{2}
	\rd t
\Big],
\\
&
\mathcal{E}(Z;\pi)
:=
\sum^{N-1}_{k=0}
\sum^{N-1}_{\ell=0}
\e\Big[
	\int^{t_{k+1}}_{t_k}
	\int^{t_{\ell+1}}_{t_\ell}
		\big|
			Z(t,s)
			-
			\overline{Z}^{\pi}(t_k,t_\ell)
		\big|^{2}
	\rd s
	\rd t
\Big],
\end{split}
\end{align}
where
\begin{align*}
\begin{split}
&\overline{Y}^{\pi}(t_k)
:=
\frac{1}{\Delta t_{k}}
\e_{t_{k}}
\Big[
	\int^{t_{k+1}}_{t_k}
		Y(t)
	\rd t
\Big],~
k=0,\dots,N-1,\\
&\overline{Z}^{\pi}(t_k,t_\ell)
:=
\frac{1}{\Delta t_{k}}
\frac{1}{\Delta t_{\ell}}
\e_{t_{\ell}}
\Big[
	\int^{t_{k+1}}_{t_k}
		\int^{t_{\ell+1}}_{t_\ell}
			Z(t,s)
		\rd s
	\rd t
\Big],~
k,\ell=0,\dots,N-1.
\end{split}
\end{align*}
Although depending on $\pi$, $\overline{Y}^{\pi}(t_k)$ and $\overline{Z}^{\pi}(t_k,t_{\ell})$ are defined through the true solution $(Y(\cdot),Z(\cdot,\cdot))$, not through some approximations.
Therefore, $\mathcal{E}(Y;\pi)$ and $\mathcal{E}(Z;\pi)$ measure a kind of time regularity.
We call  $\mathcal{E}(Y;\pi)$ and $\mathcal{E}(Z;\pi)$ the {\emph{modulus of the $L^{2}$-time regularity} of $Y(\cdot)$ and $Z(\cdot,\cdot)$, respectively.

\begin{cor}\label{Cor_0}
Under $(\mathrm{H}_{\Psi,G})$, for any $\pi=\{t_{0},t_{1},\ldots,t_{N}\} \in \Pi[0,T]$, it holds that
\begin{align*}%\label{Reg_Y}
\mathcal{E}(Y;\pi)
+
\mathcal{E}(Z;\pi)
&\leq
C
\Big\{
	M^{2}
	\big(
		|\pi|
		+
		\rho_{\Psi,G}(|\pi|)^{2}
	\big)
	+
	\sum_{k=0}^{N-1}
		\Delta t_{k}
		\e
		\Big[
			\int_{t_{k}}^{t_{k+1}}
				\big|
					\mathscr{Z}^{\pi}(t_{k},s)
				\big|^{2}
			\rd s
		\Big]
	\\&\quad+
	\sum_{k=0}^{N-1}
		\Delta t_{k}
		\sum_{\ell=0}^{N-1}
			\e\Big[
				\int_{t_{\ell}}^{t_{\ell+1}}
					\big|
						\mathscr{Z}^{\pi}(t_{k},s)
						-
						\overline{\mathscr{Z}}^{\pi}(t_{k},t_{\ell})
					\big|^{2}
				\rd s
			\Big]
\Big\}.
\end{align*}
\end{cor}
\begin{proof}
Note that $\overline{Y}^{\pi}(t_{k})$ is the best approximation of $Y(\cdot)$ on $[t_{k},t_{k+1}]$ in the following sense:
\begin{align*}
\e\Big[
	\int_{t_{k}}^{t_{k+1}}
		\big|
			Y(t)
			-
			\overline{Y}^{\pi}(t_{k})
		\big|^{2}
	\rd t
\Big]
\leq
\e\Big[
	\int_{t_{k}}^{t_{k+1}}
			\big|
				Y(t)
				-
				\eta
			\big|^{2}
		\rd t
\Big]
\end{align*}
for any $\eta \in L^{2}_{\mathcal{F}_{t_{k}}}(\Omega;\real^{m})$ (see Remark 5.2.5 in \cite{Zh17}).
Similarly, it holds that
\begin{align*}
\e\Big[
	\int_{t_{k}}^{t_{k+1}}
		\int_{t_{\ell}}^{t_{\ell+1}}
			\big|
				Z(t,s)
				-
				\overline{Z}^{\pi}(t_{k},t_{\ell})
			\big|^{2}
		\rd s
	\rd t
\Big]
\leq
\e\Big[
	\int_{t_{k}}^{t_{k+1}}
		\int_{t_{\ell}}^{t_{\ell+1}}
			\big|
				Z(t,s)
				-
				\zeta
			\big|^{2}
		\rd s
	\rd t
\Big]
\end{align*}
for any $\zeta \in L^{2}_{\mathcal{F}_{t_{\ell}}}(\Omega;\real^{m \times d})$.
Therefore, we have
\begin{align*}
&
\mathcal{E}(Y;\pi)
+
\mathcal{E}(Z;\pi)
\\&\leq
\sum_{k=0}^{N-1}
\e\Big[
	\int_{t_{k}}^{t_{k+1}}
		\big|
			Y(t)
			-
			\mathscr{Y}^{\pi}(t_{k},t_{k})
		\big|^{2}
	\rd t
\Big]
+
\sum_{k=0}^{N-1}
\sum_{\ell=0}^{N-1}
\e\Big[
	\int_{t_{k}}^{t_{k+1}}
		\int_{t_{\ell}}^{t_{\ell+1}}
			\big|
				Z(t,s)
				-
				\overline{\mathscr{Z}}^{\pi}(t_{k},t_{\ell})
			\big|^{2}
		\rd s
	\rd t
\Big]
\\&\leq
2\Big\{
	\e
	\Big[
		\int_{0}^{T}
			\big|
				Y(t)
				-
				\mathscr{Y}^{\pi}(\tau(t),t)
			\big|^{2}
		\rd t
		+
		\int_{0}^{T}
			\int_{0}^{T}
				\big|
					Z(t,s)
					-
					\mathscr{Z}^{\pi}(\tau(t),s)
				\big|^{2}
			\rd s
		\rd t
	\Big]
	\\&\quad\quad\quad
	+
	\sum_{k=0}^{N-1}
	\e
	\Big[
		\int_{t_{k}}^{t_{k+1}}
			\big|
				\mathscr{Y}^{\pi}(t_{k},t)
				-
				\mathscr{Y}^{\pi}(t_{k},t_{k})
			\big|^{2}
		\rd t
	\Big]
	+
	\sum_{k=0}^{N-1}
	\Delta t_{k}
	\sum_{\ell=0}^{N-1}
	\e
	\Big[
		\int_{t_{\ell}}^{t_{\ell+1}}
			\big|
				\mathscr{Z}^{\pi}(t_{k},s)
				-
				\overline{\mathscr{Z}}^{\pi}(t_{k},t_{\ell})
			\big|^{2}
		\rd s
	\Big]
\Big\}.
\end{align*}
Noting that
\begin{align*}
\sup_{t \in [t_{k}, t_{k+1}]}
\e
\Big[
	\big|
		\mathscr{Y}^{\pi}(t_{k},t)
		-
		\mathscr{Y}^{\pi}(t_{k},t_{k})
	\big|^{2}
\Big]
=
\sup_{t \in [t_{k}, t_{k+1}]}
\e
\Big[
	\Big|
		\int_{t_{k}}^{t}
			\mathscr{Z}^{\pi}(t_{k},s)
		\rd W(s)
	\Big|^{2}
\Big]
=
\e
\Big[
	\int_{t_{k}}^{t_{k+1}}
		\big|
			\mathscr{Z}^{\pi}(t_{k},s)
		\big|^{2}
	\rd s
\Big],
\end{align*}
by Proposition \ref{Prop_BSDE_0}, we get the assertion.
\end{proof}

\begin{rem}
The terms 
\begin{align*}
\sum_{k=0}^{N-1}\Delta t_{k}
\e\Big[
	\int_{t_{k}}^{t_{k+1}}
		\big|\mathscr{Z}^{\pi}(t_{k},s)\big|^{2}
	\rd s
\Big]
~\text{and}~
\sum_{k=0}^{N-1}\Delta t_{k} \sum_{\ell=0}^{N-1}
\e\Big[
	\int_{t_{\ell}}^{t_{\ell+1}}
		\big|\mathscr{Z}^{\pi}(t_{k},s)-\overline{\mathscr{Z}}^{\pi}(t_{k},t_{\ell})\big|^{2}
	\rd s
\Big]
\end{align*}
measure the $L^{2}$-time regularity of the martingale integrand $\mathscr{Z}^{\pi}$ of BSDE system \ref{BSDE_sys_0} with $\theta =0$.
Proposition \ref{Prop_BSDE_0} and Corollary \ref{Cor_0} show that, for the general Type-\Rnum{2} BSVIE \eqref{BSVIE_II_1} with the stochastic coefficients $(\Psi,G)$, the corresponding error terms can be estimated in terms of the $L^{2}$-time regularity of $\mathscr{Z}^{\pi}$.
In Section \ref{sec_L2_reg}, we investigate the $L^{2}$-time regularity of $\mathscr{Z}^{\pi}$ in the case of Type-\Rnum{2} BSVIE \eqref{Type-2 Markov} with $X(\cdot)$ begin the solution of SDE \eqref{SDE}.
\end{rem}

%%%%%%%%%%%%%%%
%% Section
%%%%%%%%%%%%%%%

\section{Numerical approximations for general Type-\Rnum{2} BSVIEs}\label{sec_Disc_TypeII}

In this section, we construct a backward Euler--Maruyama scheme for the adapted M-solution of the general Type-\Rnum{2} BSVIE \eqref{BSVIE_II_1}, and estimate its $L^{2}$-error in terms of the $L^{2}$-time regularity of $\mathscr{Z}^{\pi}$.

We impose the following assumption on $(\Psi,G)$, which is slightly stronger than the one in the previous section.

\begin{itemize}
\item[$(\mathrm{H}_{\Psi,G})'$]
We suppose that $(\Psi,G)$ satisfies the conditions (i), (ii), (iii) and (iv) in $(\mathrm{H}_{\Psi,G})$ with the constants $L$ and $M$.
Furthermore, we assume that, for any $(t,s), (t',s') \in \Delta [0,T]$ and $(Y,Z_{1},Z_{2}) \in L^{2}_{\mathcal{F}}(\Omega;\real^{m}) \times (L^{2}_{\mathcal{F}}(\Omega;\real^{m \times d}))^{2}$,
\begin{align*}
&
\e
\Big[
	\big|
		\Psi(t)
		-
		\Psi(t')
	\big|^{2}
	+
	\big|
		G
		(
			t,
			s,
			Y,
			Z_{1},
			Z_{2}
		)
		-
		G
		(
			t'
			,s',
			Y,
			Z_{1},
			Z_{2}
		)
	\big|^{2}
\Big]^{1/2}
\\&
\leq
L
\Big\{
	|t-t'|^{1/2}
	+
	|s-s'|^{1/2}
\Big\}
\Big\{
	M
	+
	\e\big[
		\big|Y\big|^{2}
	\big]^{1/2}
	+
	\e\big[
		\big|Z_{1}\big|^{2}
	\big]^{1/2}
	+
	\e\big[
		\big|Z_{2}\big|^{2}
	\big]^{1/2}
\Big\}.
\end{align*}
In addition, suppose that, for each $\pi=\{t_{0},t_{1},\ldots,t_{N}\} \in \Pi[0,T]$, we are given $(\Psi^{\pi},G^{\pi})$ such that
\begin{itemize}
\item
$\Psi^{\pi}(t_{k}) : \Omega \to \real^{m}$ is $\mathcal{F}_{T}$-measurable for each $k=0,\ldots, N-1$;
\item
$G^{\pi}(t_{k},t_{\ell},\cdot,\cdot,\cdot): \Omega \times \real^{m} \times \real^{m \times d} \times \real^{m \times d}\to \real^{m}$ is $\mathcal{F}_{t_{\ell}} \otimes \mathcal{B}(\real^{m}) \otimes \mathcal{B}(\real^{m \times d}) \otimes \mathcal{B}(\real^{m \times d})$-measurable for each $k=0,\ldots,N-1$ and $\ell=k,\ldots,N-1$;
\item
For each $k=0,\ldots,N-1$, $\ell=k,\ldots,N-1$ and $(y,z_{1},z_{2}), (y',z_{1}',z_{2}') \in \real^{m} \times \real^{m \times d} \times \real^{m \times d}$,
\begin{align*}
|G^{\pi}(t_{k},t_{\ell},y,z_{1},z_{2})-G^{\pi}(t_{k},t_{\ell},y',z_{1}',z_{2}')|
\leq
L
\big\{
	|y-y'|
	+
	|z_{1}-z_{1}'|
	+
	|z_{2}-z_{2}'|
\big\};
\end{align*}
\item
For each $k=0,\ldots,N-1$ and $\ell=k,\ldots,N-1$,
\begin{align*}
\e\Big[
	\big|
		\Psi^{\pi}(t_{k})
	\big|^{2}
	+
	\big|
		G^{\pi}
		(
			t_{k},
			t_{\ell},
			0,
			0,
			0
		)
	\big|^{2}
\Big]
<\infty;
\end{align*}
\item
For each $k=0,\ldots,N-1$, $\ell=k,\ldots,N-1$ and $(Y,Z_{1},Z_{2}) \in L^{2}_{\mathcal{F}}(\Omega;\real^{m}) \times (L^{2}_{\mathcal{F}}(\Omega;\real^{m \times d}))^{2}$,
\begin{align*}
&\e
\Big[
	\big|
		\Psi(t_{k})
		-
		\Psi^{\pi}(t_{k})
	\big|^{2}
	+
	\big|
		G
		(
			t_{k},
			t_{\ell},
			Y,
			Z_{1},
			Z_{2}
		)
		-
		G^{\pi}
		(
			t_{k},
			t_{\ell},
			Y,
			Z_{1},
			Z_{2}
		)
	\big|^{2}
\Big]^{1/2}
\\&\leq
L
|\pi|^{1/2}
\Big\{
	M
	+
	\e\big[
		\big|Y\big|^{2}
	\big]^{1/2}
	+
	\e\big[
		\big|Z_{1}\big|^{2}
	\big]^{1/2}
	+
	\e\big[
		\big|Z_{2}\big|^{2}
	\big]^{1/2}
\Big\}.
\end{align*}

\end{itemize}
\end{itemize}

\begin{rem}\label{Rem_EM_0}
Compared to $(\mathrm{H}_{\Psi,G})$, the above assumption imposes the continuity with respect to the time parameter $s$ on $G$.
Under $(\mathrm{H}_{\Psi,G})'$, Proposition \ref{Prop_BSDE_0} and Corollary \ref{Cor_0} hold for $\rho_{\Psi,G}(|\pi|)$ replaced by $L|\pi|^{1/2}$.
In Section \ref{sec_L2_reg}, we will choose
\begin{align*}
\Psi^{\pi}(t_{k})
=
\psi(t_{k},X^{\pi}(t_{k}),X^{\pi}(t_{N}))
~\text{and}~
G^{\pi}(t_{k},t_{\ell},y,z_{1},z_{2})
=
g(t_{k},t_{\ell},X^{\pi}(t_{k}),X^{\pi}(t_{\ell}), y, z_{1}, z_{2})
\end{align*}
with suitable deterministic functions $\psi$ and $g$, where $\{X^{\pi}(t_{k})\}_{k=0}^{N}$ is the Euler--Maruyama scheme for an SDE.
\end{rem}

Now we define an approximation scheme for Type-\Rnum{2} BSVIE \eqref{BSVIE_II_1} based on a backward Euler--Maruyama scheme.
For each $\pi=\{t_{0},t_{1},\ldots,t_{N}\} \in \Pi[0,T]$, define $\{(Y^{\pi}(t_{k},t_{\ell}),Z^{\pi}(t_{k},t_{\ell}))\}_{k,\ell=0}^{N-1}$ by
\begin{align}\label{EM_0}
\begin{cases}
Y^{\pi}(t_{k},t_{\ell})
\displaystyle
=
\e_{t_{\ell}}
\Big[
	Y^{\pi}(t_{k},t_{\ell+1})
\Big]
\\\displaystyle \quad\quad\quad\quad\quad\quad
+
\Delta t_{\ell}
G^{\pi}
(
	t_{k},
	t_{\ell},
	Y^{\pi}(t_{\ell},t_{\ell}),
	Z^{\pi}(t_{k},t_{\ell}),
	Z^{\pi}(t_{\ell},t_{k})
)
\1_{\{k<\ell\}},
~k,\ell=0,\ldots,N-1,
\\
Z^{\pi}(t_{k},t_{\ell})
\displaystyle
=
\frac{1}{\Delta t_{\ell}}
\e_{t_{\ell}}
\Big[
	Y^{\pi}(t_{k},t_{\ell+1})
	\Delta W_{\ell}^{\top}
\Big],
~k,\ell=0,\ldots,N-1,
\end{cases}
\end{align}
with $Y^{\pi}(t_{k},t_{N}):=\Psi^{\pi}(t_{k})$, $k=0,\ldots,N-1$.

The above system is an explicit scheme.
More precisely, $\{(Y^{\pi}(t_{k},t_{\ell}),Z^{\pi}(t_{k},t_{\ell}))\}_{k,\ell=0}^{N-1}$ can be constructed backward inductively as follows.
For $k=N-1$,
$
(Y^{\pi}(t_{N-1},t_{\ell}),Z^{\pi}(t_{N-1},t_{\ell}))
\in
L^{2}_{\mathcal{F}_{t_{\ell}}}(\Omega;\real^{m})
\times
L^{2}_{\mathcal{F}_{t_{\ell}}}(\Omega;\real^{m \times d})
$, $\ell=0,\ldots,N-1$, are defined by the backward induction with respect to $\ell$;
\begin{align*}
\begin{cases}
Y^{\pi}(t_{N-1},t_{\ell})
=
\e_{t_{\ell}}
\Big[
	Y^{\pi}(t_{N-1},t_{\ell+1})
\Big],
\\
Z^{\pi}(t_{N-1},t_{\ell})
=
\displaystyle
\frac{1}{\Delta t_{\ell}}
\e_{t_{\ell}}
\Big[
	Y^{\pi}(t_{N-1},t_{\ell+1})
	\Delta W_{\ell}^{\top}
\Big],
\end{cases}
\end{align*}
for $\ell=0,\ldots,N-1$, with the terminal condition $Y^{\pi}(t_{N-1},t_{N})=\Psi^{\pi}(t_{N-1})$.
Next, fix $k' \in \{0,\ldots,N-2\}$ and assume that we have already constructed
\begin{align*}
(Y^{\pi}(t_{k},t_{\ell}),Z^{\pi}(t_{k},t_{\ell}))
\in
L^{2}_{\mathcal{F}_{t_{\ell}}}(\Omega;\real^{m})
\times
L^{2}_{\mathcal{F}_{t_{\ell}}}(\Omega;\real^{m \times d}),
~\ell=0,\ldots,N-1,
\end{align*}
for $k=k'+1,\ldots,N-1$.
Specifically, we are given $Y^{\pi}(t_{\ell},t_{\ell})$ and  $Z^{\pi}(t_{\ell},t_{k'})$ for $\ell=k'+1,\ldots,N-1$.
Then $(Y^{\pi}(t_{k'},t_{\ell}),Z^{\pi}(t_{k'},t_{\ell})) \in L^{2}_{\mathcal{F}_{t_{\ell}}}(\Omega;\real^{m}) \times L^{2}_{\mathcal{F}_{t_{\ell}}}(\Omega;\real^{m \times d})$, $\ell=0,\ldots,N-1$, are defined by the backward induction with respect to $\ell$;
\begin{align*}
\begin{cases}
Y^{\pi}(t_{k'},t_{\ell})
\displaystyle
=
\e_{t_{\ell}}
\Big[
	Y^{\pi}(t_{k'},t_{\ell+1})
\Big]
+
\Delta t_{\ell}
G^{\pi}
(
	t_{k'},
	t_{\ell},
	Y^{\pi}(t_{\ell},t_{\ell}),
	Z^{\pi}(t_{k'},t_{\ell}),
	Z^{\pi}(t_{\ell},t_{k'})
),
\\
Z^{\pi}(t_{k'},t_{\ell})
\displaystyle
=
\frac{1}{\Delta t_{\ell}}
\e_{t_{\ell}}
\Big[
	Y^{\pi}(t_{k'},t_{\ell+1})
	\Delta W_{\ell}^{\top}
\Big],
\end{cases}
\end{align*}
for $\ell=k'+1,\ldots,N-1$, and
\begin{align*}
\begin{cases}
Y^{\pi}(t_{k'},t_{\ell})
\displaystyle
=
\e_{t_{\ell}}
\Big[
	Y^{\pi}(t_{k'},t_{\ell+1})
\Big],
\\Z^{\pi}(t_{k'},t_{\ell})
\displaystyle
=
\frac{1}{\Delta t_{\ell}}
\e_{t_{\ell}}
\Big[
	Y^{\pi}(t_{k'},t_{\ell+1})
	\Delta W_{\ell}^{\top}
\Big],
\end{cases}
\end{align*}
for $\ell=0,\ldots,k'$, with the terminal condition $Y^{\pi}(t_{k'},t_{N})=\Psi^{\pi}(t_{k'})$.
By the backward induction for $k$, we can construct the backward Euler--Maruyama scheme $\{(Y^{\pi}(t_{k},t_{\ell}),Z^{\pi}(t_{k},t_{\ell}))\}_{k,\ell=0}^{N-1}$.

\begin{rem}
Even in the case of Type-\Rnum{1} BSVIEs (where $G$ and $G^{\pi}$ do not depend on $z_{2}$), the construction of the backward Euler--Maruyama scheme \eqref{EM_0} is slightly different from that of Wang \cite{WaY18}.
He considered a scheme for Type-\Rnum{1} BSVIEs of the following form:
\begin{align}\label{EM_Wang}
\begin{cases}
\displaystyle
Y^{\pi}(t_{k},t_{\ell})
=
\e_{t_{\ell}}\Big[
	Y^{\pi}(t_{k},t_{\ell+1})
	+
	\Delta t_{\ell}
	G^{\pi}(t_{k},t_{\ell},Y^{\pi}(t_{\ell},t_{\ell+1}),Z^{\pi}(t_{k},t_{\ell}))
\Big],
\\
\displaystyle
Z^{\pi}(t_{k},t_{\ell})
=
\frac{1}{\Delta t_{\ell}}
\e_{t_{\ell}}
\Big[
	\big\{
		Y^{\pi}(t_{k},t_{\ell+1})
		+
		\Delta t_{\ell}
		G^{\pi}(t_{k},t_{\ell},Y^{\pi}(t_{\ell},t_{\ell+1}),Z^{\pi}(t_{k},t_{\ell}))
	\big\}
	\Delta W_{\ell}^{\top}
\Big],
\end{cases}
0\leq k\leq \ell \leq N-1,
\end{align}
with $Y^{\pi}(t_{k},t_{N})=\Psi^{\pi}(t_{k})$, $k=0,\ldots,N-1$.
Compared with this scheme, in our definition of the backward Euler--Maruyama scheme \eqref{EM_0}, $Y^{\pi}(t_{\ell},t_{\ell+1})$ in $G^{\pi}$ is replaced by $Y^{\pi}(t_{\ell},t_{\ell})$, and $G^{\pi}$ vanishes on the ``diagonal region'' $k=\ell$.
Furthermore, his construction of $Z^{\pi}(t_{k},t_{\ell})$ is indeed an implicit form, and thus, in order to calculate \eqref{EM_Wang}, we have to solve the implicit equations for $Z^{\pi}(t_{k},t_{\ell})$.
On the other hand, our construction is an explicit scheme.
For the convergence of the scheme \eqref{EM_Wang}, as mentioned in \cite{WaY18}, a kind of uniformity condition for the partition $\pi \in \Pi[0,T]$ is needed.
However, it turns out that our scheme \eqref{EM_0} converges without any structural conditions for $\pi$.
For this reason, we guess that our construction of the backward Euler--Maruyama scheme \eqref{EM_0} is more natural than that of \cite{WaY18}.

\end{rem}

Consider the scheme \eqref{EM_0}.
By using the martingale representation theorem, for each $k=0,\ldots,N-1$, we define a stochastic process $\widehat{Z}^{\pi}(t_{k},\cdot) \in L^{2}_{\mathbb{F}}(0,T;\real^{m \times d})$ by
\begin{align*}
Y^{\pi}(t_{k},t_{\ell+1})
=
\e_{t_{\ell}}
\Big[
	Y^{\pi}(t_{k},t_{\ell+1})
\Big]
+
\int_{t_{\ell}}^{t_{\ell+1}}
	\widehat{Z}^{\pi}(t_{k},s)
\rd W(s),
~\ell=0,\ldots,N-1.
\end{align*}
Then we have
\begin{align*}
Z^{\pi}(t_{k},t_{\ell})
=
\frac{1}{\Delta t_{\ell}}
\e_{t_{\ell}}
\Big[
	\int_{t_{\ell}}^{t_{\ell+1}}
		\widehat{Z}^{\pi}(t_{k},s)
	\rd s
\Big],
~k,\ell=0,\ldots,N-1,
\end{align*}
and
\begin{align*}%\label{EM_1}
Y^{\pi}(t_{k},t_{k})
=
\e
\Big[
	Y^{\pi}(t_{k},t_{k})
\Big]
+
\int_{0}^{t_{k}}
	\widehat{Z}^{\pi}(t_{k},s)
\rd W(s),
~k=0,\ldots,N-1.
\end{align*}

Now we provide an estimate of the term
\begin{align*}
\sum_{k=0}^{N-1}
\Delta t_{k}
\e
\Big[
	\big|
		\mathscr{Y}^{\pi}(t_{k},t_{k})
		-
		Y^{\pi}(t_{k},t_{k})
	\big|^{2}
\Big]
+
\sum_{k=0}^{N-1}
\Delta t_{k}
\sum_{\ell=0}^{N-1}
\Delta t_{\ell}
\e
\Big[
	\big|
		\overline{\mathscr{Z}}^{\pi}(t_{k},t_{\ell})
		-
		Z^{\pi}(t_{k},t_{\ell})
	\big|^{2}
\Big].
\end{align*}
For this purpose, let us introduce the following notations:
\begin{align*}
y_{k,\ell}
&:=
\e
\Big[
	\big|
		\mathscr{Y}^{\pi}(t_{k},t_{\ell})
		-
		Y^{\pi}(t_{k},t_{\ell})
	\big|^{2}
\Big],~
z_{k,\ell}
:=
\Delta t_{\ell}
\e
\Big[
	\big|
		\overline{\mathscr{Z}}^{\pi}(t_{k},t_{\ell})
		-
		Z^{\pi}(t_{k},t_{\ell})
	\big|^{2}
\Big],
\\
I_{k,\ell}
&:=
\e\Big[
	\int_{t_{\ell}}^{t_{\ell+1}}
		\big\{
			\big|
				\mathscr{Y}^{\pi}(t_{\ell},s)
			\big|^{2}
			+
			\big|
				\mathscr{Z}^{\pi}(t_{k},s)
			\big|^{2}
			+
			\big|
				\mathcal{I}^{\pi}[\mathscr{Z}^{\pi}(t_{\ell},\cdot)](t_{k})
			\big|^{2}
		\big\}
	\rd s
\Big],
\\
\varepsilon_{k,\ell}
&:=
\e\Big[
	\int_{t_{\ell}}^{t_{\ell+1}}
		\big\{
			\big|
				\mathscr{Y}^{\pi}(t_{\ell},s)
				\mathalpha{-}
				\mathscr{Y}^{\pi}(t_{\ell},t_{\ell})
			\big|^{2}
			+
			\big|
				\mathscr{Z}^{\pi}(t_{k},s)
				\mathalpha{-}
				\overline{\mathscr{Z}}^{\pi}(t_{k},t_{\ell})
			\big|^{2}
			+
			\big|
				\mathcal{I}^{\pi}[\mathscr{Z}^{\pi}(t_{\ell},\cdot)](t_{k})
				\mathalpha{-}
				\overline{\mathscr{Z}}^{\pi}(t_{\ell},t_{k})
			\big|^{2}
		\big\}
	\rd s
\Big],
\end{align*}
for $k,\ell = 0,\ldots,N-1$, and
\begin{align*}
\Delta G_{k,\ell}(s)
&:=
G
(
	t_{k},
	s,
	\mathscr{Y}^{\pi}(t_{\ell},s),
	\mathscr{Z}^{\pi}(t_{k},s),
	\mathcal{I}^{\pi}[\mathscr{Z}^{\pi}(t_{\ell},\cdot)](t_{k})
)
-
G^{\pi}
(
	t_{k},
	t_{\ell},
	Y^{\pi}(t_{\ell},t_{\ell}),
	Z^{\pi}(t_{k},t_{\ell}),
	Z^{\pi}(t_{\ell},t_{k})
),
\end{align*}
for $s \in [t_{\ell},t_{\ell+1})$, $\ell=k+1,\ldots,N-1$ and $k=0,\ldots,N-2$.

\begin{lemm}\label{Lem_EM_0}
Suppose that $(\mathrm{H}_{\Psi,G})'$ holds and let $\pi=\{t_{0},t_{1},\ldots,t_{N}\} \in \Pi[0,T]$ be fixed.
\begin{itemize}
\item[(i)]
It holds that
\begin{align}\label{Lem_EM_0_0}
\sum_{\ell=0}^{k-1}
z_{k,\ell}
\leq
y_{k,k},
~k=1,\ldots,N-1.
\end{align}
\item[(ii)]
There exists a constant $\delta>0$ depending only on $L$ such that, for any $\pi=\{t_{0},t_{1},\ldots,t_{N}\} \in \Pi[0,T]$ with $|\pi| \leq \delta$,
\begin{align}\label{Lem_EM_0_1}
\begin{split}
y_{k,\ell}
+
\frac{1}{2}
z_{k,\ell}
&\leq
(1+C \Delta t_{\ell})
y_{k,\ell+1}
+
\frac{1}{2}
\Big\{
	\Delta t_{\ell}
	y_{\ell,\ell}
	+
	\frac{\Delta t_{\ell}}{\Delta t_{k}}
	z_{\ell,k}
	+
	|\pi|
	\Big(
		\Delta t_{\ell}
		M^{2}
		+
		I_{k,\ell}
	\Big)
	+
	\varepsilon_{k,\ell}
\Big\},
\\&\quad\quad
\ell=k+1,\ldots,N-1,~k=0,\ldots,N-2.
\end{split}
\end{align}
\end{itemize}
\end{lemm}
\begin{proof}
\emph{Proof of (i).}
For each $k=1,\ldots,N-1$, we have
\begin{align*}
\sum_{\ell=0}^{k-1}
z_{k,\ell}
&=
\sum_{\ell=0}^{k-1}
\Delta t_{\ell}
\e
\Big[
	\Big|
		\frac{1}{\Delta t_{\ell}}
		\e_{t_{\ell}}
		\Big[
			\int_{t_{\ell}}^{t_{\ell+1}}
				\big\{
					\mathscr{Z}^{\pi}(t_{k},s)
					-
					\widehat{Z}^{\pi}(t_{k},s)
				\big\}
			\rd s
		\Big]
	\Big|^{2}
\Big]
\\&\leq
\e\Big[
	\int_{0}^{t_{k}}
		\big|
			\mathscr{Z}^{\pi}(t_{k},s)
			-
			\widehat{Z}^{\pi}(t_{k},s)
		\big|^{2}
	\rd s
\Big]
\\&=
\e
\Big[
	\Big|
		\int_{0}^{t_{k}}
			\big\{
				\mathscr{Z}^{\pi}(t_{k},s)
				-
				\widehat{Z}^{\pi}(t_{k},s)
			\big\}
		\rd W(s)
	\Big|^{2}
\Big]
\\&=
\e
\Big[
	\big|
		\mathscr{Y}^{\pi}(t_{k},t_{k})
		-
		Y^{\pi}(t_{k},t_{k})
		-
		\e
		\big[
			\mathscr{Y}^{\pi}(t_{k},t_{k})
					-
					Y^{\pi}(t_{k},t_{k})
		\big]
	\big|^{2}
\Big]
\\&\leq
\e
\Big[
	\big|
		\mathscr{Y}^{\pi}(t_{k},t_{k})
		-
		Y^{\pi}(t_{k},t_{k})
	\big|^{2}
\Big]
=
y_{k,k}.
\end{align*}
This proves \eqref{Lem_EM_0_0}.

\emph{Proof of (ii).}
Let $\gamma>0$ be an arbitrary constant.
By using Young's inequality, for each $k=0,\ldots,N-2$ and $\ell=k+1,\ldots,N-1$, we have
\begin{align*}%\label{Lem_EM_0_3}
&y_{k,\ell}
+
z_{k,\ell}
\\&
\leq
\e
\Big[
	\big|
		\mathscr{Y}^{\pi}(t_{k},t_{\ell})
		-
		Y^{\pi}(t_{k},t_{\ell})
	\big|^{2}
	+
	\int_{t_{\ell}}^{t_{\ell+1}}
		\big|
			\mathscr{Z}^{\pi}(t_{k},s)
			-
			\widehat{Z}^{\pi}(t_{k},s)
		\big|^{2}
	\rd s
\Big]
\\&=
\e
\Big[
	\big|
		\mathscr{Y}^{\pi}(t_{k},t_{\ell})
		-
		Y^{\pi}(t_{k},t_{\ell})
	+
	\int_{t_{\ell}}^{t_{\ell+1}}
		\big\{
			\mathscr{Z}^{\pi}(t_{k},s)
			-
			\widehat{Z}^{\pi}(t_{k},s)
		\big\}
	\rd W(s)
	\Big|^{2}
\Big]
\\&=
\e
\Big[
	\big|
		\mathscr{Y}^{\pi}(t_{k},t_{\ell+1})
		-
		Y^{\pi}(t_{k},t_{\ell+1})
		+
		\int_{t_{\ell}}^{t_{\ell+1}}
			\Delta G_{k,\ell}(s)
		\rd s
	\big|^{2}
\Big]
\\&\leq
(1+\gamma \Delta t_{\ell})
y_{k,\ell+1}
+
\Big(
	1
	+
	\frac{1}{\gamma \Delta t_{\ell}}
\Big)
\e
\Big[
	\Big|
		\int_{t_{\ell}}^{t_{\ell+1}}
			\Delta G_{k,\ell}(s)
		\rd s
	\Big|^{2}
\Big]
\\&\leq
(1+\gamma \Delta t_{\ell})
y_{k,\ell+1}
+
\Big(
	|\pi|
	+
	\frac{1}{\gamma}
\Big)
\e
\Big[
	\int_{t_{\ell}}^{t_{\ell+1}}
		\big|
			\Delta G_{k,\ell}(s)
		\big|^{2}
	\rd s
\Big]
\\&\leq
(1+\gamma \Delta t_{\ell})
y_{k,\ell+1}
\\*&\quad+
2\Big(
	|\pi|
	+
	\frac{1}{\gamma}
\Big)
\e
\Big[
	\int_{t_{\ell}}^{t_{\ell+1}}
		\big|
			G
			(
				t_{k},
				s,
				\mathscr{Y}^{\pi}(t_{\ell},s),
				\mathscr{Z}^{\pi}(t_{k},s),
				\mathcal{I}^{\pi}[\mathscr{Z}^{\pi}(t_{\ell},\cdot)](t_{k})
			)
			\\*&\quad\quad\quad\quad\quad\quad\quad\quad\quad\quad\quad
			-
			G^{\pi}
			(
				t_{k},
				t_{\ell},
				\mathscr{Y}^{\pi}(t_{\ell},s),
				\mathscr{Z}^{\pi}(t_{k},s),
				\mathcal{I}^{\pi}[\mathscr{Z}^{\pi}(t_{\ell},\cdot)](t_{k})
			)
		\big|^{2}
	\rd s
\Big]
\\&\quad+
2\Big(
	|\pi|
	+
	\frac{1}{\gamma}
\Big)
\e
\Big[
	\int_{t_{\ell}}^{t_{\ell+1}}
		\big|
			G^{\pi}
			(
				t_{k},
				t_{\ell},
				\mathscr{Y}^{\pi}(t_{\ell},s),
				\mathscr{Z}^{\pi}(t_{k},s),
				\mathcal{I}^{\pi}[\mathscr{Z}^{\pi}(t_{\ell},\cdot)](t_{k})
			)
			\\&\quad\quad\quad\quad\quad\quad\quad\quad\quad\quad\quad
			-
			G^{\pi}
			(
				t_{k},
				t_{\ell},
				Y^{\pi}(t_{\ell},t_{\ell}),
				Z^{\pi}(t_{k},t_{\ell}),
				Z^{\pi}(t_{\ell},t_{k})
			)
		\big|^{2}
	\rd s
\Big]
\\&\leq
(1+\gamma \Delta t_{\ell})
y_{k,\ell+1}
+
C
\Big(
	|\pi|
	+
	\frac{1}{\gamma}
\Big)
|\pi|
\Big(
	\Delta t_{\ell}M^{2}
	+
	I_{k,\ell}
\Big)
\\&\quad
+
C
\Big(
	|\pi|
	+
	\frac{1}{\gamma}
\Big)
\int_{t_{\ell}}^{t_{\ell+1}}
	\e
	\Big[
		\big|
			\mathscr{Y}^{\pi}(t_{\ell},s)
			-
			\mathscr{Y}^{\pi}(t_{\ell},t_{\ell})
		\big|^{2}
		+
		\big|
			\mathscr{Z}^{\pi}(t_{k},s)
			-
			\overline{\mathscr{Z}}^{\pi}(t_{k},t_{\ell})
		\big|^{2}
		\\&\quad\quad\quad\quad\quad\quad\quad\quad\quad\quad\quad\quad
		+
		\big|
			\mathcal{I}^{\pi}[\mathscr{Z}^{\pi}(t_{\ell},\cdot)](t_{k})
			-
			\overline{\mathscr{Z}}^{\pi}(t_{\ell},t_{k})
		\big|^{2}
		+
		\big|
			\mathscr{Y}^{\pi}(t_{\ell},t_{\ell})
			-
			Y^{\pi}(t_{\ell},t_{\ell})
		\big|^{2}
		\\&\quad\quad\quad\quad\quad\quad\quad\quad\quad\quad\quad\quad
		+
		\big|
			\overline{\mathscr{Z}}^{\pi}(t_{k},t_{\ell})
			-
			Z^{\pi}(t_{k},t_{\ell})
		\big|^{2}
		+
		\big|
			\overline{\mathscr{Z}}^{\pi}(t_{\ell},t_{k})
			-
			Z^{\pi}(t_{\ell},t_{k})
		\big|^{2}
	\Big]
\rd s
\\&\leq
(1+\gamma \Delta t_{\ell})
y_{k,\ell+1}
+
C
\Big(
	|\pi|
	+
	\frac{1}{\gamma}
\Big)
\Big\{
	|\pi|
	\Big(
		\Delta t_{\ell}M^{2}
		+
		I_{k,\ell}
	\Big)
	+
	\varepsilon_{k,\ell}
	+
	\Delta t_{\ell}
	y_{\ell,\ell}
	+
	z_{k,\ell}
	+
	\frac{\Delta t_{\ell}}{\Delta t_{k}} z_{\ell,k}
\Big\}.
\end{align*}
In the last line of the above inequalities, the constant $C>0$ depends only on $L$.
By letting $\gamma=4C$ and $\delta=\frac{1}{4C}$, for $\pi=\{t_{0},t_{1},\ldots,t_{N}\} \in \Pi[0,T]$ with $|\pi| \leq \delta$, we have $C(|\pi|+\frac{1}{\gamma}) \leq \frac{1}{2}$, and thus \eqref{Lem_EM_0_1} holds.
\end{proof}

In the following, $\delta>0$ denotes the constant appearing in Lemma \ref{Lem_EM_0}.
\begin{lemm}\label{Prop_EM_0}
Under $(\mathrm{H}_{\Psi,G})'$, for any $\pi=\{t_{0},t_{1},\ldots,t_{N}\} \in \Pi[0,T]$ with $|\pi| \leq \delta$, it holds that
\begin{align*}
&
\sum_{k=0}^{N-1}
\Delta t_{k}
\e
\Big[
	\big|
		\mathscr{Y}^{\pi}(t_{k},t_{k})
		-
		Y^{\pi}(t_{k},t_{k})
	\big|^{2}
\Big]
+
\sum_{k=0}^{N-1}
\Delta t_{k}
\sum_{\ell=0}^{N-1}
\Delta t_{\ell}
\e
\Big[
	\big|
		\overline{\mathscr{Z}}^{\pi}(t_{k},t_{\ell})
		-
		Z^{\pi}(t_{k},t_{\ell})
	\big|^{2}
\Big]
\\&\leq
C
\Big\{
	M^{2}|\pi|
	+
	\sum_{k=0}^{N-1}
	\Delta t_{k}
	\e
	\Big[
		\int_{t_{k}}^{t_{k+1}}
			\big|
				\mathscr{Z}^{\pi}(t_{k},s)
			\big|^{2}
		\rd s
	\Big]
	+
	\sum_{k=0}^{N-1}
	\Delta t_{k}
	\sum_{\ell=0}^{N-1}
	\e
	\Big[
		\int_{t_{\ell}}^{t_{\ell+1}}
			\big|
				\mathscr{Z}^{\pi}(t_{k},s)
				-
				\overline{\mathscr{Z}}^{\pi}(t_{k},t_{\ell})
			\big|^{2}
		\rd s
	\Big]
\Big\}.
\end{align*}
\end{lemm}
\begin{proof}
First, we show that, for each $k=0,\ldots,N-1$,
\begin{align}\label{Lem_EM_0_2}
\begin{split}
y_{k,k}
+
\sum_{\ell=0}^{N-1}
z_{k,\ell}
&\leq
C
\Big\{
	|\pi|
	\Big(
		M^{2}
		+
		\sum_{\ell=k+1}^{N-1}
		I_{k,\ell}
	\Big)
	+
	\sum_{\ell=k+1}^{N-1}
	\varepsilon_{k,\ell}
	+
	\e
	\Big[
		\Big(
			\sum_{\ell=k+1}^{N-1}
			\Delta t_{\ell}
			\big|
				\overline{\mathscr{Z}}^{\pi}(t_{k},t_{\ell})
				-
				Z^{\pi}(t_{k},t_{\ell})
			\big|
		\Big)^{2}
	\Big]
	\\&\quad\quad\quad
	+
	\sum_{\ell=k+1}^{N-1}
	\Delta t_{\ell}
	y_{\ell,\ell}
	+
	\e
	\Big[
		\Big(
			\sum_{\ell=k+1}^{N-1}
			\Delta t_{\ell}
			\big|
				\overline{\mathscr{Z}}^{\pi}(t_{\ell},t_{k})
				-
				Z^{\pi}(t_{\ell},t_{k})
			\big|
		\Big)^{2}
	\Big]
\Big\}.
\end{split}
\end{align}
Indeed, we have
\begin{align*}
&y_{k,k}
+
\sum_{\ell=k}^{N-1}
z_{k,\ell}
\\&
\leq
\e
\Big[
	\big|
		\mathscr{Y}^{\pi}(t_{k},t_{k})
		-
		Y^{\pi}(t_{k},t_{k})
	\big|^{2}
	+
	\int_{t_{k}}^{t_{N}}
		\big|
			\mathscr{Z}^{\pi}(t_{k},s)
			-
			\widehat{Z}^{\pi}(t_{k},s)
		\big|^{2}
	\rd s
\Big]
\\&=
\e
\Big[
	\big|
		\mathscr{Y}^{\pi}(t_{k},t_{k})
		-
		Y^{\pi}(t_{k},t_{k})
		+
		\int_{t_{k}}^{t_{N}}
			\big\{
				\mathscr{Z}^{\pi}(t_{k},s)
				-
				\widehat{Z}^{\pi}(t_{k},s)
			\big\}
		\rd W(s)
	\big|^{2}
\Big]
\\&=
\e
\Big[
	\Big|
		\Psi(t_{k})
		-
		\Psi^{\pi}(t_{k})
	+
	\sum_{\ell=k+1}^{N-1}
	\int_{t_{\ell}}^{t_{\ell+1}}
		\Delta G_{k,\ell}(s)
	\rd s
	\Big|^{2}
\Big]
\\&
\leq
C\e\Big[
	\big|\Psi(t_{k})-\Psi^{\pi}(t_{k})\big|^{2}
	\\&\quad\quad\quad
	+
	\sum_{\ell=k+1}^{N-1}
	\int_{t_{\ell}}^{t_{\ell+1}}
		\big|
			G
			(
				t_{k},
				s,
				\mathscr{Y}^{\pi}(t_{\ell},s),
				\mathscr{Z}^{\pi}(t_{k},s),
				\mathcal{I}^{\pi}[\mathscr{Z}^{\pi}(t_{\ell},\cdot)](t_{k})
			)
			\\*&\quad\quad\quad\quad\quad\quad\quad\quad\quad\quad\quad
			-
			G^{\pi}
			(
				t_{k},
				t_{\ell},
				\mathscr{Y}^{\pi}(t_{\ell},s),
				\mathscr{Z}^{\pi}(t_{k},s),
				\mathcal{I}^{\pi}[\mathscr{Z}^{\pi}(t_{\ell},\cdot)](t_{k})
			)
		\big|^{2}
	\rd s
\Big]
\\&\quad
+
C\e\Big[
	\Big(
		\sum_{\ell=k+1}^{N-1}
		\int_{t_{\ell}}^{t_{\ell+1}}
			\big|
				G^{\pi}
				(
					t_{k},
					t_{\ell},
					\mathscr{Y}^{\pi}(t_{\ell},s),
					\mathscr{Z}^{\pi}(t_{k},s),
					\mathcal{I}^{\pi}[\mathscr{Z}^{\pi}(t_{\ell},\cdot)](t_{k})
				)
				\\&\quad\quad\quad\quad\quad\quad\quad\quad\quad\quad\quad
				-
				G^{\pi}
				(
					t_{k},
					t_{\ell},
					Y^{\pi}(t_{\ell},t_{\ell}),
					Z^{\pi}(t_{k},t_{\ell}),
					Z^{\pi}(t_{\ell},t_{k})
				)
			\big|
		\rd s
	\Big)^{2}
\Big]
\\&
\leq
C
|\pi|
\Big(
	M^{2}
	+
	\sum_{\ell=k+1}^{N-1}
	I_{k,\ell}
\Big)
\\&\quad
+
C
\e\Big[
	\Big(
		\sum_{\ell=k+1}^{N-1}
		\int_{t_{\ell}}^{t_{\ell+1}}
			\big\{
				\big|
					\mathscr{Y}^{\pi}(t_{\ell},s)
					-
					\mathscr{Y}^{\pi}(t_{\ell},t_{\ell})
				\big|
				+
				\big|
					\mathscr{Z}^{\pi}(t_{k},s)
					-
					\overline{\mathscr{Z}}^{\pi}(t_{k},t_{\ell})
				\big|
				\\&\quad\quad\quad\quad\quad\quad\quad\quad\quad\quad\quad
				+
				\big|
					\mathcal{I}^{\pi}[\mathscr{Z}^{\pi}(t_{\ell},\cdot)](t_{k})
					-
					\overline{\mathscr{Z}}^{\pi}(t_{\ell},t_{k})
				\big|
			\big\}
		\rd s
	\Big)^{2}
\Big]
\\&\quad
+
C
\e\Big[
	\Big(
		\sum_{\ell=k+1}^{N-1}
		\Delta t_{\ell}
		\big\{
			\big|
				\mathscr{Y}^{\pi}(t_{\ell},t_{\ell})
				-
				Y^{\pi}(t_{\ell},t_{\ell})
			\big|
			+
			\big|
				\overline{\mathscr{Z}}^{\pi}(t_{k},t_{\ell})
				-
				Z^{\pi}(t_{k},t_{\ell})
			\big|
			+
			\big|
				\overline{\mathscr{Z}}^{\pi}(t_{\ell},t_{k})
				-
				Z^{\pi}(t_{\ell},t_{k})
			\big|
		\big\}
	\Big)^{2}
\Big]
\\&\leq
C
\Big\{
	|\pi|
	\Big(
		M^{2}
		+
		\sum_{\ell=k+1}^{N-1}
		I_{k,\ell}
	\Big)
	+
	\sum_{\ell=k+1}^{N-1}
	\varepsilon_{k,\ell}
	+
	\sum_{\ell=k+1}^{N-1}
	\Delta t_{\ell}
	y_{\ell,\ell}
	+
	\e\Big[
		\Big(
			\sum_{\ell=k+1}^{N-1}
			\Delta t_{\ell}
			\big|
				\overline{\mathscr{Z}}^{\pi}(t_{k},t_{\ell})
				-
				Z^{\pi}(t_{k},t_{\ell})
			\big|
		\Big)^{2}
	\Big]
	\\&\quad\quad\quad
	+
	\e\Big[
		\Big(
			\sum_{\ell=k+1}^{N-1}
			\Delta t_{\ell}
			\big|
				\overline{\mathscr{Z}}^{\pi}(t_{\ell},t_{k})
				-
				Z^{\pi}(t_{\ell},t_{k})
			\big|
		\Big)^{2}
	\Big]
\Big\}.
\end{align*}
This estimate, together with \eqref{Lem_EM_0_0}, prove \eqref{Lem_EM_0_2}.

From \eqref{Lem_EM_0_2} and \eqref{Lem_EM_0_0}, by using the discrete Gronwall-like inequality (see Lemma \ref{Lem_Gron_disc}) with
\begin{align*}
a_{k}
&=
y_{k,k}
+
\sum_{\ell=0}^{N-1}
z_{k,\ell},
~
b_{k}
=
|\pi|
\Big(
	M^{2}
	+
	\sum_{\ell=k+1}^{N-1}
	I_{k,\ell}
\Big)
+
\sum_{\ell=k+1}^{N-1}
\varepsilon_{k,\ell}
+
\e
\Big[
	\Big(
		\sum_{\ell=k+1}^{N-1}
		\Delta t_{\ell}
		\big|
			\overline{\mathscr{Z}}^{\pi}(t_{k},t_{\ell})
			-
			Z^{\pi}(t_{k},t_{\ell})
		\big|
	\Big)^{2}
\Big],
\\
c_{k}
&=0,~
\zeta_{k,\ell}(\omega)
:=
\big|
	\overline{\mathscr{Z}}^{\pi}(t_{k},t_{\ell})
	-
	Z^{\pi}(t_{k},t_{\ell})
\big|
~\text{and}~
(S,\Sigma,\mu)
=
(\Omega,\mathcal{F},\p),
\end{align*}
we have, for any sufficiently large $\gamma>0$ depending only on $L$ and $T$,
\begin{align*}
&\sum_{k=0}^{N-1}
\Gamma_{k}
\Big\{
	y_{k,k}
	+
	\sum_{\ell=0}^{N-1}
	z_{k,\ell}
\Big\}
\\&\leq
C
\sum_{k=0}^{N-1}
\Gamma_{k}
\Big\{
	|\pi|
	\Big(
		M^{2}
		+
		\sum_{\ell=k+1}^{N-1}
		I_{k,\ell}
	\Big)
	+
	\sum_{\ell=k+1}^{N-1}
	\varepsilon_{k,\ell}
	+
	\e
	\Big[
		\Big(
			\sum_{\ell=k+1}^{N-1}
			\Delta t_{\ell}
			\big|
				\overline{\mathscr{Z}}^{\pi}(t_{k},t_{\ell})
				-
				Z^{\pi}(t_{k},t_{\ell})
			\big|
		\Big)^{2}
	\Big]
\Big\},
\end{align*}
and hence
\begin{align}\label{Prop_EM_0_1}
\begin{split}
\sum_{k=0}^{N-1}
\Gamma_{k}
\Big\{
	y_{k,k}
	+
	\sum_{\ell=0}^{N-1}
	z_{k,\ell}
\Big\}
&\leq
C_{\gamma}
\Big\{
	|\pi|
	\Big(
		M^{2}
		+
		\sum_{k=0}^{N-2}
		\Delta t_{k}
		\sum_{\ell=k+1}^{N-1}
		I_{k,\ell}
	\Big)
	+
	\sum_{k=0}^{N-2}
	\Delta t_{k}
	\sum_{\ell=k+1}^{N-1}
	\varepsilon_{k,\ell}
\Big\}
\\&\quad
+
C
\sum_{k=0}^{N-2}
\Gamma_{k}
\e
\Big[
	\Big(
		\sum_{\ell=k+1}^{N-1}
		\Delta t_{\ell}
		\big|
			\overline{\mathscr{Z}}^{\pi}(t_{k},t_{\ell})
			-
			Z^{\pi}(t_{k},t_{\ell})
		\big|
	\Big)^{2}
\Big],
\end{split}
\end{align}
where $\Gamma_{k}:=\int_{t_{k}}^{t_{k+1}} e^{\gamma t} \rd t$, $k=0,\ldots,N-1$.

Let $\gamma>0$ be fixed.
We show that the last term in the right-hand side of \eqref{Prop_EM_0_1} is estimated as follows:
\begin{align}\label{Lem_EM_1_1}
\begin{split}
&
\sum_{k=0}^{N-2}
\Gamma_{k}
\e
\Big[
	\Big(
		\sum_{\ell=k+1}^{N-1}
		\Delta t_{\ell}
		\big|
			\overline{\mathscr{Z}}^{\pi}(t_{k},t_{\ell})
			-
			Z^{\pi}(t_{k},t_{\ell})
		\big|
	\Big)^{2}
\Big]
\\&\leq
\frac{C}{\gamma}
\sum_{k=0}^{N-1}
\Gamma_{k} y_{k,k}
+
C_{\gamma}
\Big\{
	|\pi|
	\Big(
		M^{2}
		+
		\sum_{k=0}^{N-2}
		\Delta t_{k}
		\sum_{\ell=k+1}^{N-1}
		I_{k,\ell}
	\Big)
	+
	\sum_{k=0}^{N-2}
	\Delta t_{k}
	\sum_{\ell=k+1}^{N-1}
	\varepsilon_{k,\ell}
\Big\}.
\end{split}
\end{align}
Note that $\Delta t_{k} e^{\gamma t_{k}} \leq \Gamma_{k} \leq \Delta t_{k} e^{\gamma t_{k+1}}$ for $k=0,\ldots,N-1$.
Also, by defining $\Gamma_{N}:=\Gamma_{N-1}$ and $\Delta t_{N}:=\Delta t_{N-1}$, we have $\frac{\Gamma_{k}}{\Delta t_{k}} \leq \frac{\Gamma_{k+1}}{\Delta t_{k+1}}$ for $k=0,\ldots,N-1$.
We observe that
\begin{align*}
&\sum_{k=0}^{N-2}
\Gamma_{k}
\e
\Big[
	\Big(
		\sum_{\ell=k+1}^{N-1}
		\Delta t_{\ell}
		\big|
			\overline{\mathscr{Z}}^{\pi}(t_{k},t_{\ell})
			-
			Z^{\pi}(t_{k},t_{\ell})
		\big|
	\Big)^{2}
\Big]
\notag\\&
=
\sum_{k=0}^{N-2}
\Gamma_{k}
\e
\Big[
	\Big(
		\sum_{\ell=k+1}^{N-1}
		\Big(
			\int_{t_{\ell}}^{t_{\ell+1}}
				e^{-\frac{\gamma}{2}t}
				e^{\frac{\gamma}{2}t}
			\rd t
		\Big)
		\big|
			\overline{\mathscr{Z}}^{\pi}(t_{k},t_{\ell})
			-
			Z^{\pi}(t_{k},t_{\ell})
		\big|
	\Big)^{2}
\Big]
\notag\\&
\leq
\sum_{k=0}^{N-2}
\Gamma_{k}
\Big(
	\int_{t_{k+1}}^{t_{N}}
		e^{-\gamma t}
	\rd t
\Big)
\sum_{\ell=k+1}^{N-1}
\Gamma_{\ell}
\e
\Big[
	\big|
		\overline{\mathscr{Z}}^{\pi}(t_{k},t_{\ell})
		-
		Z^{\pi}(t_{k},t_{\ell})
	\big|^{2}
\Big]
\notag\\&
\leq
\frac{1}{\gamma}
\sum_{k=0}^{N-2}
\Gamma_{k}
e^{-\gamma t_{k+1}}
\sum_{\ell=k+1}^{N-1}
\Gamma_{\ell}
\e
\Big[
	\big|
		\overline{\mathscr{Z}}^{\pi}(t_{k},t_{\ell})
		-
		Z^{\pi}(t_{k},t_{\ell})
	\big|^{2}
\Big]
\notag\\&\leq
\frac{1}{\gamma}
\sum_{k=0}^{N-2}
\Delta t_{k}
\sum_{\ell=k+1}^{N-1}
\Gamma_{\ell}
\e
\Big[
	\big|
		\overline{\mathscr{Z}}^{\pi}(t_{k},t_{\ell})
		-
		Z^{\pi}(t_{k},t_{\ell})
	\big|^{2}
\Big],
\end{align*}
and hence
\begin{align}\label{Lem_EM_1_2}
\begin{split}
\sum_{k=0}^{N-2}
\Gamma_{k}
\e
\Big[
	\Big(
		\sum_{\ell=k+1}^{N-1}
		\Delta t_{\ell}
		\big|
			\overline{\mathscr{Z}}^{\pi}(t_{k},t_{\ell})
			-
			Z^{\pi}(t_{k},t_{\ell})
		\big|
	\Big)^{2}
\Big]
\leq
\frac{1}{\gamma}
\sum_{k=0}^{N-2}
\Delta t_{k}
\sum_{\ell=k+1}^{N-1}
\frac{\Gamma_{\ell}}{\Delta t_{\ell}}
z_{k,\ell}.
\end{split}
\end{align}
Let $k=0,\ldots,N-2$ be fixed.
We estimate $\sum_{\ell=k+1}^{N-1}\frac{\Gamma_{\ell}}{\Delta t_{\ell}}z_{k,\ell}$.
By \eqref{Lem_EM_0_1}, we have, for each $\ell=k+1,\ldots,N-1$,
\begin{align}
\frac{\Gamma_{\ell}}{\Delta t_{\ell}}
y_{k,\ell}
+
\frac{1}{2}
\frac{\Gamma_{\ell}}{\Delta t_{\ell}}
z_{k,\ell}
&\leq
(1+C \Delta t_{\ell})
\frac{\Gamma_{\ell}}{\Delta t_{\ell}}
y_{k,\ell+1}
+
\frac{1}{2}
\Big\{
	\Gamma_{\ell}
	y_{\ell,\ell}
	+
	\frac{\Gamma_{\ell}}{\Delta t_{k}}
	z_{\ell,k}
	+
	|\pi|
	\Big(
		\Gamma_{\ell}
		M^{2}
		+
		\frac{\Gamma_{\ell}}{\Delta t_{\ell}}
		I_{k,\ell}
	\Big)
	+
	\frac{\Gamma_{\ell}}{\Delta t_{\ell}}
	\varepsilon_{k,\ell}
\Big\}
\notag\\&
\leq
(1+C \Delta t_{\ell})
\frac{\Gamma_{\ell+1}}{\Delta t_{\ell+1}}
y_{k,\ell+1}
+
\frac{1}{2}
\Big\{
	\Gamma_{\ell}
	y_{\ell,\ell}
	+
	\frac{\Gamma_{\ell}}{\Delta t_{k}}
	z_{\ell,k}
	+
	|\pi|
	\Big(
		\Gamma_{\ell}
		M^{2}
		+
		\frac{\Gamma_{\ell}}{\Delta t_{\ell}}
		I_{k,\ell}
	\Big)
	+
	\frac{\Gamma_{\ell}}{\Delta t_{\ell}}
	\varepsilon_{k,\ell}
\Big\}.
\label{Lem_EM_1_3}
\end{align}
Hence we have
\begin{align*}
&
\sum_{\ell=k+1}^{N-1}
\frac{\Gamma_{\ell}}{\Delta t_{\ell}}
y_{k,\ell}
+
\frac{1}{2}
\sum_{\ell=k+1}^{N-1}
\frac{\Gamma_{\ell}}{\Delta t_{\ell}}
z_{k,\ell}
\notag\\&
\leq
(1+C \Delta t_{N-1})
\frac{\Gamma_{N}}{\Delta t_{N}}
y_{k,N}
+
\sum_{\ell=k+2}^{N-1}
\frac{\Gamma_{\ell}}{\Delta t_{\ell}}
y_{k,\ell}
+
C
\sum_{\ell=k+2}^{N-1}
\Delta t_{\ell-1}
\frac{\Gamma_{\ell}}{\Delta t_{\ell}}
y_{k,\ell}
\notag\\&\quad+
\frac{1}{2}
\Big\{
	\sum_{\ell=k+1}^{N-1}
	\Gamma_{\ell}
	y_{\ell,\ell}
	+
	\frac{1}{\Delta  t_{k}}
	\sum_{\ell=k+1}^{N-1}
	\Gamma_{\ell}
	z_{\ell,k}
	+
	\sum_{\ell=k+1}^{N-1}
	|\pi|
	\Big(
		\Gamma_{\ell} M^{2}
		+
		\frac{\Gamma_{\ell}}{\Delta t_{\ell}}
		I_{k,\ell}
	\Big)
	+
	\sum_{\ell=k+1}^{N-1}
	\frac{\Gamma_{\ell}}{\Delta t_{\ell}}
	\varepsilon_{k,\ell}
\Big\}
\\&
\leq
\sum_{\ell=k+1}^{N-1}
\frac{\Gamma_{\ell}}{\Delta t_{\ell}}
y_{k,\ell}
+
C
\max_{\ell=k+1,\ldots, N-1}
\frac{\Gamma_{\ell}}{\Delta t_{\ell}}
y_{k,\ell}
+
C_{\gamma}
\Big\{
	|\pi|
	\Big(
		M^{2}
		+
		\sum_{\ell=k+1}^{N-1}
		I_{k,\ell}
	\Big)
	+
	\sum_{\ell=k+1}^{N-1}
	\varepsilon_{k,\ell}
\Big\}
\\&\quad
+
\frac{1}{2}
\Big\{
	\sum_{\ell=k+1}^{N-1}
	\Gamma_{\ell}
	y_{\ell,\ell}
	+
	\frac{1}{\Delta  t_{k}}
	\sum_{\ell=k+1}^{N-1}
	\Gamma_{\ell}
	z_{\ell,k}
\Big\}.
\end{align*}
This implies that
\begin{align*}%\label{Lem_EM_1_4}
\begin{split}
\sum_{\ell=k+1}^{N-1}
\frac{\Gamma_{\ell}}{\Delta t_{\ell}}
z_{k,\ell}
&\leq
C
\max_{\ell=k+1,\ldots, N-1}
\frac{\Gamma_{\ell}}{\Delta t_{\ell}}
y_{k,\ell}
+
C_{\gamma}
\Big\{
	|\pi|
	\Big(
		M^{2}
		+
		\sum_{\ell=k+1}^{N-1}
		I_{k,\ell}
	\Big)
	+
	\sum_{\ell=k+1}^{N-1}
	\varepsilon_{k,\ell}
\Big\}
\\&\quad
+
\sum_{\ell=k+1}^{N-1}
\Gamma_{\ell}
y_{\ell,\ell}
+
\frac{1}{\Delta t_{k}}
\sum_{\ell=k+1}^{N-1}
\Gamma_{\ell}
z_{\ell,k}.
\end{split}
\end{align*}
Furthermore, by \eqref{Lem_EM_1_3} and the discrete Gronwall inequality (cf.\ Lemma 5.4 in \cite{Zh04}), we have
\begin{align*}
&\max_{\ell=k+1,\ldots, N-1}
\frac{\Gamma_{\ell}}{\Delta t_{\ell}}
y_{k,\ell}
\notag\\&\leq
C\frac{\Gamma_{N}}{\Delta t_{N}}
y_{k,N}
+
C
\sum_{\ell=k+1}^{N-1}
\Big\{
	\Gamma_{\ell}
	y_{\ell,\ell}
	+
	\frac{\Gamma_{\ell}}{\Delta  t_{k}}
	z_{\ell,k}
	+
	|\pi|
	\Big(
		\Gamma_{\ell} M^{2}
		+
		\frac{\Gamma_{\ell}}{\Delta t_{\ell}}
		I_{k,\ell}
	\Big)
	+
	\frac{\Gamma_{\ell}}{\Delta t_{\ell}}
	\varepsilon_{k,\ell}
\Big\}
\notag\\&
\leq
C_{\gamma}
\Big\{
	|\pi|
	\Big(
		M^{2}
		+
		\sum_{\ell=k+1}^{N-1}
		I_{k,\ell}
	\Big)
	+
	\sum_{\ell=k+1}^{N-1}
	\varepsilon_{k,\ell}
\Big\}
+
C
\Big\{
	\sum_{\ell=k+1}^{N-1}
	\Gamma_{\ell} y_{\ell,\ell}
	+
	\frac{1}{\Delta t_{k}}
	\sum_{\ell=k+1}^{N-1}
	\Gamma_{\ell} z_{\ell,k}
\Big\}.
%\label{Lem_EM_1_5}
\end{align*}
Therefore, we get
\begin{align}
\label{Lem_EM_1_6}
\sum_{\ell=k+1}^{N-1}
\frac{\Gamma_{\ell}}{\Delta t_{\ell}}
z_{k,\ell}
\leq
C_{\gamma}
\Big\{
	|\pi|
	\Big(
		M^{2}
		+
		\sum_{\ell=k+1}^{N-1}
		I_{k,\ell}
	\Big)
	+
	\sum_{\ell=k+1}^{N-1}
	\varepsilon_{k,\ell}
\Big\}
+
C
\Big\{
	\sum_{\ell=k+1}^{N-1}
	\Gamma_{\ell} y_{\ell,\ell}
	+
	\frac{1}{\Delta t_{k}}
	\sum_{\ell=k+1}^{N-1}
	\Gamma_{\ell} z_{\ell,k}
\Big\}
\end{align}
for $k=0,\ldots,N-2$.
Thus, by \eqref{Lem_EM_1_2} and \eqref{Lem_EM_1_6}, we have
\begin{align*}
&
\sum_{k=0}^{N-2}
\Gamma_{k}
\e
\Big[
\Big(
	\sum_{\ell=k+1}^{N-1}
	\Delta t_{\ell}
	\big|
		\overline{\mathscr{Z}}^{\pi}(t_{k},t_{\ell})
		-
		Z^{\pi}(t_{k},t_{\ell})
	\big|
\Big)^{2}
\Big]
\\&\leq
C_{\gamma}
\Big\{
	|\pi|
	\Big(
		M^{2}
		+
		\sum_{k=0}^{N-2}
		\Delta t_{k}
		\sum_{\ell=k+1}^{N-1}
		I_{k,\ell}
	\Big)
	+
	\sum_{k=0}^{N-2}
	\Delta t_{k}
	\sum_{\ell=k+1}^{N-1}
	\varepsilon_{k,\ell}
\Big\}
+
\frac{C}{\gamma}
\Big\{
	\sum_{k=0}^{N-2}
	\Delta t_{k}
	\sum_{\ell=k+1}^{N-1}
	\Gamma_{\ell}
	y_{\ell,\ell}
	+
	\sum_{k=0}^{N-2}
	\sum_{\ell=k+1}^{N-1}
	\Gamma_{\ell} z_{\ell,k}
\Big\}.
\end{align*}
Noting that
\begin{align*}
\sum_{k=0}^{N-2}
\Delta t_{k}
\sum_{\ell=k+1}^{N-1}
\Gamma_{\ell}
y_{\ell,\ell}
\leq
T
\sum_{k=0}^{N-1}
\Gamma_{k}
y_{k,k}
\end{align*}
and
\begin{align*}
\sum_{k=0}^{N-2}
\sum_{\ell=k+1}^{N-1}
\Gamma_{\ell} z_{\ell,k}
=
\sum_{k=1}^{N-1}
\Gamma_{k}
\sum_{\ell=0}^{k-1}
z_{k,\ell}
\leq
\sum_{k=0}^{N-1}
\Gamma_{k}
y_{k,k},
\end{align*}
we obtain the desired estimate \eqref{Lem_EM_1_1}.

By \eqref{Prop_EM_0_1} and \eqref{Lem_EM_1_1}, we have
\begin{align*}
\sum_{k=0}^{N-1}
\Gamma_{k}
\Big\{
	y_{k,k}
	+
	\sum_{\ell=0}^{N-1}
	z_{k,\ell}
\Big\}
&\leq
C_{\gamma}
\Big\{
	|\pi|
	\Big(
		M^{2}
		+
		\sum_{k=0}^{N-2}
		\Delta t_{k}
		\sum_{\ell=k+1}^{N-1}
		I_{k,\ell}
	\Big)
	+
	\sum_{k=0}^{N-2}
	\Delta t_{k}
	\sum_{\ell=k+1}^{N-1}
	\varepsilon_{k,\ell}
\Big\}
+
\frac{C}{\gamma}
\sum_{k=0}^{N-1}
\Gamma_{k} y_{k,k}.
\end{align*}
Therefore, by choosing $\gamma>0$ large enough (depending only on $L$ and $T$), we obtain
\begin{align*}
&\sum_{k=0}^{N-1}
\Delta t_{k}
\Big\{
	y_{k,k}
	+
	\sum_{\ell=0}^{N-1}
	z_{k,\ell}
\Big\}
\leq
C
\Big\{
	|\pi|
	\Big(
		M^{2}
		+
		\sum_{k=0}^{N-2}
		\Delta t_{k}
		\sum_{\ell=k+1}^{N-1}
		I_{k,\ell}
	\Big)
	+
	\sum_{k=0}^{N-2}
	\Delta t_{k}
	\sum_{\ell=k+1}^{N-1}
	\varepsilon_{k,\ell}
\Big\}.
\end{align*}
By using the estimate \eqref{Lem_3_1_2}, we have
\begin{align*}
&\sum_{k=0}^{N-2}
\Delta t_{k}
\sum_{\ell=k+1}^{N-1}
I_{k,\ell}
\\&=
\sum_{k=0}^{N-2}
\Delta t_{k}
\sum_{\ell=k+1}^{N-1}
\e\Big[
	\int_{t_{\ell}}^{t_{\ell+1}}
		\big\{
			\big|
				\mathscr{Y}^{\pi}(t_{\ell},s)
			\big|^{2}
			+
			\big|
				\mathscr{Z}^{\pi}(t_{k},s)
			\big|^{2}
			+
			\big|
				\mathcal{I}^{\pi}[\mathscr{Z}^{\pi}(t_{\ell},\cdot)](t_{k})
			\big|^{2}
		\big\}
	\rd s
\Big]
\\&\leq
T
\sum_{k=0}^{N-1}
\Delta t_{k}
\e\Big[
	\sup_{s \in [0,T]}
	\big|
		\mathscr{Y}^{\pi}(t_{k},s)
	\big|^{2}
\Big]
+
\sum_{k=0}^{N-2}
\Delta t_{k}
\e\Big[
	\int_{t_{k+1}}^{T}
		\big|
			\mathscr{Z}^{\pi}(t_{k},s)
		\big|^{2}
	\rd s
\Big]
+
\sum_{k=1}^{N-1}
\Delta t_{k}
\e\Big[
	\int_{0}^{t_{k}}
		\big|
			\mathscr{Z}^{\pi}(t_{k},s)
		\big|^{2}
	\rd s
\Big]
\\&\leq
C
\sum_{k=0}^{N-1}
\Delta t_{k}
\e\Big[
	\big|\Psi(t_{k})\big|^{2}
	+
	\int_{t_{k+1}}^{T}
		\big|G(t_{k},s,0,0,0)\big|^{2}
	\rd s
\Big]
\leq
CM^{2}.
\end{align*}
Furthermore, noting that
\begin{align*}%\label{Lem_EM_0_00}
\sup_{s \in [t_{\ell},t_{\ell+1}]}
\e
\Big[
	\big|
		\mathscr{Y}^{\pi}(t_{\ell},s)
		-
		\mathscr{Y}^{\pi}(t_{\ell},t_{\ell})
	\big|^{2}
\Big]
=
\sup_{s \in [t_{\ell},t_{\ell+1}]}
\e
\Big[
	\Big|
		\int_{t_{\ell}}^{s}
			\mathscr{Z}^{\pi}(t_{\ell},r)
		\rd W(r)
	\Big|^{2}
\Big]
=
\e
\Big[
	\int_{t_{\ell}}^{t_{\ell+1}}
		\big|
			\mathscr{Z}^{\pi}(t_{\ell},s)
		\big|^{2}
	\rd s
\Big]
\end{align*}
and
\begin{align*}
\e
\Big[
	\big|
		\mathcal{I}^{\pi}[\mathscr{Z}^{\pi}(t_{\ell},\cdot)](t_{k})
		-
		\overline{\mathscr{Z}}^{\pi}(t_{\ell},t_{k})
	\big|^{2}
\Big]
&=
\e
\Big[
	\Big|
		\frac{1}{\Delta t_{k}}
		\int_{t_{k}}^{t_{k+1}}
			\big\{
				\mathscr{Z}^{\pi}(t_{\ell},s)
				-
				\overline{\mathscr{Z}}^{\pi}(t_{\ell},t_{k})
			\big\}
		\rd s
	\Big|^{2}
\Big]
\\*&\leq
\frac{1}{\Delta t_{k}}
\e\Big[
	\int_{t_{k}}^{t_{k+1}}
		\big|
			\mathscr{Z}^{\pi}(t_{\ell},s)
			-
			\overline{\mathscr{Z}}^{\pi}(t_{\ell},t_{k})
		\big|^{2}
	\rd s
\Big],
\end{align*}
it holds that
\begin{align*}
&\sum_{k=0}^{N-2}
\Delta t_{k}
\sum_{\ell=k+1}^{N-1}
\varepsilon_{k,\ell}
\\&=
\sum_{k=0}^{N-2}
\Delta t_{k}
\sum_{\ell=k+1}^{N-1}
\e\Big[
	\int_{t_{\ell}}^{t_{\ell+1}}
		\big\{
			\big|
				\mathscr{Y}^{\pi}(t_{\ell},s)
				-
				\mathscr{Y}^{\pi}(t_{\ell},t_{\ell})
			\big|^{2}
			+
			\big|
				\mathscr{Z}^{\pi}(t_{k},s)
				-
				\overline{\mathscr{Z}}^{\pi}(t_{k},t_{\ell})
			\big|^{2}
			\\&\quad\quad\quad\quad\quad\quad\quad\quad\quad\quad\quad\quad\quad
			+
			\big|
				\mathcal{I}^{\pi}[\mathscr{Z}^{\pi}(t_{\ell},\cdot)](t_{k})
				-
				\overline{\mathscr{Z}}^{\pi}(t_{\ell},t_{k})
			\big|^{2}
		\big\}
	\rd s
\Big]
\\&\leq
C
\Big\{
	\sum_{k=0}^{N-1}
	\Delta t_{k}
	\e
	\Big[
		\int_{t_{k}}^{t_{k+1}}
			\big|
				\mathscr{Z}^{\pi}(t_{k},s)
			\big|^{2}
		\rd s
	\Big]
	+
	\sum_{k=0}^{N-2}
	\Delta t_{k}
	\sum_{\ell=k+1}^{N-1}
	\e
	\Big[
		\int_{t_{\ell}}^{t_{\ell+1}}
			\big|
				\mathscr{Z}^{\pi}(t_{k},s)
				-
				\overline{\mathscr{Z}}^{\pi}(t_{k},t_{\ell})
			\big|^{2}
		\rd s
	\Big]
	\\&\quad\quad\quad
	+
	\sum_{k=0}^{N-2}
	\sum_{\ell=k+1}^{N-1}
	\Delta t_{\ell}
	\e\Big[
		\int_{t_{k}}^{t_{k+1}}
			\big|
				\mathscr{Z}^{\pi}(t_{\ell},s)
				-
				\overline{\mathscr{Z}}^{\pi}(t_{\ell},t_{k})
			\big|^{2}
		\rd s
	\Big]
\Big\}
\\&\leq
C
\Big\{
	\sum_{k=0}^{N-1}
	\Delta t_{k}
	\e
	\Big[
		\int_{t_{k}}^{t_{k+1}}
			\big|
				\mathscr{Z}^{\pi}(t_{k},s)
			\big|^{2}
		\rd s
	\Big]
	+
	\sum_{k=0}^{N-1}
	\Delta t_{k}
	\sum_{\ell=0}^{N-1}
	\e
	\Big[
		\int_{t_{\ell}}^{t_{\ell+1}}
			\big|
				\mathscr{Z}^{\pi}(t_{k},s)
				-
				\overline{\mathscr{Z}}^{\pi}(t_{k},t_{\ell})
			\big|^{2}
		\rd s
	\Big]
\Big\}.
\end{align*}
Therefore, we get the assertion.
\end{proof}

Combining Proposition \ref{Prop_BSDE_0} and Lemma \ref{Prop_EM_0}, we obtain an estimate for the $L^{2}$-error between the backward Euler--Maruyama scheme $\{(Y^{\pi}(t_{k},t_{\ell}),Z^{\pi}(t_{k},t_{\ell}))\}_{k,\ell=0}^{N-1}$ and the adapted M-solution $(Y(\cdot),Z(\cdot,\cdot))$ of BSVIE \eqref{BSVIE_II_1} in terms of the $L^{2}$-time regularity of $\mathscr{Z}^{\pi}$.

\begin{prop}\label{Prop_EM_1}
Under $(\mathrm{H}_{\Psi,G})'$, for any $\pi=\{t_{0},t_{1},\ldots,t_{N}\} \in \Pi[0,T]$ with $|\pi| \leq \delta$, it holds that
\begin{align*}
&
\sum_{k=0}^{N-1}
\e\Big[
	\int_{t_{k}}^{t_{k+1}}
		\big|
			Y(t)
			-
			Y^{\pi}(t_{k},t_{k})
		\big|^{2}
	\rd t
\Big]
+
\sum_{k=0}^{N-1}
\sum_{\ell=0}^{N-1}
\e\Big[
	\int_{t_{k}}^{t_{k+1}}
		\int_{t_{\ell}}^{t_{\ell+1}}
			\big|
				Z(t,s)
				-
				Z^{\pi}(t_{k},t_{\ell})
			\big|^{2}
		\rd s
	\rd t
\Big]
\\&\leq
C
\Big\{
	M^{2}
	|\pi|
	+
	\sum_{k=0}^{N-1}
	\Delta t_{k}
	\e
	\Big[
		\int_{t_{k}}^{t_{k+1}}
			\big|
				\mathscr{Z}^{\pi}(t_{k},s)
			\big|^{2}
		\rd s
	\Big]
	+
	\sum_{k=0}^{N-1}
		\Delta t_{k}
		\sum_{\ell=0}^{N-1}
			\e\Big[
				\int_{t_{\ell}}^{t_{\ell+1}}
					\big|
						\mathscr{Z}^{\pi}(t_{k},s)
						-
						\overline{\mathscr{Z}}^{\pi}(t_{k},t_{\ell})
					\big|^{2}
				\rd s
			\Big]
\Big\}.
\end{align*}
\end{prop}
\begin{proof}
We observe that
\begin{align*}
&
\sum_{k=0}^{N-1}
\e\Big[
	\int_{t_{k}}^{t_{k+1}}
		\big|
			Y(t)
			-
			Y^{\pi}(t_{k},t_{k})
		\big|^{2}
	\rd t
\Big]
+
\sum_{k=0}^{N-1}
\sum_{\ell=0}^{N-1}
\e\Big[
	\int_{t_{k}}^{t_{k+1}}
		\int_{t_{\ell}}^{t_{\ell+1}}
			\big|
				Z(t,s)
				-
				Z^{\pi}(t_{k},t_{\ell})
			\big|^{2}
		\rd s
	\rd t
\Big]
\\&\leq
3\Big\{
	\e
	\Big[
		\int_{0}^{T}
			\big|
				Y(t)
				-
				\mathscr{Y}^{\pi}(\tau(t),t)
			\big|^{2}
		\rd t
		+
		\int_{0}^{T}
			\int_{0}^{T}
				\big|
					Z(t,s)
					-
					\mathscr{Z}^{\pi}(\tau(t),s)
				\big|^{2}
			\rd s
		\rd t
	\Big]
	\\&\quad\quad\quad
	+
	\sum_{k=0}^{N-1}
	\e
	\Big[
		\int_{t_{k}}^{t_{k+1}}
			\big|
				\mathscr{Y}^{\pi}(t_{k},t)
				-
				\mathscr{Y}^{\pi}(t_{k},t_{k})
			\big|^{2}
		\rd t
	\Big]
	+
	\sum_{k=0}^{N-1}
	\Delta t_{k}
	\sum_{\ell=0}^{N-1}
	\e
	\Big[
		\int_{t_{\ell}}^{t_{\ell+1}}
			\big|
				\mathscr{Z}^{\pi}(t_{k},s)
				-
				\overline{\mathscr{Z}}^{\pi}(t_{k},t_{\ell})
			\big|^{2}
		\rd s
	\Big]
	\\&\quad\quad\quad
	+
	\sum_{k=0}^{N-1}
	\Delta t_{k}
	\e
	\Big[
		\big|
			\mathscr{Y}^{\pi}(t_{k},t_{k})
			-
			Y^{\pi}(t_{k},t_{k})
		\big|^{2}
	\Big]
	+
	\sum_{k=0}^{N-1}
	\Delta t_{k}
	\sum_{\ell=0}^{N-1}
	\Delta t_{\ell}
	\e
	\Big[
		\big|
			\overline{\mathscr{Z}}^{\pi}(t_{k},t_{\ell})
			-
			Z^{\pi}(t_{k},t_{\ell})
		\big|^{2}
	\Big]
\Big\}.
\end{align*}
Noting that
\begin{align*}
\sup_{t \in [t_{k}, t_{k+1}]}
\e
\Big[
	\big|
		\mathscr{Y}^{\pi}(t_{k},t)
		-
		\mathscr{Y}^{\pi}(t_{k},t_{k})
	\big|^{2}
\Big]
=
\sup_{t \in [t_{k}, t_{k+1}]}
\e
\Big[
	\Big|
		\int_{t_{k}}^{t}
			\mathscr{Z}^{\pi}(t_{k},s)
		\rd W(s)
	\Big|^{2}
\Big]
=
\e
\Big[
	\int_{t_{k}}^{t_{k+1}}
		\big|
			\mathscr{Z}^{\pi}(t_{k},s)
		\big|^{2}
	\rd s
\Big],
\end{align*}
by Proposition \ref{Prop_BSDE_0} (see also Remark \ref{Rem_EM_0}) and Lemma \ref{Prop_EM_0}, we get the assertion.
\end{proof}

%%%%%%%%%%%%%%%
%% Section
%%%%%%%%%%%%%%%

\section{Main results}\label{sec_L2_reg}

\subsection{Statements of our main results}
In this section, we consider the system of Type-\Rnum{2} BSVIE \eqref{Type-2 Markov} and SDE \eqref{SDE} which we rewrite for readers' convenience:
\begin{align}\label{BSVIE_II_0}
\begin{split}
Y(t)
=
\psi(t,X(t),X(T))
+
\int^{T}_{t}
	g
	(
		t,
		s,
		X(t),
		X(s),
		Y(s),
		Z(t,s),
		Z(s,t)
	)
\rd s
-
\int^{T}_{t}
	Z(t,s)
\rd W(s),
~t\in [0,T],
\end{split}
\end{align}
and
\begin{align}\label{SDE_0}
X(t)
&=
x
+
\int_{0}^{t}
	b(s,X(s))
\rd s
+
\int_{0}^{t}
	\sigma(s,X(s))
\rd W(s),
~t \in [0,T].
%~x=(x^{1},\ldots,x^{n})^{\top} \in \real^{n}.
\end{align}

We impose the following assumptions on the coefficients $(\psi,g)$ and $(b,\sigma)$.

\begin{itemize}
\item[$(\mathrm{H}_{\psi,g})$]
The maps $\psi:[0,T] \times \real^{n} \times \real^{n} \to \real^{m}$ and $g:\Delta[0,T] \times \real^{n} \times \real^{n} \times \real^{m} \times \real^{m \times d} \times \real^{m \times d} \to \real^{m}$ are measurable, and there exists a constant $L>0$ such that
\begin{align*}
\int_{0}^{T}
	\bigl|\psi(t,0,0)\bigr|^{2}
\rd t
+
\int_{0}^{T}
	\int_{t}^{T}
		\bigl|
			g(t,s,0,0,0,0,0)
		\bigr|^{2}
	\rd s
\rd t
\leq
L
\end{align*}
and
\begin{align*}
&
\big|
	\psi(t,x_{1},x_{2})
	-
	\psi(t,x_{1}',x_{2}')
\big|
+
\big|
	g(t,s,x_{1},x_{2},y,z_{1},z_{2})
	-
	g(t,s,x_{1}',x_{2}',y',z_{1}',z_{2}')
\big|
\\&\leq
L
\big\{
	|x_{1}-x_{1}'|
	+
	|x_{2}-x_{2}'|
	+
	|y-y'|
	+
	|z_{1}-z_{1}'|
	+
	|z_{2}-z_{2}'|
\big\},
\end{align*}
for any $(t,s) \in \Delta [0,T]$ and $(x_{1},x_{2},y,z_{1},z_{2}),(x_{1}',x_{2}',y',z_{1}',z_{2}') \in \real^{n} \times\real^{n} \times \real^{m} \times \real^{m \times d} \times \real^{m \times d}$.
Furthermore, there exists a continuous and increasing function $\rho_{\psi,g}:[0,\infty) \to [0,\infty)$ with $\rho_{\psi,g}(0)=0$ such that, for any $0 \leq t,t' \leq s\leq T$ and $(x_{1},x_{2},y,z_{1},z_{2}) \in \real^{n} \times\real^{n} \times \real^{m} \times \real^{m \times d} \times \real^{m \times d}$,
\begin{align*}
&\big|
	\psi(t,x_{1},x_{2})
	-
	\psi(t',x_{1},x_{2})
\big|
+
\big|
	g(t,s,x_{1},x_{2},y,z_{1},z_{2})
	-
	g(t',s,x_{1},x_{2},y,z_{1},z_{2})
\big|
\\&\leq
\rho_{\psi,g}(|t-t'|)
\big\{
	1
	+
	|x_{1}|
	+
	|x_{2}|
	+
	|y|
	+
	|z_{1}|
	+
	|z_{2}|
\big\}.
\end{align*}
\end{itemize}

\begin{itemize}
\item[$(\mathrm{H}_{b,\sigma})$]
The maps $b:[0,T] \times \real^{n} \to \real^{n}$ and $\sigma:[0,T] \times \real^{n} \to \real^{n \times d}$ are measurable, and there exists a constant $L>0$ such that, for any $s \in [0,T]$ and $x,x' \in \real^{n}$,
\begin{align*}
\big|
	b(s,0)
\big|
+
\big|
	\sigma(s,0)
\big|
&\leq
L,~
\big|
	b(s,x)
	-
	b(s,x')
\big|
+
\big|
	\sigma(s,x)
	-
	\sigma(s,x')
\big|
\leq
L|x-x'|.
\end{align*}
Furthermore, there exists a continuous and increasing function $\rho_{\sigma}:[0,\infty) \to [0,\infty)$ with $\rho_{\sigma}(0)=0$ such that, for any $s,s' \in [0,T]$ and $x \in \real^{n}$,
\begin{align*}
\big|
	\sigma(s,x)
	-
	\sigma(s',x)
\big|
&\leq
\rho_{\sigma}(|s-s'|)
(1+|x|).
\end{align*}
\end{itemize}

We also consider the following assumptions of $(\psi,g)$ and $(b,\sigma)$ which are slightly stronger than the above ones.

\begin{itemize}
\item[$(\mathrm{H}_{\psi,g})'$]
$(\psi,g)$ satisfies $(\mathrm{H}_{\psi,g})$ with the constant $L$.
Furthermore,
\begin{align*}
&\big|
	\psi(t,x_{1},x_{2})
	-
	\psi(t',x_{1},x_{2})
\big|
+
\big|
	g(t,s,x_{1},x_{2},y,z_{1},z_{2})
	-
	g(t',s',x_{1},x_{2},y,z_{1},z_{2})
\big|
\\&
\leq
L
\big\{
	|t-t'|^{1/2}
	+
	|s-s'|^{1/2}
\big\}
\big\{
	1
	+
	|x_{1}|
	+
	|x_{2}|
	+
	|y|
	+
	|z_{1}|
	+
	|z_{2}|
\big\}
\end{align*}
for any $(t,s), (t',s') \in \Delta[0,T]$ and $(x_{1},x_{2},y,z_{1},z_{2}) \in \real^{n} \times\real^{n} \times \real^{m} \times \real^{m \times d} \times \real^{m \times d}$.
\end{itemize}

\begin{itemize}
\item[$(\mathrm{H}_{b,\sigma})'$]
$(b,\sigma)$ satisfies $(\mathrm{H}_{b,\sigma})$ with the constant $L$.
Furthermore,
\begin{align*}
\big|
	b(s,x)
	-
	b(s',x)
\big|
+
\big|
	\sigma(s,x)
	-
	\sigma(s',x)
\big|
\leq
L|s-s'|^{1/2}
(1+|x|)
\end{align*}
for any $s,s'\in [0,T]$ and $x \in \real^{n}$.
\end{itemize}

For a numerical approximation of SDE \eqref{SDE_0}, we consider the Euler--Maruyama scheme $\{X^{\pi}(t_{k})\}_{k=0}^{N}$, which is defined for each $\pi=\{t_{0},t_{1},\ldots,t_{N}\} \in\Pi[0,T]$ by
\begin{align*}
X^{\pi}(t_{k+1})
=
X^{\pi}(t_{k})
+
b(t_{k},X^{\pi}(t_{k})) \Delta t_{k}
+
\sigma(t_{k},X^{\pi}(t_{k})) \Delta W_{k},~
k=0,\ldots,N-1,
\end{align*}
with the initial condition $X^{\pi}(0)=x$.

The following lemma is well-known, see for example \cite{KP95,Zh17}.

\begin{lemm}\phantomsection\label{Lem_SDE_1}
\begin{itemize}
\item[(i)]
Under $(\mathrm{H}_{b,\sigma})$, there exists a unique strong solution $X(\cdot) \in L^{2}_{\mathbb{F}}(\Omega;C([0,T];\real^{n}))$ of SDE \eqref{SDE_0}.
Furthermore, for any $p \geq 2$, it holds that
\begin{align}
\label{Lem_SDE_1_0}
\e
\Big[
	\sup_{s \in [0,T]}
	\big|
		X(s)
	\big|^{p}
\Big]
&\leq
C_{p}
(1+|x|^{p}),
\\
\label{Lem_SDE_1_1}
\e
\Big[
	\sup_{s'\in [s,t]}
	\big|
		X(s')
		-
		X(s)
	\big|^{p}
\Big]
&\leq
C_{p}
(1+|x|^{p})
|t-s|^{p/2},~\text{for any}~0 \leq s \leq t \leq T.
\end{align}

\item[(ii)]
Under $(\mathrm{H}_{b,\sigma})'$, for any $\pi=\{t_{0},t_{1},\ldots,t_{N}\} \in\Pi[0,T]$ and $p \geq 2$, it holds that
\begin{align*}%\label{EM_X_0}
\max_{k=0,\ldots, N}
\e
\Big[
	\big|
		X(t_{k})
		-
		X^{\pi}(t_{k})
	\big|^{p}
\Big]
\leq
C_{p}(1+|x|^{p})
|\pi|^{p/2}.
\end{align*}
\end{itemize}
\end{lemm}

\begin{rem}
\begin{itemize}
\item[(i)]
We remark that, in $(\mathrm{H}_{\psi,g})$ and $(\mathrm{H}_{b,\sigma})$, the map $s \mapsto \sigma(s,x)$ is assumed to be continuous, while the maps $s \mapsto g(t,s,x,y,z_{1},z_{2})$ and $s \mapsto b(s,x)$ are not.
The continuity condition of $\sigma(s,x)$ with respect to $s$ will be used in the proof of Theorem \ref{Thm_L2_0} below.
Alternatively, in $(\mathrm{H}_{\psi,g})'$ and $(\mathrm{H}_{b,\sigma})'$, we impose the $1/2$-H\"older continuity with respect to both $t$ and  $s$ on $\psi$, $g$, $b$ and $\sigma$, which is used to estimate the error for the Euler--Maruyama scheme.

\item[(ii)]
In $(\mathrm{H}_{b,\sigma})'$, the H\"older continuity with respect to the time parameter $s$ of $b$ and $\sigma$ is not uniform in $x$ unlike previous studies.
However, for $\varphi=b,\sigma$, noting that
\begin{align*}
&\e\Big[
	\sup_{s'\in [s,t]}
	\big|
		\varphi(s',X(s'))
		-
		\varphi(s,X(s))
	\big|^{p}
\Big ]
\\&\leq
C_{p}\Big\{
	|t-s|^{p/2}
	\e\Big[
		1
		+
		\sup_{s \in [0,T]}
		\big|
			X(s)
		\big|^{p}
	\Big]
	+
	\e\Big[
		\sup_{s'\in [s,t]}
		\big|
			X(s')
			-
			X(s)
		\big|^{p}
	\Big]
\Big\}
\\&
\leq
C_{p}(1+|x|^{p})|t-s|^{p/2},
\end{align*}
we can show Lemma \ref{Lem_SDE_1} (ii) by the same way as in the proof of Theorem 5.3.1 in  \cite{Zh17}.

\end{itemize}
\end{rem}

Under $(\mathrm{H}_{\psi,g})$ and $(\mathrm{H}_{b,\sigma})$, the coefficients
\begin{align*}
&
\Psi(t)
=
\psi(t,X(t),X(T)),~t \in [0,T],\\
&
G
(t,s,y,z_{1},z_{2})
=
g
(
	t,
	s,
	X(t),
	X(s),
	y,
	z_{1},
	z_{2}
),~
(t,s,y,z_{1},z_{2}) \in \Delta[0,T] \times \real^{m} \times \real^{m \times d} \times \real^{m \times d},
\end{align*}
satisfy $(\mathrm{H}_{\Psi,G})$ with the constant $M$ replaced by $C\sqrt{1+|x|^{2}}$ and $\rho_{\Psi,G}(t)$ replaced by $\sqrt{t}+\rho_{\psi,g}(t)$.
In this case, BSDE system \eqref{BSDE_sys_0} with $\theta=0$ becomes
\begin{align}\label{BSDE_sys_2}
\begin{cases}
\mathscr{Y}^{\pi}(t_{k},s)
=
\displaystyle
\mathscr{Y}^{\pi}(t_{k},t_{\ell+1})
+
\int_{s}^{t_{\ell+1}}
	g
	(
		t_{k},
		r,
		X(t_{k}),
		X(r),
		\mathscr{Y}^{\pi}(t_{\ell},r),
		\mathscr{Z}^{\pi}(t_{k},r),
		\mathcal{I}^{\pi}[\mathscr{Z}^{\pi}(t_{\ell},\cdot)](t_{k})
	)
	\1_{\{k<\ell\}}
\rd r
\\
\displaystyle \quad\quad\quad\quad\quad\quad
-
\int_{s}^{t_{\ell+1}}
	\mathscr{Z}^{\pi}(t_{k},r)
\rd W(r),~
s \in [t_{\ell},t_{\ell+1}],~k,\ell=0,\ldots, N-1,
\\
\mathscr{Y}^{\pi}(t_{k},t_{N})
=
\psi(t_{k},X(t_{k}),X(t_{N})),~
k=0,\ldots,N-1.
\end{cases}
\end{align}
Furthermore, under $(\mathrm{H}_{\psi,g})'$ and $(\mathrm{H}_{b,\sigma})'$, by defining
\begin{align*}
&\Psi^{\pi}(t_{k})
=
\psi(t_{k},X^{\pi}(t_{k}),X^{\pi}(t_{N})),~
k=0,\ldots,N-1,
\\
&G^{\pi}(t_{k},t_{\ell},y,z_{1},z_{2})
=
g(t_{k},t_{\ell},X^{\pi}(t_{k}),X^{\pi}(t_{\ell}), y, z_{1}, z_{2}),\\
&
\hspace{2cm}
(y,z_{1},z_{2}) \in \real^{m} \times \real^{m \times d} \times \real^{m \times d},~
k=0,\ldots,N-1,~\ell=k,\ldots,N-1,
\end{align*}
for each $\pi=\{t_{0},t_{1},\ldots,t_{N}\} \in \Pi[0,T]$, we see that $(\mathrm{H}_{\Psi,G})'$ holds with $M$ replaced by $C\sqrt{1+|x|^{2}}$.
In this case, the backward Euler--Maruyama scheme \eqref{EM_0} becomes
\begin{align}\label{EM_2}
\begin{cases}
Y^{\pi}(t_{k},t_{\ell})
\displaystyle
=
\e_{t_{\ell}}
\Big[
	Y^{\pi}(t_{k},t_{\ell+1})
\Big]
\\
\displaystyle \quad\quad\quad\quad\quad\quad
+
\Delta t_{\ell}
g
(
	t_{k},
	t_{\ell},
	X^{\pi}(t_{k}),
	X^{\pi}(t_{\ell}),
	Y^{\pi}(t_{\ell},t_{\ell}),
	Z^{\pi}(t_{k},t_{\ell}),
	Z^{\pi}(t_{\ell},t_{k})
)
\1_{\{k<\ell\}},
~k,\ell=0,\ldots,N-1,
\\
Z^{\pi}(t_{k},t_{\ell})
\displaystyle
=
\frac{1}{\Delta t_{\ell}}
\e_{t_{\ell}}
\Big[
	Y^{\pi}(t_{k},t_{\ell+1})
	\Delta W_{\ell}^{\top}
\Big],
~k,\ell=0,\ldots,N-1,
\end{cases}
\end{align}
with $Y^{\pi}(t_{k},t_{N}):=\psi(t_{k},X^{\pi}(t_{k}),X^{\pi}(t_{N}))$, $k=0,\ldots,N-1$.

Now we state our main results.
The first main result is in regard to the estimate of the modulus of the $L^{2}$-time regularity of $(Y(\cdot),Z(\cdot,\cdot))$ (see \eqref{Reg_YZ}).

\begin{theo}\label{Thm_1}
Under $(\mathrm{H}_{\psi,g})$ and $(\mathrm{H}_{b,\sigma})$, for any $\pi=\{t_{0},t_{1},\ldots,t_{N}\} \in \Pi[0,T]$, it holds that
\begin{align}\label{Reg_Y_0}
\mathcal{E}(Y;\pi)
+
\mathcal{E}(Z;\pi)
&\leq
C(1+|x|^{2})
\big\{
	|\pi|
	+
	\rho_{\psi,g}(|\pi|)^{2}
	+
	\rho_{\sigma}(|\pi|)^{2}
\big\}.
\end{align}
\end{theo}

\begin{rem}
Unlike the case of Type-\Rnum{1} BSVIEs, the time regularity of adapted M-solutions of Type-\Rnum{2} BSVIEs is a difficult problem due to the appearance of the term $Z(s,t)$ in the driver.
On one hand, Yong~\cite{Yo08} showed the continuity (in the strong $L^2$-sense) for general Type-\Rnum{2} BSVIE~\eqref{BSVIE_II_1} under technical differentiability assumptions for the coefficients.
On the other hand, \eqref{Reg_Y_0} provides a quantitative estimate for the modulus of the $L^2$-time regularity of the adapted M-solution of Type-\Rnum{2} BSVIE \eqref{BSVIE_II_0} with $X(\cdot)$ being the solution of SDE \eqref{SDE_0}, without smoothness of the coefficients.
This kind of regularity estimate for adapted M-solutions appears for the first time in the literature of BSVIEs.
\end{rem}

The second is in regard to the numerical approximation based on the backward Euler--Maruyama scheme.

\begin{theo}\label{Thm_2}
Under $(\mathrm{H}_{\psi,g})'$ and $(\mathrm{H}_{b,\sigma})'$, there exists a constant $\delta >0$ depending only on $L$ such that, for any $\pi=\{t_{0},t_{1},\ldots,t_{N}\} \in \Pi[0,T]$ with $|\pi| \leq\delta$, it holds that
\begin{align}\label{eq_thm2}
\begin{split}
&
\sum_{k=0}^{N-1}
\e\Big[
	\int_{t_{k}}^{t_{k+1}}
		\big|
			Y(t)
			-
			Y^{\pi}(t_{k},t_{k})
		\big|^{2}
	\rd t
\Big]
+
\sum_{k=0}^{N-1}
\sum_{\ell=0}^{N-1}
\e\Big[
	\int_{t_{k}}^{t_{k+1}}
		\int_{t_{\ell}}^{t_{\ell+1}}
			\big|
				Z(t,s)
				-
				Z^{\pi}(t_{k},t_{\ell})
			\big|^{2}
		\rd s
	\rd t
\Big]
\\&\leq
C(1+|x|^{2})|\pi|.
\end{split}
\end{align}
\end{theo}

\begin{rem}
\begin{itemize}
\item[(i)]
We note that, in \cite{Zh04,WaY18}, a kind of uniformity condition named as the $K$-uniform condition was imposed on the partition $\pi \in \Pi[0,T]$.
On the other hand, Gobet and Makhlouf \cite{GoMa10} considered a numerical scheme for BSDEs with irregular terminal function in terms of a special form of partitions which does not satisfy the $K$-uniform condition.
Compared with the above papers, our results hold true for general partitions.

\item[(ii)]
In the literature of numerical approximations for BSDEs based on the backward Euler--Maruyama scheme, the $L^2$-error for the first component $Y(\cdot)$ of the adapted solution is often considered in the forms $\max_{k=0,\ldots,N-1}\e[|Y(t_k)-y^\pi(t_k)|^2]$ or $\max_{k=0,\ldots,N-1}\e[\sup_{t\in[t_k,t_{k+1}]}|Y(t)-y^\pi(t_k)|^2]$, where $y^\pi$ is the corresponding scheme (see \cite{Zh04,Zh17}).
Also, Wang~\cite{WaY18} considered the same kind of error terms for a Type-\Rnum{1} BSVIE. These are the cases because of the time regularity of $Y(\cdot)$ in the strong $L^2$-sense (see \cite{Ha21} for the time regularity in the strong $L^2$-sense of $Y(\cdot)$ for Type-\Rnum{1} BSVIEs).
However, in the case of Type-\Rnum{2} BSVIEs, the time regularity of the first component $Y(\cdot)$ of the adapted M-solution is a delicate problem, and $Y(t)$ is defined only for a.e.\ $t\in[0,T]$ in general.
For this reason, we guess that the error term of the ``integral form'' as in \eqref{eq_thm2} is reasonable in the case of Type-\Rnum{2} BSVIEs.
\end{itemize}
\end{rem}

In order to prove the above results, we have to estimate the modulus of the $L^{2}$-time regularity of $\mathscr{Z}^{\pi}$, which we study in the next subsection.

\subsection{$L^{2}$-time regularity of $\mathscr{Z}^{\pi}$}\label{sec_L2}

We investigate the $L^2$-time regularity of the martingale integrand $\mathscr{Z}^\pi$ of BSDE system~\eqref{BSDE_sys_2}.
For this purpose, we introduce additional notations which we use throughout this subsection.
For each $d_1\in\n$, $I_{d_1}$ denotes the identity matrix in $\real^{d_1\times d_1}$.
For each $\xi\in\real^{d_1\times d_2}$ with $d_1,d_2\in\n$, $\xi^{(\alpha)}\in\real^{d_1}$ denotes the $\alpha$-th column vector of $\xi$ for each $\alpha=1,\ldots,d_2$, that is, $\xi=(\xi^{(1)},\ldots,\xi^{(d_2)})$.
For each $d_1,d_2,d_3\in\n$, we denote by $\real^{(d_1\times d_2)\times d_3}$ the space of all elements $\xi=(\xi^{(1)},\ldots,\xi^{(d_3)})$ where $\xi^{(\alpha)}\in\real^{d_1\times d_2}$ for each $\alpha=1,\ldots,d_3$, which is endowed with the norm $|\xi|:=(\sum^{d_3}_{\alpha=1}|\xi^{(\alpha)}|^2)^{1/2}$.
For each $C^1$ function $\varphi:\real^{d_1}\to\real^{d_2}$ with $d_1,d_2\in\n$, $\partial_\xi\varphi$ denotes the derivative of $\varphi$ with respect to $\xi\in\real^{d_1}$, which takes values in $\real^{d_2\times d_1}$.

In the following, we use the Malliavin calculus technique for SDEs and BSDEs. For notations and fundamental results, we refer the readers to \cite{Nu06,PaPe92}.
Specifically, the operator $D$ denotes the Malliavin derivative, and $\mathbb{D}^{1,2}$ denotes the domain of $D$ in $L^2_{\mathcal{F}_T}(\Omega;\real)$.
For each $\xi\in\mathbb{D}^{1,2}$, $D\xi=(D_\theta\xi)_{\theta\in[0,T]}$ can be seen as an element of $L^2_{\mathcal{F}_T}(0,T;\real^d)$.
Also, for each $\xi\in\mathbb{D}^{1,2}$ and $j=1,\ldots,d$, $D^j\xi=(D^j_\theta\xi)_{\theta\in[0,T]}\in L^2_{\mathcal{F}_T}(0,T;\real)$ denotes the $j$-th component of $D\xi$. For each $d_1,d_2\in\n$, we denote by $\mathbb{D}^{1,2}(\real^{d_1\times d_2})$ the space of all $\xi\in L^2_{\mathcal{F}_T}(\Omega;\real^{d_1\times d_2})$ such that each component of $\xi$ is in $\mathbb{D}^{1,2}$, and $D\xi$ and $D^j\xi$, $j=1,\ldots,d$, are understood by the component-wise manner.

Since the following lemma is standard and can be found in \cite{Nu06,PaPe92}, we omit the proof.

\begin{lemm}\label{Lem_SDE_0}
Suppose that $(\mathrm{H}_{b,\sigma})$ holds, and assume that $b$ and $\sigma$ are $C^{1}$ in $x \in \real^{n}$.
Then $X(s)$ is in $\mathbb{D}^{1,2}(\real^{n})$ for any $s \in [0,T]$.
Furthermore, for each $j=1,\ldots,d$, there exists a version of $\{D_{\theta}^{j}X(s)\,|\,(\theta,s) \in [0,T]^{2}\}$ such that
\begin{align}\label{Lem_SDE_0_1}
D_{\theta}^{j}X(s)
=
\nabla X(s)
(\nabla X(\theta))^{-1}
\sigma^{(j)}(\theta,X(\theta))
\1_{[\theta,T]}(s),~
(\theta,s) \in [0,T]^{2},
\end{align}
where $\nabla X(\cdot) \in L^{2}_{\mathbb{F}}(\Omega;C([0,T];\real^{n \times n}))$ is the solution of the following variational equation:
\begin{align*}%\label{Lem_SDE_0_2}
\nabla X(t)
=
I_{n}
+
\int^{t}_{0}
	\partial_{x}b(s,X(s))
	\nabla X(s)
\rd s
+
\sum_{\alpha=1}^{d}
\int^{t}_{0}
	\partial_{x}\sigma^{(\alpha)}(s,X(s))
	\nabla X(s)
\rd W^{\alpha}(s),
~t\in[0,T].
\end{align*}
Moreover, for any $p \geq 2$, the following estimates hold:
\begin{align}
\label{Lem_SDE_0_3}
&
\e\Big[
	\sup_{s \in [0,T]}
	\big|
		\nabla X(s)
	\big|^{p}
\Big]
+
\e\Big[
	\sup_{s \in [0,T]}
	\big|
		(\nabla X(s))^{-1}
	\big|^{p}
\Big]
\leq
C_{p},\\
\label{Lem_SDE_0_4}
&
\e\Big[
	\sup_{s'\in [s,t]}
	\big|
		(\nabla X(s'))^{-1}
		-
		(\nabla X(s))^{-1}
	\big|^{p}
\Big]
\leq
C_{p}
|t-s|^{p/2},~\text{for any}~0 \leq s \leq t \leq T.
\end{align}
\end{lemm}

Fix $\pi=\{t_{0},t_{1},\ldots,t_{N}\} \in \Pi[0,T]$.
When $\psi$ and $g$ in $(\mathrm{H}_{\psi,g})$ are $C^{1}$ in $(x_{1},x_{2},y,z_{1},z_{2}) \in \real^{n} \times\real^{n} \times \real^{m} \times \real^{m\times d} \times \real^{m\times d}$, we define
\begin{align*}
\psi_{\xi}(t_{k})
&:=
\partial_{\xi} \psi(t_{k},X(t_{k}),X(T)),
\\g_{\xi}(t_{k},s)
&:=
\partial_{\xi}
g
(
	t_{k},
	s,
	X(t_{k}),
	X(s),
	\mathscr{Y}^{\pi}(\tau(s),s),
	\mathscr{Z}^{\pi}(t_{k},s),
	\mathcal{I}^{\pi}[\mathscr{Z}^{\pi}(\tau(s),\cdot)](t_{k})
),
\end{align*}
for $\xi=x_{1},x_{2},y,z_{1}^{(\alpha)},z_{2}^{(\alpha)}$, $
\alpha=1,\ldots,d$.
By the Lipschitz continuity of $\psi$ and $g$, the above processes are bounded by $L$.
We introduce the following \emph{variational BSDE system} with parameter $\theta \in [0,T]$:
\begin{align}\label{Prop_L2_2}
\begin{cases}
\nabla_{\theta}\mathscr{Y}^{\pi}(t_{k},s)
\displaystyle
=
\nabla_{\theta}\mathscr{Y}^{\pi}(t_{k},t_{\ell+1})
\\\displaystyle\quad\quad\quad\quad\quad\quad\quad
+
\int_{s}^{t_{\ell+1}}
	\Big\{
		g_{x_{1}}(t_{k},r)
		\nabla X(t_{k})
		\1_{\{\theta \leq t_{k}\}}
		+
		g_{x_{2}}(t_{k},r)
		\nabla X(r)
		+
		g_{y}(t_{k},r)
		\nabla_{\theta} \mathscr{Y}^{\pi}(t_{\ell},r)
		\\
		\displaystyle \quad\quad\quad\quad\quad\quad\quad\quad\quad\quad\quad\quad
		+
		\sum_{\alpha=1}^{d}
		g_{z_{1}^{(\alpha)}}(t_{k},r)
		\nabla_{\theta} \mathscr{Z}^{\pi,(\alpha)}(t_{k},r)
		\\
		\displaystyle \quad\quad\quad\quad\quad\quad\quad\quad\quad\quad\quad\quad
		+
		\sum_{\alpha=1}^{d}
		g_{z_{2}^{(\alpha)}}(t_{k},r)
		\mathcal{I}^{\pi,\theta}[\nabla_{\theta} \mathscr{Z}^{\pi,(\alpha)}(t_{\ell},\cdot)](t_{k})
	\Big\}
	\1_{\{k<\ell\}}
\rd r
\\
\displaystyle \quad\quad\quad\quad\quad\quad\quad
-
\int_{s}^{t_{\ell+1}}
	\nabla_{\theta} \mathscr{Z}^{\pi}(t_{k},r)
\rd W(r),
~s \in [t_{\ell}, t_{\ell+1}],
~k,\ell=0,\ldots,N-1,
\\
\nabla_{\theta}\mathscr{Y}^{\pi}(t_{k},t_{N})
\displaystyle =
\psi_{x_{1}}(t_{k})
\nabla X(t_{k})
\1_{\{\theta \leq t_{k}\}}
+
\psi_{x_{2}}(t_{k})
\nabla X(t_{N}),
~k=0,\ldots,N-1,
\end{cases}
\end{align}
where $\nabla_{\theta} \mathscr{Z}^{\pi}(t_{k},s)=(\nabla_{\theta} \mathscr{Z}^{\pi,(1)}(t_{k},s),\ldots,\nabla_{\theta} \mathscr{Z}^{\pi,(d)}(t_{k},s))$ which takes values in $\real^{(m \times n) \times d}$.
By Lemma \ref{Lem_BSDE_sys_0}, for each $\theta \in [0,T]$, there exists a unique solution
\begin{align*}
\{
	(\nabla_{\theta}\mathscr{Y}^{\pi}(t_{k},\cdot),
	\nabla_{\theta}\mathscr{Z}^{\pi}(t_{k},\cdot))
\}_{k=0}^{N-1}
\in
\big(
	L^{2}_{\mathbb{F}}(\Omega, C([0,T];\real^{m \times n}))
	\times
	L^{2}_{\mathbb{F}}(0,T;\real^{(m \times n) \times d})
\big)^{N}
\end{align*}
of variational BSDE system \eqref{Prop_L2_2}.

Now we provide an expression of $\mathscr{Z}^{\pi}$ in terms of the solution of variational BSDE system \eqref{Prop_L2_2}.

\begin{prop}\label{Prop_L2_0}
Suppose that $(\mathrm{H}_{\psi,g})$ and $(\mathrm{H}_{b,\sigma})$ hold.
Assume that $b$, $\sigma$ are $C^{1}$ in $x \in \real^{n}$ and $\psi, g$ are $C^{1}$ in $(x_{1},x_{2},y,z_{1},z_{2}) \in \real^{n} \times\real^{n} \times \real^{m} \times \real^{m \times d} \times \real^{m \times d}$.
Then for any $\pi=\{t_{0},t_{1},\ldots,t_{N}\} \in \Pi[0,T]$, $k=0,\ldots,N-1$ and $s \in [0,T]$, $\mathscr{Y}^{\pi}(t_{k},s)$ is in $\mathbb{D}^{1,2}(\real^{m})$ and $\mathscr{Z}^{\pi}(t_{k},s)$ is in $\mathbb{D}^{1,2}(\real^{m \times d})$.
Furthermore, for each $j=1,\ldots,d$, there exists a version of $\{(D_{\theta}^{j}\mathscr{Y}^{\pi}(t_{k},s),D_{\theta}^{j}\mathscr{Z}^{\pi}(t_{k},s))\,|\,0 \leq \theta \leq s \leq T,~k=0,\ldots,N-1\}$ such that
\begin{align}\label{Prop_L2_1}
\begin{cases}
D_{\theta}^{j}\mathscr{Y}^{\pi}(t_{k},s)
\displaystyle
=
\nabla_{\theta}\mathscr{Y}^{\pi}(t_{k},s)
(\nabla X(\theta))^{-1}
\sigma^{(j)}(\theta,X(\theta)),
\\
D_{\theta}^{j}\mathscr{Z}^{\pi,(\alpha)}(t_{k},s)
\displaystyle
=
\nabla_{\theta}\mathscr{Z}^{\pi,(\alpha)}(t_{k},s)
(\nabla X(\theta))^{-1}
\sigma^{(j)}(\theta,X(\theta)),~
\alpha=1,\ldots,d,
\end{cases}
\end{align}
for $0\leq \theta \leq s \leq T$ and $k=0,\ldots,N-1$.
Moreover, it holds that
\begin{align}\label{Prop_L2_3}
\mathscr{Z}^{\pi}(t_{k},s)
=
\nabla_{s}\mathscr{Y}^{\pi}(t_{k},s)
(\nabla X(s))^{-1}
\sigma(s,X(s)),
~\text{a.e.}~s \in [0,T],
~k=0,\ldots,N-1.
\end{align}
\end{prop}
\begin{proof}
For simplicity of notation, we suppose $d=1$, that is, $W(\cdot)$ is a one-dimensional Brownian motion.

We show that the following property holds for each $k=0,\ldots,N-1$, backward inductively:
\begin{itemize}
\item[$(\mathrm{P}_{k})$]
For any $s \in [0,T]$, $\mathscr{Y}^{\pi}(t_{k},s)$ and $\mathscr{Z}^{\pi}(t_{k},s)$ are in $\mathbb{D}^{1,2}(\real^{m})$, and it holds that
\begin{align*}
\begin{cases}
D_{\theta}\mathscr{Y}^{\pi}(t_{k},s)
\displaystyle
=
\nabla_{\theta}\mathscr{Y}^{\pi}(t_{k},s)
(\nabla X(\theta))^{-1}
\sigma(\theta,X(\theta)),
\\
D_{\theta}\mathscr{Z}^{\pi}(t_{k},s)
\displaystyle
=
\nabla_{\theta}\mathscr{Z}^{\pi}(t_{k},s)
(\nabla X(\theta))^{-1}
\sigma(\theta,X(\theta)),
\end{cases}
~0\leq \theta \leq s \leq T.
\end{align*}
\end{itemize}

For $k=N-1$, note that $(\mathscr{Y}^{\pi}(t_{N-1},\cdot),\mathscr{Z}^{\pi}(t_{N-1},\cdot))$ is the adapted solution of the following BSDE:
\begin{align*}
\mathscr{Y}^{\pi}(t_{N-1},s)
=
\psi(t_{N-1},X(t_{N-1}),X(T))
-
\int_{s}^{T}
	\mathscr{Z}^{\pi}(t_{N-1},r)
\rd W(r),
~s \in [0,T].
\end{align*}
Thus, by the well-known result on the Malliavin calculus for BSDEs (cf.\ \cite{PaPe92}), we see that $\mathscr{Y}^{\pi}(t_{N-1},s)$ and $\mathscr{Z}^{\pi}(t_{N-1},s)$ are in $\mathbb{D}^{1,2}(\real^{m})$ for any $s \in [0,T]$.
Moreover, for any $\theta \in [0,T]$, $(D_{\theta}\mathscr{Y}^{\pi}(t_{N-1},\cdot), D_{\theta}\mathscr{Z}^{\pi}(t_{N-1},\cdot))$ is the adapted solution of the following BSDE:
\begin{align*}
D_{\theta}\mathscr{Y}^{\pi}(t_{N-1},s)
=
D_{\theta}
\psi(t_{N-1},X(t_{N-1}),X(T))
-
\int_{s}^{T}
	D_{\theta}\mathscr{Z}^{\pi}(t_{N-1},r)
\rd W(r),
~s \in [\theta, T].
\end{align*}
By using the chain rule (cf.\ Proposition 1.2.3 in \cite{Nu06}) and \eqref{Lem_SDE_0_1}, we see that
\begin{align*}
D_{\theta}
\psi(t_{N-1},X(t_{N-1}),X(T))
&=
\psi_{x_{1}}(t_{N-1})
D_{\theta}X(t_{N-1})
+
\psi_{x_{2}}(t_{N-1})
D_{\theta}X(T)
\\&=
\big\{
	\psi_{x_{1}}(t_{N-1})
	\nabla X(t_{N-1})\1_{\{\theta \leq t_{N-1}\}}
	+
	\psi_{x_{2}}(t_{N-1})
	\nabla X(T)
\big\}
(\nabla X(\theta))^{-1}
\sigma(\theta,X(\theta)).
\end{align*}
On the other hand, by \eqref{Prop_L2_2},
\begin{align*}
\nabla_{\theta}\mathscr{Y}^{\pi}(t_{N-1},s)
&=
\psi_{x_{1}}(t_{N-1})
\nabla X(t_{N-1})
\1_{\{\theta \leq t_{N-1}\}}
+
\psi_{x_{2}}(t_{N-1})
\nabla X(T)
\\&\quad
-
\int_{s}^{T}
	\nabla_{\theta} \mathscr{Z}^{\pi}(t_{N-1},r)
\rd W(r),
~s \in [0,T].
\end{align*}
Hence, by the uniqueness of the adapted solution of the BSDE, we see that,
\begin{align*}
\begin{cases}
D_{\theta}\mathscr{Y}^{\pi}(t_{N-1},s)
=
\nabla_{\theta}\mathscr{Y}^{\pi}(t_{N-1},s)
(\nabla X(\theta))^{-1}
\sigma(\theta,X(\theta)),
\\
D_{\theta}\mathscr{Z}^{\pi}(t_{N-1},s)
=
\nabla_{\theta}\mathscr{Z}^{\pi}(t_{N-1},s)
(\nabla X(\theta))^{-1}
\sigma(\theta,X(\theta)),
\end{cases}
~s \in [\theta,T].
\end{align*}
Thus $(\mathrm{P}_{N-1})$ holds.

Next, for a fixed $k' \in \{0,\ldots,N-2\}$, we assume that $(\mathrm{P}_{k})$ holds for any $k=k'+1,\ldots,N-1$.
We observe that $(\mathscr{Y}^{\pi}(t_{k'},\cdot),\mathscr{Z}^{\pi}(t_{k'},\cdot))$ is the adapted solution of the BSDE
\begin{align*}
\mathscr{Y}^{\pi}(t_{k'},s)
&=
\psi(t_{k'},X(t_{k'}),X(T))
\\&\quad
+
\int_{s}^{T}
	g(
		t_{k'},
		r,
		X(t_{k'}),
		X(r),
		\mathscr{Y}^{\pi}(\tau(r),r),
		\mathscr{Z}^{\pi}(t_{k'},r),
		\mathcal{I}^{\pi}[\mathscr{Z}^{\pi}(\tau(r),\cdot)](t_{k'})
	)
	\1_{[t_{k'+1},T)}(r)
\rd r
\\&\quad
-
\int_{s}^{T}
	\mathscr{Z}^{\pi}(t_{k'},r)
\rd W(r),~s \in [0,T].
\end{align*}
By the assumption of the induction, for any $r \in [t_{k'+1},T)$, $\mathscr{Y}^{\pi}(\tau(r),r)$ and $ \mathcal{I}^{\pi}[\mathscr{Z}^{\pi}(\tau(r),\cdot)](t_{k'})$ are in $\mathbb{D}^{1,2}(\real^{m})$.
Thus, by the well-known result on the Malliavin calculus for BSDEs (cf.\ \cite{PaPe92}), $\mathscr{Y}^{\pi}(t_{k'},s)$ and $\mathscr{Z}^{\pi}(t_{k'},s)$ are in $\mathbb{D}^{1,2}(\real^{m})$ for any $s \in [0,T]$.
Also, noting the chain rule, for each $\theta \in [0,T]$, we see that the pair $(D_{\theta} \mathscr{Y}^{\pi}(t_{k'},\cdot),D_{\theta} \mathscr{Z}^{\pi}(t_{k'},\cdot))$ solves the following BSDE:
\begin{align*}
D_{\theta}\mathscr{Y}^{\pi}(t_{k'},s)
&=
\psi_{x_{1}}(t_{k'})
D_{\theta}X(t_{k'})
+
\psi_{x_{2}}(t_{k'})
D_{\theta}X(T)
\\&\quad
+
\int_{s}^{T}
	\big\{
		g_{x_{1}}(t_{k'},r)
		D_{\theta} X(t_{k'})
		+
		g_{x_{2}}(t_{k'},r)
		D_{\theta} X(r)
		+
		g_{y}(t_{k'},r)
		D_{\theta} \mathscr{Y}^{\pi}(\tau(r),r)
		\\&\quad\quad\quad\quad\quad
		+
		g_{z_{1}}(t_{k'},r)
		D_{\theta} \mathscr{Z}^{\pi}(t_{k'},r)
		+
		g_{z_{2}}(t_{k'},r)
		D_{\theta}
		\mathcal{I}^{\pi}[\mathscr{Z}^{\pi}(\tau(r),\cdot)](t_{k'})
	\big\}
	\1_{[t_{k'+1},T)}(r)
\rd r
\\&\quad
-
\int_{s}^{T}
	D_{\theta} \mathscr{Z}^{\pi}(t_{k'},r)
\rd W(r),
~s \in [\theta,T].
\end{align*}
On the other hand, by \eqref{Prop_L2_2}, $(\nabla_{\theta}\mathscr{Y}^{\pi}(t_{k'},\cdot), \nabla_{\theta}\mathscr{Z}^{\pi}(t_{k'},\cdot))$ is the adapted solution of the BSDE
\begin{align*}
\nabla_{\theta}\mathscr{Y}^{\pi}(t_{k'},s)
&=
\psi_{x_{1}}(t_{k'})
\nabla X(t_{k'})
\1_{\{\theta \leq t_{k'}\}}
+
\psi_{x_{2}}(t_{k'})
\nabla X(T)
\\*&\quad
+
\int_{s}^{T}
	\big\{
		g_{x_{1}}(t_{k'},r)
		\nabla X(t_{k'})
		\1_{\{\theta \leq t_{k'}\}}
		+
		g_{x_{2}}(t_{k'},r)
		\nabla X(r)
		+
		g_{y}(t_{k'},r)
		\nabla_{\theta} \mathscr{Y}^{\pi}(\tau(r),r)
		\\*&\quad\quad\quad\quad\quad
		+
		g_{z_{1}}(t_{k'},r)
		\nabla_{\theta} \mathscr{Z}^{\pi}(t_{k'},r)
		+
		g_{z_{2}}(t_{k'},r)
		\mathcal{I}^{\pi,\theta}[\nabla_{\theta} \mathscr{Z}^{\pi}(\tau(r),\cdot)](t_{k'})
	\big\}
	\1_{[t_{k'+1},T)}(r)
\rd r
\\*&\quad
-
\int_{s}^{T}
	\nabla_{\theta} \mathscr{Z}^{\pi}(t_{k'},r)
\rd W(r),
~s \in [0,T].
\end{align*}
By the assumption of the induction, for any $r \in [t_{k'+1} \vee \theta, T)$, we have
\begin{align*}%\label{Prop_L2_6}
D_{\theta}\mathscr{Y}^{\pi}(\tau(r),r)
&=
\nabla_{\theta}\mathscr{Y}^{\pi}(\tau(r),r)
(\nabla X(\theta))^{-1}
\sigma(\theta,X(\theta))
\end{align*}
and, noting that $\mathscr{Z}^{\pi}(\tau(r),\cdot)$ is adapted,
\begin{align*}%\label{Prop_L2_5}
D_{\theta}
\mathcal{I}^{\pi}[\mathscr{Z}^{\pi}(\tau(r),\cdot)](t_{k'})
&=
\frac{1}{\Delta t_{k'}}
D_{\theta}
\int_{t_{k'}}^{t_{k'+1}}
	\mathscr{Z}^{\pi}(\tau(r),s)
\rd s
=
\frac{1}{\Delta t_{k'}}
\int_{t_{k'}\vee \theta}^{t_{k'+1} \vee \theta}
	D_{\theta}
	\mathscr{Z}^{\pi}(\tau(r),s)
\rd s
\notag\\&=
\frac{1}{\Delta t_{k'}}
\int_{t_{k'}\vee \theta}^{t_{k'+1} \vee \theta}
	\nabla_{\theta}\mathscr{Z}^{\pi}(\tau(r),s)
	(\nabla X(\theta))^{-1}
	\sigma(\theta,X(\theta))
\rd s
\notag\\&=
\mathcal{I}^{\pi,\theta}[\nabla_{\theta} \mathscr{Z}^{\pi}(\tau(r),\cdot)](t_{k'})
(\nabla X(\theta))^{-1}
\sigma(\theta,X(\theta)).
%\notag\\&=
%\nabla_{t_{j}} \overline{\mathscr{Z}}^{\pi}(\tau(r),t_{k'})
%(\nabla X(\theta))^{-1}
%\sigma(\theta,X(\theta))
%\1_{\{j<k'\}}.
\end{align*}
From the above, together with \eqref{Lem_SDE_0_1}, by the uniqueness of the adapted solution of the BSDE, we have
\begin{align*}
\begin{cases}
D_{\theta}\mathscr{Y}^{\pi}(t_{k'},s)
=
\nabla_{\theta}\mathscr{Y}^{\pi}(t_{k'},s)
(\nabla X(\theta))^{-1}
\sigma(\theta,X(\theta)),
\\
D_{\theta}\mathscr{Z}^{\pi}(t_{k'},s)
=
\nabla_{\theta}\mathscr{Z}^{\pi}(t_{k'},s)
(\nabla X(\theta))^{-1}
\sigma(\theta,X(\theta)),
\end{cases}
~s \in [\theta,T].
\end{align*}
Thus $(\mathrm{P}_{k'})$ holds.
By the backward induction, we see that $(\mathrm{P}_{k})$ holds for every $k=0,\ldots,N-1$.

It remains to prove \eqref{Prop_L2_3}.
Let $k,\ell=0,\ldots,N-1$ and $s \in (t_{\ell},t_{\ell+1}]$ be fixed.
We observe that
\begin{align*}
\mathscr{Y}^{\pi}(t_{k},t_{\ell})
&=
\mathscr{Y}^{\pi}(t_{k},t_{\ell+1})
+
\int_{t_{\ell}}^{t_{\ell+1}}
	g(
		t_{k},
		r,
		X(t_{k}),
		X(r),
		\mathscr{Y}^{\pi}(t_{\ell},r),
		\mathscr{Z}^{\pi}(t_{k},r),
		\mathcal{I}^{\pi}[\mathscr{Z}^{\pi}(t_{\ell},\cdot)](t_{k})
	)
	\1_{\{k<\ell\}}
\rd r
\\*&\quad
-
\int_{t_{\ell}}^{t_{\ell+1}}
	\mathscr{Z}^{\pi}(t_{k},r)
\rd W(r).
\end{align*}
We apply $D_{s}$ on both side of the above equality.
By the chain rule and Proposition 1.3.8 in \cite{Nu06}, we have
\begin{align*}
0&=
D_{s}\mathscr{Y}^{\pi}(t_{k},t_{\ell+1})
\\&\quad
+
\int_{s}^{t_{\ell+1}}
	\big\{
		g_{x_{1}}(t_{k},r)
		D_{s}X(t_{k})
		+
		g_{x_{2}}(t_{k},r)
		D_{s}X(r)
		+
		g_{y}(t_{k},r)
		D_{s}\mathscr{Y}^{\pi}(t_{\ell},r)
		\\&\quad\quad\quad\quad\quad\quad
		+
		g_{z_{1}}(t_{k},r)
		D_{s}\mathscr{Z}^{\pi}(t_{k},r)
		+
		g_{z_{2}}(t_{k},r)
		D_{s}
		\mathcal{I}^{\pi}[\mathscr{Z}^{\pi}(t_{\ell},\cdot)](t_{k})
	\big\}
	\1_{\{k<\ell\}}
\rd r
\\&\quad
-
\mathscr{Z}^{\pi}(t_{k},s)
-
\int_{s}^{t_{\ell+1}}
	D_{s}
	\mathscr{Z}^{\pi}(t_{k},r)
\rd W(r)
\\
&=
D_{s}\mathscr{Y}^{\pi}(t_{k},t_{\ell+1})
+
\int_{s}^{t_{\ell+1}}
	\big\{
		g_{x_{2}}(t_{k},r)
		D_{s}X(r)
		+
		g_{y}(t_{k},r)
		D_{s}\mathscr{Y}^{\pi}(t_{\ell},r)
		+
		g_{z_{1}}(t_{k},r)
		D_{s}\mathscr{Z}^{\pi}(t_{k},r)
	\big\}
	\1_{\{k<\ell\}}
\rd r
\\&\quad
-
\mathscr{Z}^{\pi}(t_{k},s)
-
\int_{s}^{t_{\ell+1}}
	D_{s}
	\mathscr{Z}^{\pi}(t_{k},r)
\rd W(r).
\end{align*}
Thus, by using \eqref{Lem_SDE_0_1}, \eqref{Prop_L2_1} and the equation \eqref{Prop_L2_2}, we obtain
\begin{align*}
\mathscr{Z}^{\pi}(t_{k},s)
&=
\Big\{
	\nabla_{s}\mathscr{Y}^{\pi}(t_{k},t_{\ell+1})
	\\*&\quad\quad
	+
	\int_{s}^{t_{\ell+1}}
		\big\{
			g_{x_{2}}(t_{k},r)
			\nabla X(r)
			+
			g_{y}(t_{k},r)
			\nabla_{s}\mathscr{Y}^{\pi}(t_{\ell},r)
			+
			g_{z_{1}}(t_{k},r)
			\nabla_{s}\mathscr{Z}^{\pi}(t_{k},r)
		\big\}
		\1_{\{k<\ell\}}
	\rd r
	\\*&\quad\quad
	-
	\int_{s}^{t_{\ell+1}}
		\nabla_{s}
		\mathscr{Z}^{\pi}(t_{k},r)
	\rd W(r)
\Big\}
(\nabla X(s))^{-1}
\sigma(s,X(s))
\\*&=
\nabla_{s}\mathscr{Y}^{\pi}(t_{k},s)
(\nabla X(s))^{-1}
\sigma(s,X(s)).
\end{align*}
Hence \eqref{Prop_L2_3} holds, and we complete the proof.
\end{proof}

\begin{rem}
Variational BSDE system~\eqref{Prop_L2_2} is a continuum of BSDE systems  parametrized by $\theta\in[0,T]$ (where each BSDE system consists of a finite number of BSDEs).
On one hand, due to the appearance of the operator $\mathcal{I}^{\pi,\theta}[\cdot]$, the variational BSDE system of a true Type-\Rnum{2} BSVIE (where $g_{z^{(\alpha)}_2}(t_k,r)\neq0$) cannot be reduced to a finite number of systems.
On the other hand, in the case of Type-\Rnum{1} BSVIEs (where $g(t,s,x_{1},x_{2},y,z_{1},z_{2})$ does not depend on $z_{2}$, and hence  $g_{z^{(\alpha)}_2}(t_k,r)=0$), it is reduced to a finite number of BSDE systems.
Indeed, in this case, the dependency of the system on the parameter $\theta$ stems only from the indicator functions $\1_{\{\theta\leq t_k\}}$.
Since $\theta\leq t_k$ if and only if $i<k$ for each $k=0,\ldots,N-1$ and $\theta\in(t_i,t_{i+1}]$ with $i=0,\ldots,N-1$, by the uniqueness of the solution of the BSDE system (see Lemma~\ref{Lem_BSDE_sys_0}), the variational BSDE system with parameter $\theta\in(0,T]$ depends only on the number $i\in\{0,1,\ldots,N-1\}$ such that $\theta\in(t_i,t_{i+1}]$.
Moreover, if in addition the free term $\psi(t,x_1,x_2)$ and the driver $g(t,s,x_1,x_2,y,z_1)$ of the Type-\Rnum{1} BSVIE do not depend on $x_1$, then the terms $\psi_{x_1}(t_k)$ and $g_{x_1}(t_k,r)$ vanish, and thus the variational BSDE system becomes independent of the parameter $\theta\in[0,T]$.
\end{rem}

In order to investigate the $L^2$-time regularity of $\mathscr{Z}^\pi$, we prove some key properties of variational BSDE system~\eqref{Prop_L2_2}.
The following lemma shows a useful structural property.

\begin{lemm}\label{Lem_Lp_0}
Let the assumptions in Proposition \ref{Prop_L2_0} hold.
Then for any $\pi=\{t_{0},t_{1},\ldots,t_{N}\} \in \Pi[0,T]$, it holds that
\begin{align*}
(\nabla_{\theta} \mathscr{Y}^{\pi}(t_{k},s), \nabla_{\theta} \mathscr{Z}^{\pi}(t_{k},s))
=
(\nabla_{0} \mathscr{Y}^{\pi}(t_{k},s), \nabla_{0} \mathscr{Z}^{\pi}(t_{k},s)),~
s \in [0,T],~\theta \in [0,t_{k}],~k=0,\ldots,N-1.
\end{align*}
\end{lemm}
\begin{proof}
For simplicity of notation, we suppose $d=1$, that is, $W(\cdot)$ is a one-dimensional Brownian motion.

Let $\theta \in [0,t_{j}]$ with $j=0,\ldots,N-1$ be fixed.
We show that, by a backward induction, the following property holds for each $k=j,j+1,\ldots,N-1$:
\begin{itemize}
\item[$(\mathrm{P}_{k})$]
For any $s \in [0,T]$,
\begin{align*}
(\nabla_{\theta} \mathscr{Y}^{\pi}(t_{k},s), \nabla_{\theta} \mathscr{Z}^{\pi}(t_{k},s))
=
(\nabla_{0} \mathscr{Y}^{\pi}(t_{k},s), \nabla_{0} \mathscr{Z}^{\pi}(t_{k},s)).
\end{align*}
\end{itemize}

For $k=N-1$, we observe that
\begin{align*}
\nabla_{\theta}\mathscr{Y}^{\pi}(t_{N-1},s)
&=
\psi_{x_{1}}(t_{N-1})
\nabla X(t_{N-1})
+
\psi_{x_{2}}(t_{N-1})
\nabla X(T)
-
\int_{s}^{T}
	\nabla_{\theta} \mathscr{Z}^{\pi}(t_{N-1},r)
\rd W(r),
~s \in [0,T],
\\
\nabla_{0}\mathscr{Y}^{\pi}(t_{N-1},s)
&=
\psi_{x_{1}}(t_{N-1})
\nabla X(t_{N-1})
+
\psi_{x_{2}}(t_{N-1})
\nabla X(T)
-
\int_{s}^{T}
	\nabla_{0} \mathscr{Z}^{\pi}(t_{N-1},r)
\rd W(r),
~s \in [0,T].
\end{align*}
By the uniqueness of the adapted solution of the BSDE, we see that $(\mathrm{P}_{N-1})$ holds.

Next, for a fixed $k' \in \{j,j+1,\ldots,N-2\}$, assume that $(\mathrm{P}_{k})$ holds for any $k=k'+1,\ldots,N-1$.
We observe that
\begin{align*}
\nabla_{\theta}\mathscr{Y}^{\pi}(t_{k'},s)
&=
\psi_{x_{1}}(t_{k'})
\nabla X(t_{k'})
+
\psi_{x_{2}}(t_{k'})
\nabla X(T)
\\*&\quad+
\int_{s}^{T}
	\big\{
		g_{x_{1}}(t_{k'},r)
		\nabla X(t_{k'})
		+
		g_{x_{2}}(t_{k'},r)
		\nabla X(r)
		+
		g_{y}(t_{k'},r)
		\nabla_{\theta} \mathscr{Y}^{\pi}(\tau(r),r)
		\\*&\quad\quad\quad\quad\quad\quad
		+
		g_{z_{1}}(t_{k'},r)
		\nabla_{\theta} \mathscr{Z}^{\pi}(t_{k'},r)
		+
		g_{z_{2}}(t_{k'},r)
		\mathcal{I}^{\pi,\theta}[\nabla_{\theta}\mathscr{Z}^{\pi}(\tau(r),\cdot)](t_{k'})
	\big\}
	\1_{[t_{k'+1},T)}(r)
\rd r
\\*&\quad
-
\int_{s}^{T}
	\nabla_{\theta} \mathscr{Z}^{\pi}(t_{k'},r)
\rd W(r),~s \in [0,T].
\end{align*}
By the assumption of the induction, for any $r \in [t_{k'+1}, T)$,
\begin{align*}
\nabla_{\theta} \mathscr{Y}^{\pi}(\tau(r),r)
=
\nabla_{0} \mathscr{Y}^{\pi}(\tau(r),r)
\end{align*}
and
\begin{align*}
\mathcal{I}^{\pi,\theta}[\nabla_{\theta}\mathscr{Z}^{\pi}(\tau(r),\cdot)](t_{k'})
&=
\frac{1}{\Delta t_{k'}}
\int_{t_{k'}\vee \theta}^{t_{k'+1} \vee \theta}
	\nabla_{\theta}
	\mathscr{Z}^{\pi}(\tau(r),s)
\rd s
=
\frac{1}{\Delta t_{k'}}
\int_{t_{k'}}^{t_{k'+1}}
	\nabla_{0}
	\mathscr{Z}^{\pi}(\tau(r),s)
\rd s
\\*&=
\mathcal{I}^{\pi}[\nabla_{0}\mathscr{Z}^{\pi}(\tau(r),\cdot)](t_{k'}).
\end{align*}
Thus we have
\begin{align*}
\nabla_{\theta}\mathscr{Y}^{\pi}(t_{k'},s)
&=
\psi_{x_{1}}(t_{k'})
\nabla X(t_{k'})
+
\psi_{x_{2}}(t_{k'})
\nabla X(T)
\\*&\quad+
\int_{s}^{T}
	\big\{
		g_{x_{1}}(t_{k'},r)
		\nabla X(t_{k'})
		+
		g_{x_{2}}(t_{k'},r)
		\nabla X(r)
		+
		g_{y}(t_{k'},r)
		\nabla_{0} \mathscr{Y}^{\pi}(\tau(r),r)
		\\*&\quad\quad\quad\quad\quad\quad
		+
		g_{z_{1}}(t_{k'},r)
		\nabla_{\theta} \mathscr{Z}^{\pi}(t_{k'},r)
		+
		g_{z_{2}}(t_{k'},r)
		\mathcal{I}^{\pi}[\nabla_{0}\mathscr{Z}^{\pi}(\tau(r),\cdot)](t_{k'})
	\big\}
	\1_{[t_{k'+1},T)}(r)
\rd r
\\*&\quad
-
\int_{s}^{T}
	\nabla_{\theta} \mathscr{Z}^{\pi}(t_{k'},r)
\rd W(r),~s \in [0,T].
\end{align*}
On the other hand,  we have
\begin{align*}
\nabla_{0}\mathscr{Y}^{\pi}(t_{k'},s)
&=
\psi_{x_{1}}(t_{k'})
\nabla X(t_{k'})
+
\psi_{x_{2}}(t_{k'})
\nabla X(T)
\\&\quad+
\int_{s}^{T}
	\big\{
		g_{x_{1}}(t_{k'},r)
		\nabla X(t_{k'})
		+
		g_{x_{2}}(t_{k'},r)
		\nabla X(r)
		+
		g_{y}(t_{k'},r)
		\nabla_{0} \mathscr{Y}^{\pi}(\tau(r),r)
		\\&\quad\quad\quad\quad\quad\quad
		+
		g_{z_{1}}(t_{k'},r)
		\nabla_{0} \mathscr{Z}^{\pi}(t_{k'},r)
		+
		g_{z_{2}}(t_{k'},r)
		\mathcal{I}^{\pi}[\nabla_{0}\mathscr{Z}^{\pi}(\tau(r),\cdot)](t_{k'})
	\big\}
	\1_{[t_{k'+1},T)}(r)
\rd r
\\*&\quad
-
\int_{s}^{T}
	\nabla_{0} \mathscr{Z}^{\pi}(t_{k'},r)
\rd W(r),~s \in [0,T].
\end{align*}
By the uniqueness of the adapted solution of the BSDE, we see that $(\mathrm{P}_{k'})$ holds.
By the backward induction, $(\mathrm{P}_{k})$ holds for any $k=j,j+1,\ldots,N-1$.
Since $\theta \in [0,t_{j}]$ and $j=0,\ldots,N-1$ are arbitrary, we see that the assertion holds.
\end{proof}

Next, we provide some key $L^p$-estimates for $p \geq 2$.

\begin{lemm}\label{Lem_Lp_1}
Let the assumptions in Proposition \ref{Prop_L2_0} hold.
Fix arbitrary $p \geq 2$.
Then for any $\pi=\{t_{0},t_{1},\ldots,t_{N}\} \in \Pi[0,T]$, it holds that
\begin{align}
\begin{split}\label{Lem_Lp_1_0}
&\sup_{\theta \in [0,T]}
\e
\Big[
	\sup_{s \in [0,T]}
	\Big(
		\sum_{k=0}^{N-1}
		\Delta t_{k}
		\big|
			\nabla_{\theta}
			\mathscr{Y}^{\pi}(t_{k},s)
		\big|^{2}
	\Big)^{p/2}
	+
	\Big(
		\int_{0}^{T}
			\big|
				\nabla_{\theta} \mathscr{Y}^{\pi}(\tau(t),t)
			\big|^{2}
		\rd t
	\Big)^{p/2}
\\&\quad\quad\quad\quad\quad\quad\quad\quad
+
\Big(
	\int_{0}^{T}
		\int_{0}^{T}
			\big|
				\nabla_{\theta}
				\mathscr{Z}^{\pi}(\tau(t),s)
			\big|^{2}
		\rd s
	\rd t
	\Big)^{p/2}
\Big]
\leq
C_{p},
\end{split}
\\
\begin{split}\label{Lem_Lp_1_1}
&
\e
\Big[
	\sup_{s \in [0,T]}
	\Big(
		\int_{0}^{T}
			\big|
				\nabla_{t}
				\mathscr{Y}^{\pi}(\tau(t),s)
			\big|^{2}
		\rd t
	\Big)^{p/2}
	+
	\Big(
		\int_{0}^{T}
			\big|
				\nabla_{t}
				\mathscr{Y}^{\pi}(\tau(t),t)
			\big|^{2}
		\rd t
	\Big)^{p/2}
\\&\quad\quad\quad\quad\quad\quad\quad\quad
+
\Big(
	\int_{0}^{T}
		\int_{0}^{T}
			\big|
				\nabla_{t}
				\mathscr{Z}^{\pi}(\tau(t),s)
			\big|^{2}
		\rd s
	\rd t
	\Big)^{p/2}
\Big]
\leq
C_{p},
\end{split}
\\
\begin{split}\label{Lem_Lp_1_2}
&
\max_{k=0,\ldots,N-2}
\sup_{\theta \in [t_{k+1},T]}
\e
\Big[
	\sup_{s \in [0,T]}
	\big|
		\nabla_{\theta}
		\mathscr{Y}^{\pi}(t_{k},s)
	\big|^{p}
	+
	\Big(
		\int_{0}^{T}
			\big|
				\nabla_{\theta} \mathscr{Z}^{\pi}(t_{k},s)
			\big|^{2}
		\rd s
	\Big)^{p/2}
\Big]
\leq
C_{p}
\end{split}
\end{align}
and
\begin{align}
\begin{split}\label{Lem_Lp_1_3}
&\max_{k=0,\ldots,N-2}
\max_{\ell=k+1,\ldots, N-1}
\sup_{\theta \in [t_{\ell},t_{\ell+1}]}
\e
\Big[
	\sup_{s \in [t_{\ell},T]}
	\big|
		\nabla_{\theta}
		\mathscr{Y}^{\pi}(t_{k},s)
		-
		\nabla_{t_{k+1}}
		\mathscr{Y}^{\pi}(t_{k},s)
	\big|^{p}
	\\&\quad\quad\quad\quad\quad\quad\quad\quad\quad\quad\quad\quad\quad\quad\quad
	+
	\Big(
		\int_{t_{\ell}}^{T}
			\big|
				\nabla_{\theta}
				\mathscr{Z}^{\pi}(t_{k},s)
				-
				\nabla_{t_{k+1}}
				\mathscr{Z}^{\pi}(t_{k},s)
			\big|^{2}
		\rd s
	\Big)^{p/2}
\Big]
\leq
C_{p}
|\pi|^{p/2}.
\end{split}
\end{align}
\end{lemm}
\begin{proof}
For simplicity of notation, we suppose $d=1$, that is, $W(\cdot)$ is a one-dimensional Brownian motion.

\emph{Proof of \eqref{Lem_Lp_1_0}.}
Noting that $\e[\sup_{s \in [0,T]} |\nabla X(s)|^{p}] \leq C_{p}$, by the $L^{p}$-esimtate for BSDE systems (see Theorem \ref{Lem_Lp} (ii)), we see that \eqref{Lem_Lp_1_0} holds.

\emph{Proof of \eqref{Lem_Lp_1_1}.}
For any $t \in [0,T] \setminus \pi$, it holds that
\begin{align*}
\nabla_{t}\mathscr{Y}^{\pi}(\tau(t),t)
&=
\psi_{x_{2}}(\tau(t))
\nabla X(T)
\\*&\quad+
\int_{t}^{T}
	\big\{
		g_{x_{2}}(\tau(t),s)
		\nabla X(s)
		+
		g_{y}(\tau(t),s)
		\nabla_{t} \mathscr{Y}^{\pi}(\tau(s),s)
		\\*&\quad\quad\quad\quad\quad
		+
		g_{z_{1}}(\tau(t),s)
		\nabla_{t} \mathscr{Z}^{\pi}(\tau(t),s)
		+
		g_{z_{2}}(\tau(t),s)
		\mathcal{I}^{\pi,t}[\nabla_{t}\mathscr{Z}^{\pi}(\tau(s),\cdot)](\tau(t))
	\big\}
	\1_{[\tau^{*}(t),T)}(s)
\rd s
\\&\quad
-
\int_{t}^{T}
	\nabla_{t} \mathscr{Z}^{\pi}(\tau(t),s)
\rd W(s).
\end{align*}
From Lemma \ref{Lem_Lp_0}, for any $s \in [\tau^{*}(t),T)$, we have
\begin{align*}
\nabla_{t}\mathscr{Y}^{\pi}(\tau(s),s)
=
\nabla_{0}\mathscr{Y}^{\pi}(\tau(s),s)
\quad\text{and}\quad
\mathcal{I}^{\pi,t}[\nabla_{t}\mathscr{Z}^{\pi}(\tau(s),\cdot)](\tau(t))
&=
\mathcal{I}^{\pi,t}[\nabla_{0}\mathscr{Z}^{\pi}(\tau(s),\cdot)](\tau(t)).
\end{align*}
Thus, we have
\begin{align*}
\nabla_{t}\mathscr{Y}^{\pi}(\tau(t),t)
&=
\psi_{x_{2}}(\tau(t))
\nabla X(T)
\\&\quad+
\int_{t}^{T}
	\big\{
		g_{x_{2}}(\tau(t),s)
		\nabla X(s)
		+
		g_{y}(\tau(t),s)
		\nabla_{0} \mathscr{Y}^{\pi}(\tau(s),s)
		+
		g_{z_{1}}(\tau(t),s)
		\nabla_{t} \mathscr{Z}^{\pi}(\tau(t),s)
		\\&\quad\quad\quad\quad\quad
		+
		g_{z_{2}}(\tau(t),s)
		\mathcal{I}^{\pi,t}[\nabla_{0}\mathscr{Z}^{\pi}(\tau(s),\cdot)](\tau(t))
	\big\}
	\1_{[\tau^{*}(t),T)}(s)
\rd s
\\&\quad
-
\int_{t}^{T}
	\nabla_{t} \mathscr{Z}^{\pi}(\tau(t),s)
\rd W(s).
\end{align*}
Also, it holds that
\begin{align*}
\nabla_{t}\mathscr{Y}^{\pi}(\tau(t),t)
=
\e\Big[
	\nabla_{t}\mathscr{Y}^{\pi}(\tau(t),t)
\Big]
+
\int_{0}^{t}
	\nabla_{t}\mathscr{Z}^{\pi}(\tau(t),s)
\rd W(s).
\end{align*}
Thus, by defining
\begin{align*}
\begin{cases}
\eta(t)
:=
\nabla_{t}\mathscr{Y}^{\pi}(\tau(t),t),~
t \in [0,T],\\
\zeta(t,s)
:=
\nabla_{t}\mathscr{Z}^{\pi}(\tau(t),s),~
(t,s) \in [0,T]^{2},
\end{cases}
\end{align*}
we see that $(\eta(\cdot),\zeta(\cdot,\cdot))$  is the adapted M-solution of BSVIE \eqref{BSVIE_II_1} with the coefficients
\begin{align*}
&\Psi(t)
=
\psi_{x_{2}}(\tau(t))
\nabla X(T),
\\
&G(t,s,y,z_{1},z_{2})
=
\big\{
	g_{x_{2}}(\tau(t),s)
	\nabla X(s)
	+
	g_{y}(\tau(t),s)
	\nabla_{0} \mathscr{Y}^{\pi}(\tau(s),s)
	+
	g_{z_{1}}(\tau(t),s)
	z_{1}
	\\&\hspace{3cm}
	+
	g_{z_{2}}(\tau(t),s)
	\mathcal{I}^{\pi,t}[\nabla_{0}\mathscr{Z}^{\pi}(\tau(s),\cdot)](\tau(t))
\big\}
\1_{[\tau^{*}(t),T)}(s).
\end{align*}
Note that $\psi_{x_{2}}$, $g_{x_{2}}$, $g_{y}$, $g_{z_{1}}$ and $g_{z_{2}}$ are bounded by $L$.
Furthermore, from \eqref{Lem_Lp_1_0}, it holds that
\begin{align*}
\e\Big[
	\Big(
		\int_{0}^{T}
			\big|
				\nabla_{0} \mathscr{Y}^{\pi}(\tau(t),t)
			\big|^{2}
		\rd t
	\Big)^{p/2}
\Big]
\leq C_{p}
\end{align*}
and
\begin{align*}
&\e\Big[
	\Big(
		\int_{0}^{T}
			\int_{\tau^{*}(t)}^{T}
				\big|
					\mathcal{I}^{\pi,t}[\nabla_{0}\mathscr{Z}^{\pi}(\tau(s),\cdot)](\tau(t))
				\big|^{2}
			\rd s
		\rd t
	\Big)^{p/2}
\Big]
\\&\leq
\e\Big[
	\Big(
		\int_{0}^{T}
			\int_{\tau^{*}(t)}^{T}
				\frac{1}{\Delta \tau(t)}
				\int_{t}^{\tau^{*}(t)}
					\big|
						\nabla_{0}
						\mathscr{Z}^{\pi}(\tau(s),u)
					\big|^{2}
				\rd u
			\rd s
		\rd t
	\Big)^{p/2}
\Big]
\\&\leq
\e\Big[
	\Big(
		\int_{0}^{T}
			\int_{0}^{T}
				\big|
					\nabla_{0}
					\mathscr{Z}^{\pi}(\tau(t),s)
				\big|^{2}
			\rd s
		\rd t
	\Big)^{p/2}
\Big]
\leq C_{p}.
\end{align*}
Therefore, by the $L^{p}$-estimate for BSVIEs (see Theorem \ref{Lem_Lp} (i)), we see that \eqref{Lem_Lp_1_1} holds.

\emph{Proof of \eqref{Lem_Lp_1_2}.}
Let $k=0,\ldots,N-2$ and $\theta \in [t_{k+1},T]$ be fixed.
Noting that $\mathcal{I}^{\pi,\theta}[\cdot](t_{k})=0$, we have
\begin{align*}
\nabla_{\theta}\mathscr{Y}^{\pi}(t_{k},s)
&=
\psi_{x_{2}}(t_{k})
\nabla X(T)
\\&\quad+
\int_{s}^{T}
	\big\{
		g_{x_{2}}(t_{k},r)
		\nabla X(r)
		+
		g_{y}(t_{k},r)
		\nabla_{\theta} \mathscr{Y}^{\pi}(\tau(r),r)
		+
		g_{z_{1}}(t_{k},r)
		\nabla_{\theta} \mathscr{Z}^{\pi}(t_{k},r)
	\big\}
	\1_{[t_{k+1},T)}(r)
\rd r
\\&\quad
-
\int_{s}^{T}
	\nabla_{\theta} \mathscr{Z}^{\pi}(t_{k},r)
\rd W(r),~s \in [0,T].
\end{align*}
By the a priori estimate for BSDEs (see Lemma \ref{apriori_0}) and the estimate \eqref{Lem_Lp_1_0}, we have
\begin{align*}
&
\e
\Big[
	\sup_{s \in [0,T]}
	\big|
		\nabla_{\theta}
		\mathscr{Y}^{\pi}(t_{k},s)
	\big|^{p}
	+
	\Big(
		\int_{0}^{T}
			\big|
				\nabla_{\theta} \mathscr{Z}^{\pi}(t_{k},s)
			\big|^{2}
		\rd s
	\Big)^{p/2}
\Big]
\\&\leq
C_{p}
\e\Big[
	\big|
		\psi_{x_{2}}(t_{k})
		\nabla X(T)
	\big|^{p}
	+
	\Big(
		\int_{t_{k+1}}^{T}
			\big|
				g_{x_{2}}(t_{k},r)
				\nabla X(r)
				+
				g_{y}(t_{k},r)
				\nabla_{\theta} \mathscr{Y}^{\pi}(\tau(r),r)
			\big|
		\rd r
	\Big)^{p}
\Big]
\\&\leq
C_{p}
\e\Big[
	\sup_{s \in [0,T]}
	\big|
		\nabla X(s)
	\big|^{p}
	+
	\Big(
		\int_{0}^{T}
			\big|
				\nabla_{\theta} \mathscr{Y}^{\pi}(\tau(t),t)
			\big|^{2}
		\rd t
	\Big)^{p/2}
\Big]
\leq C_{p}.
\end{align*}
Thus, the estimate \eqref{Lem_Lp_1_2} holds.

\emph{Proof of \eqref{Lem_Lp_1_3}.}
Let $k=0,\ldots,N-2$, $\ell=k+1,\ldots,N-1$ and $\theta \in [t_{\ell},t_{\ell+1}]$ be fixed.
Noting that $\mathcal{I}^{\pi,\theta}[\cdot](t_{k})=\mathcal{I}^{\pi,t_{k+1}}[\cdot](t_{k})=0$, it holds from Lemma \ref{Lem_Lp_0} that, for any $s \in [t_{\ell},T]$,
\begin{align*}
\nabla_{\theta}\mathscr{Y}^{\pi}(t_{k},s)
&=
\psi_{x_{2}}(t_{k})
\nabla X(T)
\\*&\quad+
\int_{s}^{T}
	\big\{
		g_{x_{2}}(t_{k},r)
		\nabla X(r)
		+
		g_{y}(t_{k},r)
		\nabla_{\theta} \mathscr{Y}^{\pi}(\tau(r),r)
		+
		g_{z_{1}}(t_{k},r)
		\nabla_{\theta} \mathscr{Z}^{\pi}(t_{k},r)
	\big\}
\rd r
\\*&\quad
-
\int_{s}^{T}
	\nabla_{\theta} \mathscr{Z}^{\pi}(t_{k},r)
\rd W(r)
\\&=
\psi_{x_{2}}(t_{k})
\nabla X(T)
\\&\quad+
\int_{s}^{T}
	\big\{
		g_{x_{2}}(t_{k},r)
		\nabla X(r)
		+
		g_{y}(t_{k},r)
		\nabla_{0} \mathscr{Y}^{\pi}(\tau(r),r)
		\1_{[t_{\ell+1},T)}(r)
		\\&\quad\quad\quad\quad\quad
		+
		g_{y}(t_{k},r)
		\nabla_{\theta} \mathscr{Y}^{\pi}(t_{\ell},r)
		\1_{[t_{\ell},t_{\ell+1})}(r)
		+
		g_{z_{1}}(t_{k},r)
		\nabla_{\theta} \mathscr{Z}^{\pi}(t_{k},r)
	\big\}
\rd r
\\&\quad
-
\int_{s}^{T}
	\nabla_{\theta} \mathscr{Z}^{\pi}(t_{k},r)
\rd W(r),
\end{align*}
and
\begin{align*}
\nabla_{t_{k+1}}\mathscr{Y}^{\pi}(t_{k},s)
&=
\psi_{x_{2}}(t_{k})
\nabla X(T)
\\&\quad+
\int_{s}^{T}
	\big\{
		g_{x_{2}}(t_{k},r)
		\nabla X(r)
		+
		g_{y}(t_{k},r)
		\nabla_{t_{k+1}} \mathscr{Y}^{\pi}(\tau(r),r)
		+
		g_{z_{1}}(t_{k},r)
		\nabla_{t_{k+1}} \mathscr{Z}^{\pi}(t_{k},r)
	\big\}
\rd r
\\&\quad
-
\int_{s}^{T}
	\nabla_{t_{k+1}} \mathscr{Z}^{\pi}(t_{k},r)
\rd W(r)
\\&=
\psi_{x_{2}}(t_{k})
\nabla X(T)
\\&\quad+
\int_{s}^{T}
	\big\{
		g_{x_{2}}(t_{k},r)
		\nabla X(r)
		+
		g_{y}(t_{k},r)
		\nabla_{0} \mathscr{Y}^{\pi}(\tau(r),r)
		+
		g_{z_{1}}(t_{k},r)
		\nabla_{t_{k+1}} \mathscr{Z}^{\pi}(t_{k},r)
	\big\}
\rd r
\\&\quad
-
\int_{s}^{T}
	\nabla_{t_{k+1}} \mathscr{Z}^{\pi}(t_{k},r)
\rd W(r).
\end{align*}
By the stability estimate for BSDEs (see Lemma \ref{apriori_0}) and the estimate \eqref{Lem_Lp_1_0}, we have
\begin{align*}
&\e
\Big[
	\sup_{s \in [t_{\ell},T]}
	\big|
		\nabla_{\theta}
		\mathscr{Y}^{\pi}(t_{k},s)
		-
		\nabla_{t_{k+1}}
		\mathscr{Y}^{\pi}(t_{k},s)
	\big|^{p}
	+
	\Big(
		\int_{t_{\ell}}^{T}
			\big|
				\nabla_{\theta}
				\mathscr{Z}^{\pi}(t_{k},s)
				-
				\nabla_{t_{k+1}}
				\mathscr{Z}^{\pi}(t_{k},s)
			\big|^{2}
		\rd s
	\Big)^{p/2}
\Big]
\\*&\leq
C_{p}
\e\Big[
	\Big(
		\int_{t_{\ell}}^{t_{\ell+1}}
			\big|
				g_{y}(t_{k},r)
				\nabla_{\theta} \mathscr{Y}^{\pi}(t_{\ell},r)
				-
				g_{y}(t_{k},r)
				\nabla_{0} \mathscr{Y}^{\pi}(t_{\ell},r)
			\big|
		\rd r
	\Big)^{p}
\Big]
\\&\leq
C_{p}|\pi|^{p/2}
\Big\{
	\e\Big[
		\Big(
			\int_{t_{\ell}}^{t_{\ell+1}}
				\big|
					\nabla_{\theta} \mathscr{Y}^{\pi}(t_{\ell},r)
				\big|^{2}
			\rd r
		\Big)^{p/2}
	\Big]
	+
	\e\Big[
		\Big(
			\int_{t_{\ell}}^{t_{\ell+1}}
				\big|
					\nabla_{0} \mathscr{Y}^{\pi}(t_{\ell},r)
				\big|^{2}
			\rd r
		\Big)^{p/2}
	\Big]
\Big\}
\\&\leq
C_{p}|\pi|^{p/2}
\sup_{\theta \in [0,T]}
\e\Big[
	\Big(
		\int_{0}^{T}
			\big|
				\nabla_{\theta} \mathscr{Y}^{\pi}(\tau(t),t)
			\big|^{2}
		\rd t
	\Big)^{p/2}
\Big]
\leq C_{p}|\pi|^{p/2}.
\end{align*}
Thus, the estimate \eqref{Lem_Lp_1_3} holds, and we complete the proof.
\end{proof}

Now we are ready to estimate the modulus of the $L^{2}$-time regularity of $\mathscr{Z}^{\pi}$.

\begin{theo}\label{Thm_L2_0}
Under $(\mathrm{H}_{\psi,g})$ and $(\mathrm{H}_{b,\sigma})$, for any $\pi=\{t_{0},t_{1},\ldots,t_{N}\} \in \Pi[0,T]$, it holds that
\begin{align}\label{Thm_L2_0_1}
\begin{split}
&
\sum_{k=0}^{N-1}
\Delta t_{k}
\e
\Big[
	\int_{t_{k}}^{t_{k+1}}
		\big|
			\mathscr{Z}^{\pi}(t_{k},s)
		\big|^{2}
	\rd s
\Big]
\leq
C(1+|x|^{2})|\pi|
\end{split}
\end{align}
and
\begin{align}\label{Thm_L2_0_0}
\begin{split}
&
\sum_{k=0}^{N-1}
\Delta t_{k}
\sum_{\ell=0}^{N-1}
\e
\Big[
	\int_{t_{\ell}}^{t_{\ell+1}}
		\big|
			\mathscr{Z}^{\pi}(t_{k},s)
			-
			\overline{\mathscr{Z}}^{\pi}(t_{k},t_{\ell})
		\big|^{2}
	\rd s
\Big]
\leq
C(1+|x|^{2})
\big\{
	|\pi|
	+
	\rho_{\sigma}(|\pi|)^{2}
\big\}.
\end{split}
\end{align}
\end{theo}
\begin{proof}
For each $\varepsilon>0$, let $(b_{\varepsilon},\sigma_{\varepsilon},\psi_{\varepsilon},g_{\varepsilon})$ be a smooth mollifier of $(b,\sigma,\psi,g)$.
Note that $(b_{\varepsilon},\sigma_{\varepsilon},\psi_{\varepsilon},g_{\varepsilon})$ satisfies $(\mathrm{H}_{\psi,g})$ and $(\mathrm{H}_{b,\sigma})$ with the same constant $L$ and continuous functions $\rho_{\psi,g}$ and $\rho_{\sigma}$.
For each $\pi=\{t_{0},t_{1},\ldots,t_{N}\} \in \Pi[0,T]$, denote by $(X_{\varepsilon},\mathscr{Y}_{\varepsilon}^{\pi},\mathscr{Z}_{\varepsilon}^{\pi})$ the corresponding solution of SDE \eqref{SDE_0} and BSDE system \eqref{BSDE_sys_2}.
By using the stability estimate for SDEs (cf.\ Theorem 3.2.4 in \cite{Zh17}), we see that $\lim_{\varepsilon \downarrow 0} X_{\varepsilon}(\cdot)=X(\cdot)$ in $L_{\mathbb{F}}^{2}(\Omega;C([0,T];\real^{n}))$.
Then, by the stability estimate for BSDE systems (see Lemma \ref{Lem_BSDE_sys_0}), we can easily show that
\begin{align*}
\lim_{\varepsilon \downarrow 0}
\sum_{k=0}^{N-1}
\Delta t_{k}
\e\Big[
	\int_{0}^{T}
		\big|
			\mathscr{Z}_{\varepsilon}^{\pi}(t_{k},s)
			-
			\mathscr{Z}^{\pi}(t_{k},s)
		\big|^{2}
	\rd s
\Big]
=
0
\end{align*}
and
\begin{align*}
\lim_{\varepsilon \downarrow 0}
\sum_{k=0}^{N-1}
\Delta t_{k}
\sum_{\ell=0}^{N-1}
\Delta t_{\ell}
\e\Big[
	\big|
		\overline{\mathscr{Z}}_{\varepsilon}^{\pi}(t_{k},t_{\ell})
		-
		\overline{\mathscr{Z}}^{\pi}(t_{k},t_{\ell})
	\big|^{2}
\Big]
=
0.
\end{align*}
Thus, without loss of generality, we may assume that $b(s,x)$, $\sigma(s,x)$ are $C^{1}$ in $x \in \real^{n}$, and $\psi(t,x_{1},x_{2})$, $g(t,s,x_{1},x_{2},y,z_{1},z_{2})$ are $C^{1}$ in $(x_{1},x_{2},y,z_{1},z_{2}) \in \real^{n} \times \real^{n} \times \real^{m} \times \real^{m \times d} \times \real^{m \times d}$.
In this case, the assumptions in Proposition \ref{Prop_L2_0} hold.

We first prove \eqref{Thm_L2_0_1}.
By using the representation formula \eqref{Prop_L2_3}, together with the estimates \eqref{Lem_SDE_1_0}, \eqref{Lem_SDE_0_3} and \eqref{Lem_Lp_1_1}, we have
\begin{align*}
&
\sum_{k=0}^{N-1}
\Delta t_{k}
\e
\Big[
	\int_{t_{k}}^{t_{k+1}}
		\big|
			\mathscr{Z}^{\pi}(t_{k},s)
		\big|^{2}
	\rd s
\Big]
=
\sum_{k=0}^{N-1}
\Delta t_{k}
\e
\Big[
	\int_{t_{k}}^{t_{k+1}}
		\big|
			\nabla_{s}\mathscr{Y}^{\pi}(t_{k},s)
			(\nabla X(s))^{-1}
			\sigma(s,X(s))
		\big|^{2}
	\rd s
\Big]
\\&\leq
C
\sum_{k=0}^{N-1}
\Delta t_{k}
\e
\Big[
	\int_{t_{k}}^{t_{k+1}}
		\big|
			\nabla_{s}\mathscr{Y}^{\pi}(t_{k},s)
		\big|^{2}
	\rd s
	\sup_{s \in [0,T]}
	\big|
		(\nabla X(s))^{-1}
	\big|^{2}
	\Big(
		1
		+
		\sup_{s \in [0,T]}
		\big|X(s)\big|^{2}
	\Big)
\Big]
\\&\leq
C|\pi|
\e
\Big[
	\int_{0}^{T}
		\big|
			\nabla_{t}\mathscr{Y}^{\pi}(\tau(t),t)
		\big|^{2}
	\rd t
	\sup_{s \in [0,T]}
	\big|
		(\nabla X(s))^{-1}
	\big|^{2}
	\Big(
		1
		+
		\sup_{s \in [0,T]}
		\big|X(s)\big|^{2}
	\Big)
\Big]
\\&\leq
C|\pi|
\e
\Big[
	\Big(
		\int_{0}^{T}
			\big|
				\nabla_{t}\mathscr{Y}^{\pi}(\tau(t),t)
			\big|^{2}
		\rd t
	\Big)^{3}
\Big]^{1/3}
\e
\Big[
	\sup_{s \in [0,T]}
	\big|
		(\nabla X(s))^{-1}
	\big|^{6}
\Big]^{1/3}
\Big(
	1
	+
	\e
	\Big[
		\sup_{s \in [0,T]}
		\big|X(s)\textbf{}|^{6}
	\Big]^{1/3}
\Big)
\\&\leq
C(1+|x|^{2})|\pi|.
\end{align*}
Thus, the estimate \eqref{Thm_L2_0_1} holds.

Next, we prove \eqref{Thm_L2_0_0}.
Let $k,\ell=0,\ldots,N-1$ be fixed.
Note that $\overline{\mathscr{Z}}^{\pi}(t_{k},t_{\ell})$ is the best approximation of $\mathscr{Z}^{\pi}(t_{k},\cdot)$ on $[t_{\ell},t_{\ell+1}]$ in the following sense:
\begin{align*}
\e\Big[
	\int_{t_{\ell}}^{t_{\ell+1}}
		\big|
			\mathscr{Z}^{\pi}(t_{k},s)
			-
			\overline{\mathscr{Z}}^{\pi}(t_{k},t_{\ell})
		\big|^{2}
	\rd s
\Big]
\leq
\e\Big[
	\int_{t_{\ell}}^{t_{\ell+1}}
		\big|
			\mathscr{Z}^{\pi}(t_{k},s)
			-
			\zeta
		\big|^{2}
	\rd s
\Big]
\end{align*}
for any $\zeta \in L^{2}_{\mathcal{F}_{t_{\ell}}}(\Omega;\real^{m \times d})$.
Thus, by the representation formula \eqref{Prop_L2_3}, we have
\begin{align}\label{Thm_L2_0_5}
\begin{split}
&\e
\Big[
	\int_{t_{\ell}}^{t_{\ell+1}}
		\big|
			\mathscr{Z}^{\pi}(t_{k},s)
			-
			\overline{\mathscr{Z}}^{\pi}(t_{k},t_{\ell})
		\big|^{2}
	\rd s
\Big]
\\&
\leq
\e
\Big[
	\int_{t_{\ell}}^{t_{\ell+1}}
		\big|
			\nabla_{s}\mathscr{Y}^{\pi}(t_{k},s)
			(\nabla X(s))^{-1}
			\sigma(s,X(s))
			-
			\nabla_{t_{\ell}}\mathscr{Y}^{\pi}(t_{k},t_{\ell})
			(\nabla X(t_{\ell}))^{-1}
			\sigma(t_{\ell},X(t_{\ell}))
		\big|^{2}
	\rd s
\Big]
\\&
\leq
3
\big\{
	I_{k,\ell}^{(1)}
	+
	I_{k,\ell}^{(2)}
	+
	I_{k,\ell}^{(3)}
\big\},
\end{split}
\end{align}
where
\begin{align*}
I_{k,\ell}^{(1)}
&:=
\e
\Big[
	\int_{t_{\ell}}^{t_{\ell+1}}
		\big|
			\nabla_{s}\mathscr{Y}^{\pi}(t_{k},s)
		\big|^{2}
		\big|
			(\nabla X(s))^{-1}
		\big|^{2}
		\big|
			\sigma(s,X(s))
			-
			\sigma(t_{\ell},X(t_{\ell}))
		\big|^{2}
	\rd s
\Big],
\\
I_{k,\ell}^{(2)}
&:=
\e
\Big[
	\int_{t_{\ell}}^{t_{\ell+1}}
		\big|
			\nabla_{s}\mathscr{Y}^{\pi}(t_{k},s)
		\big|^{2}
		\big|
			(\nabla X(s))^{-1}
			-
			(\nabla X(t_{\ell}))^{-1}
		\big|^{2}
		\big|
			\sigma(t_{\ell},X(t_{\ell}))
		\big|^{2}
	\rd s
\Big],
\\
I_{k,\ell}^{(3)}
&:=
\e
\Big[
	\int_{t_{\ell}}^{t_{\ell+1}}
		\big|
			\nabla_{s}\mathscr{Y}^{\pi}(t_{k},s)
			-
			\nabla_{t_{\ell}}\mathscr{Y}^{\pi}(t_{k},t_{\ell})
		\big|^{2}
		\big|
			(\nabla X(t_{\ell}))^{-1}
		\big|^{2}
		\big|
			\sigma(t_{\ell},X(t_{\ell}))
		\big|^{2}
	\rd s
\Big].
\end{align*}
We estimate $\sum_{k=0}^{N-1}\Delta t_{k} \sum_{\ell=0}^{N-1} I_{k,\ell}^{(i)}$ for $i=1,2,3$.

For $i=1$, we divide the sum $\sum_{k=0}^{N-1}\Delta t_{k} \sum_{\ell=0}^{N-1} I_{k,\ell}^{(1)}$ into three cases: $k>\ell$, $k=\ell$ and $k<\ell$.

\emph{The sum of $I_{k,\ell}^{(1)}$ for $k>\ell$.}
By using Lemma \ref{Lem_Lp_0}, together with the estimates \eqref{Lem_SDE_1_0}, \eqref{Lem_SDE_1_1}, \eqref{Lem_SDE_0_3} and \eqref{Lem_Lp_1_0}, we have
\begin{align*}
&
\sum_{k=1}^{N-1}\Delta t_{k} \sum_{\ell=0}^{k-1}
I_{k,\ell}^{(1)}
=
\sum_{k=1}^{N-1}\Delta t_{k} \sum_{\ell=0}^{k-1}
\e
\Big[
	\int_{t_{\ell}}^{t_{\ell+1}}
		\big|
			\nabla_{s}\mathscr{Y}^{\pi}(t_{k},s)
		\big|^{2}
		\big|
			(\nabla X(s))^{-1}
		\big|^{2}
		\big|
			\sigma(s,X(s))
			-
			\sigma(t_{\ell},X(t_{\ell}))
		\big|^{2}
	\rd s
\Big]
\\&=
\sum_{k=1}^{N-1}\Delta t_{k} \sum_{\ell=0}^{k-1}
\e
\Big[
	\int_{t_{\ell}}^{t_{\ell+1}}
		\big|
			\nabla_{0}\mathscr{Y}^{\pi}(t_{k},s)
		\big|^{2}
		\big|
			(\nabla X(s))^{-1}
		\big|^{2}
		\big|
			\sigma(s,X(s))
			-
			\sigma(t_{\ell},X(t_{\ell}))
		\big|^{2}
	\rd s
\Big]
\\&\leq
C
\sum_{\ell=0}^{N-2}
\Delta t_{\ell}
\e
\Big[
	\sup_{s \in [0,T]}
	\Big(
		\sum_{k=\ell+1}^{N-1}
		\Delta t_{k}
		\big|
			\nabla_{0}\mathscr{Y}^{\pi}(t_{k},s)
		\big|^{2}
	\Big)
	\sup_{s \in [0,T]}
	\big|
		(\nabla X(s))^{-1}
	\big|^{2}
	\\&\quad \quad \quad\quad\quad\quad\quad
	\times
	\Big\{
		\rho_{\sigma}(|\pi|)^{2}
		\Big(
			1
			+
			\sup_{s \in [0,T]}|X(s)|^{2}
		\Big)
		+
		\sup_{s \in [t_{\ell},t_{\ell+1}]}
		\big|
			X(s)
			-
			X(t_{\ell})
		\big|^{2}
	\Big\}
\Big]
\\&\leq
C
\sum_{\ell=0}^{N-1}
\Delta t_{\ell}
\e
\Big[
	\sup_{s \in [0,T]}
	\Big(
		\sum_{k=0}^{N-1}
		\Delta t_{k}
		\big|
			\nabla_{0}\mathscr{Y}^{\pi}(t_{k},s)
		\big|^{2}
	\Big)^{3}
\Big]^{1/3}
\e
\Big[
	\sup_{s \in [0,T]}
	\big|
		(\nabla X(s))^{-1}
	\big|^{6}
\Big]^{1/3}
\\*&\quad \quad \quad\quad\quad
\times
\Big\{
	\rho_{\sigma}(|\pi|)^{2}
	\Big(
		1
		+
		\e\Big[
			\sup_{s \in [0,T]}\big|X(s)\big|^{6}
		\Big]^{1/3}
	\Big)
	+
	\e\Big[
		\sup_{s \in [t_{\ell},t_{\ell+1}]}
		\big|
			X(s)
			-
			X(t_{\ell})
		\big|^{6}
	\Big]^{1/3}
\Big\}
\\*&\leq
C(1+|x|^{2})
\big\{
	|\pi|
	+
	\rho_{\sigma}(|\pi|)^{2}
\big\}.
\end{align*}

\emph{The sum of $I_{k,\ell}^{(1)}$ for $k=\ell$.}
By using the estimates \eqref{Lem_SDE_1_0}, \eqref{Lem_SDE_1_1}, \eqref{Lem_SDE_0_3} and \eqref{Lem_Lp_1_1}, we have
\begin{align*}
&
\sum_{k=0}^{N-1}
\Delta t_{k}
I_{k,k}^{(1)}
=
\sum_{k=0}^{N-1}
\Delta t_{k}
\e
\Big[
	\int_{t_{k}}^{t_{k+1}}
		\big|
			\nabla_{s}\mathscr{Y}^{\pi}(t_{k},s)
		\big|^{2}
		\big|
			(\nabla X(s))^{-1}
		\big|^{2}
		\big|
			\sigma(s,X(s))
			-
			\sigma(t_{k},X(t_{k}))
		\big|^{2}
	\rd s
\Big]
\\&\leq
C
\sum_{k=0}^{N-1}
\Delta t_{k}
\e
\Big[
	\int_{0}^{T}
		\big|
			\nabla_{t}\mathscr{Y}^{\pi}(\tau(t),t)
		\big|^{2}
	\rd t
	\sup_{s \in [0,T]}
	\big|
		(\nabla X(s))^{-1}
	\big|^{2}
	\\&\quad \quad \quad\quad\quad\quad\quad\quad
	\times
	\Big\{
		\rho_{\sigma}(|\pi|)^{2}
		\Big(
			1
			+
			\sup_{s \in [0,T]}|X(s)|^{2}
		\Big)
		+
		\sup_{s \in [t_{k},t_{k+1}]}
		\big|
			X(s)
			-
			X(t_{k})
		\big|^{2}
	\Big\}
\Big]
\\&\leq
C
\sum_{k=0}^{N-1}
\Delta t_{k}
\e
\Big[
	\Big(
		\int_{0}^{T}
			\big|
				\nabla_{t}\mathscr{Y}^{\pi}(\tau(t),t)
			\big|^{2}
		\rd t
	\Big)^{3}
\Big]^{1/3}
\e
\Big[
	\sup_{s \in [0,T]}
	\big|
		(\nabla X(s))^{-1}
	\big|^{6}
\Big]^{1/3}
\\&\quad \quad \quad\quad\quad
\times
\Big\{
	\rho_{\sigma}(|\pi|)^{2}
	\Big(
		1
		+
		\e\Big[
			\sup_{s \in [0,T]}|X(s)|^{6}
		\Big]^{1/3}
	\Big)
	+
	\e\Big[
		\sup_{s \in [t_{k},t_{k+1}]}
		\big|
			X(s)
			-
			X(t_{k})
		\big|^{6}
	\Big]^{1/3}
\Big\}
\\&\leq
C(1+|x|^{2})
\big\{
	|\pi|
	+
	\rho_{\sigma}(|\pi|)^{2}
\big\}.
\end{align*}

\emph{The sum of $I_{k,\ell}^{(1)}$ for $k<\ell$.}
By using the estimates \eqref{Lem_SDE_1_0}, \eqref{Lem_SDE_1_1}, \eqref{Lem_SDE_0_3} and \eqref{Lem_Lp_1_2}, we have
\begin{align*}
&
\sum_{k=0}^{N-2}\Delta t_{k} \sum_{\ell=k+1}^{N-1}
I_{k,\ell}^{(1)}
=
\sum_{k=0}^{N-2}\Delta t_{k} \sum_{\ell=k+1}^{N-1}
\e
\Big[
	\int_{t_{\ell}}^{t_{\ell+1}}
		\big|
			\nabla_{s}\mathscr{Y}^{\pi}(t_{k},s)
		\big|^{2}
		\big|
			(\nabla X(s))^{-1}
		\big|^{2}
		\big|
			\sigma(s,X(s))
			-
			\sigma(t_{\ell},X(t_{\ell}))
		\big|^{2}
	\rd s
\Big]
\\&\leq
\sum_{k=0}^{N-2}\!
\Delta t_{k}\!
\sum_{\ell=k+1}^{N-1}\!
\int_{t_{\ell}}^{t_{\ell+1}}\!
	\e
	\Big[
		\big|
			\nabla_{s}\mathscr{Y}^{\pi}(t_{k},s)
		\big|^{6}
	\Big]^{1/3}
	\e
	\Big[
		\sup_{s \in [0,T]}
		\big|
			(\nabla X(s))^{-1}
		\big|^{6}
	\Big]^{1/3}
	\e\Big[
		\big|
			\sigma(s,X(s))
			-
			\sigma(t_{\ell},X(t_{\ell}))
		\big|^{6}
	\Big]^{1/3}
\rd s
\\&\leq
C
\sum_{k=0}^{N-2}
\Delta t_{k}
\sum_{\ell=k+1}^{N-1}
\Delta t_{\ell}
\sup_{\theta \in [t_{k+1},T]}
\e
\Big[
	\sup_{s \in [0,T]}
	\big|
		\nabla_{\theta}\mathscr{Y}^{\pi}(t_{k},s)
	\big|^{6}
\Big]^{1/3}
\e
\Big[
	\sup_{s \in [0,T]}
	\big|
		(\nabla X(s))^{-1}
	\big|^{6}
\Big]^{1/3}
\\&\hspace{4cm}\times
\Big\{
	\rho_{\sigma}(|\pi|)^{2}
	\Big(
		1
		+
		\e\Big[
			\sup_{s \in [0,T]}|X(s)|^{6}
		\Big]^{1/3}
	\Big)
	+
	\e\Big[
		\sup_{s \in [t_{\ell},t_{\ell+1}]}
		\big|
			X(s)
			-
			X(t_{\ell})
		\big|^{6}
	\Big]^{1/3}
\Big\}
\\&\leq
C(1+|x|^{2})
\big\{
	|\pi|
	+
	\rho_{\sigma}(|\pi|)^{2}
\big\}.
\end{align*}

Thus, we get
\begin{align}\label{Thm_L2_0_6}
\sum_{k=0}^{N-1}
\Delta t_{k}
\sum_{\ell=0}^{N-1}
I_{k,\ell}^{(1)}
&\leq
C(1+|x|^{2})
\big\{
	|\pi|
	+
	\rho_{\sigma}(|\pi|)^{2}
\big\}.
\end{align}

By the same way as above, noting the estimate \eqref{Lem_SDE_0_4}, we can show that
\begin{align}\label{Thm_L2_0_7}
\sum_{k=0}^{N-1}
\Delta t_{k}
\sum_{\ell=0}^{N-1}
I_{k,\ell}^{(2)}
&\leq
C(1+|x|^{2})|\pi|.
\end{align}

It remains to estimate $\sum_{k=0}^{N-1}\Delta t_{k}\sum_{\ell=0}^{N-1}I_{k,\ell}^{(3)}$.
Again we divide the sum into three cases: $k>\ell$, $k=\ell$ and $k<\ell$.

\emph{The sum of $I_{k,\ell}^{(3)}$ for $k>\ell$.}
By using Lemma \ref{Lem_Lp_0}, we have
\begin{align*}
&\sum_{k=1}^{N-1}
\Delta t_{k}
\sum_{\ell=0}^{k-1}
I_{k,\ell}^{(3)}
=
\sum_{k=1}^{N-1}
\Delta t_{k}
\sum_{\ell=0}^{k-1}
\e
\Big[
	\int_{t_{\ell}}^{t_{\ell+1}}
		\big|
			\nabla_{s}\mathscr{Y}^{\pi}(t_{k},s)
			-
			\nabla_{t_{\ell}}\mathscr{Y}^{\pi}(t_{k},t_{\ell})
		\big|^{2}
		\big|
			(\nabla X(t_{\ell}))^{-1}
		\big|^{2}
		\big|
			\sigma(t_{\ell},X(t_{\ell}))
		\big|^{2}
	\rd s
\Big]
\\&=
\sum_{k=1}^{N-1}
\Delta t_{k}
\sum_{\ell=0}^{k-1}
\e
\Big[
	\int_{t_{\ell}}^{t_{\ell+1}}
		\big|
			\nabla_{0}\mathscr{Y}^{\pi}(t_{k},s)
			-
			\nabla_{0}\mathscr{Y}^{\pi}(t_{k},t_{\ell})
		\big|^{2}
		\big|
			(\nabla X(t_{\ell}))^{-1}
		\big|^{2}
		\big|
			\sigma(t_{\ell},X(t_{\ell}))
		\big|^{2}
	\rd s
\Big]
\\*&\leq
|\pi|
\sum_{k=1}^{N-1}
\Delta t_{k}
\sum_{\ell=0}^{k-1}
\sup_{s \in [t_{\ell},t_{\ell+1}]}
\e
\Big[
	\big|
		\nabla_{0}\mathscr{Y}^{\pi}(t_{k},s)
		-
		\nabla_{0}\mathscr{Y}^{\pi}(t_{k},t_{\ell})
	\big|^{2}
	\big|
		(\nabla X(t_{\ell}))^{-1}
	\big|^{2}
	\big|
		\sigma(t_{\ell},X(t_{\ell}))
	\big|^{2}
\Big].
\end{align*}
Observe that, for any $k=0,\ldots,N-1$, $\ell=0,\ldots,k$ and $s \in [t_{\ell},t_{\ell+1}]$,
\begin{align}\label{eq_YZ_0}
\begin{split}
&\e
\Big[
	\big|
		\nabla_{0}\mathscr{Y}^{\pi}(t_{k},s)
		-
		\nabla_{0}\mathscr{Y}^{\pi}(t_{k},t_{\ell})
	\big|^{2}
	\big|
		(\nabla X(t_{\ell}))^{-1}
	\big|^{2}
	\big|
		\sigma(t_{\ell},X(t_{\ell}))
	\big|^{2}
\Big]
\\&=
\e
\Big[
	\Big|
		\int_{t_{\ell}}^{s}
			\nabla_{0}\mathscr{Z}^{\pi}(t_{k},r)
		\rd W(r)
	\Big|^{2}
	\big|
		(\nabla X(t_{\ell}))^{-1}
	\big|^{2}
	\big|
		\sigma(t_{\ell},X(t_{\ell}))
	\big|^{2}
\Big]
\\&=
\e
\Big[
	\int_{t_{\ell}}^{s}
		\big|
			\nabla_{0}\mathscr{Z}^{\pi}(t_{k},r)
		\big|^{2}
	\rd r
	\big|
		(\nabla X(t_{\ell}))^{-1}
	\big|^{2}
	\big|
		\sigma(t_{\ell},X(t_{\ell}))
	\big|^{2}
\Big].
\end{split}
\end{align}
Thus, by using the estimates \eqref{Lem_SDE_1_0}, \eqref{Lem_SDE_0_3} and \eqref{Lem_Lp_1_0}, we have
\begin{align*}
&\sum_{k=1}^{N-1}
\Delta t_{k}
\sum_{\ell=0}^{k-1}
I_{k,\ell}^{(3)}
\\&\leq
|\pi|
\sum_{k=1}^{N-1}
\Delta t_{k}
\sum_{\ell=0}^{k-1}
\e
\Big[
	\int_{t_{\ell}}^{t_{\ell+1}}
		\big|
			\nabla_{0}\mathscr{Z}^{\pi}(t_{k},s)
		\big|^{2}
	\rd s
	\big|
		(\nabla X(t_{\ell}))^{-1}
	\big|^{2}
	\big|
		\sigma(t_{\ell},X(t_{\ell}))
	\big|^{2}
\Big]
\\&\leq
C
|\pi|
\e
\Big[
	\sum_{k=1}^{N-1}
	\Delta t_{k}
		\int_{0}^{t_{k}}
			\big|
				\nabla_{0}\mathscr{Z}^{\pi}(t_{k},s)
			\big|^{2}
		\rd s
		\sup_{s \in [0,T]}
		\big|
			(\nabla X(s))^{-1}
		\big|^{2}
		\Big(
			1
			+
			\sup_{s \in [0,T]}
			\big|X(s)\big|^{2}
		\Big)
\Big]
\\&\leq
C|\pi|
\e
\Big[
	\Big(
		\int_{0}^{T}
			\int_{0}^{T}
				\big|
					\nabla_{0}\mathscr{Z}^{\pi}(\tau(t),s)
				\big|^{2}
			\rd s
		\rd t
	\Big)^{3}
\Big]^{1/3}
\e\Big[
	\sup_{s \in [0,T]}
	\big|
		(\nabla X(s))^{-1}
	\big|^{6}
\Big]^{1/3}
\Big(
	1+
	\e\Big[
		\sup_{s \in [0,T]}
		\big|X(s)\big|^{6}
	\Big]^{1/3}
\Big)
\\&\leq
C(1+|x|^{2})|\pi|.
\end{align*}

\emph{The sum of $I_{k,\ell}^{(3)}$ for $k=\ell$.}
By using Lemma \ref{Lem_Lp_0}, we have
\begin{align*}
&\sum_{k=0}^{N-1}
\Delta t_{k}
I_{k,k}^{(3)}
=
\sum_{k=0}^{N-1}
\Delta t_{k}
\e
\Big[
	\int_{t_{k}}^{t_{k+1}}
		\big|
			\nabla_{s}\mathscr{Y}^{\pi}(t_{k},s)
			-
			\nabla_{t_{k}}\mathscr{Y}^{\pi}(t_{k},t_{k})
		\big|^{2}
		\big|
			(\nabla X(t_{k}))^{-1}
		\big|^{2}
		\big|
			\sigma(t_{k},X(t_{k}))
		\big|^{2}
	\rd s
\Big]
\\&
=
\sum_{k=0}^{N-1}
\Delta t_{k}
\e
\Big[
	\int_{t_{k}}^{t_{k+1}}
		\big|
			\nabla_{s}\mathscr{Y}^{\pi}(t_{k},s)
			-
			\nabla_{0}\mathscr{Y}^{\pi}(t_{k},t_{k})
		\big|^{2}
		\big|
			(\nabla X(t_{k}))^{-1}
		\big|^{2}
		\big|
			\sigma(t_{k},X(t_{k}))
		\big|^{2}
	\rd s
\Big]
\\&\leq
3
\Big\{
\sum_{k=0}^{N-1}
\Delta t_{k}
\e
\Big[
	\int_{t_{k}}^{t_{k+1}}
		\big|
			\nabla_{s}\mathscr{Y}^{\pi}(t_{k},s)
		\big|^{2}
		\big|
			(\nabla X(t_{k}))^{-1}
		\big|^{2}
		\big|
			\sigma(t_{k},X(t_{k}))
		\big|^{2}
	\rd s
\Big]
\\&\quad\quad\quad+
\sum_{k=1}^{N-1}
\Delta t_{k}
\e
\Big[
	\int_{t_{k}}^{t_{k+1}}
		\big|
			\nabla_{0}\mathscr{Y}^{\pi}(t_{k},s)
		\big|^{2}
		\big|
			(\nabla X(t_{k}))^{-1}
		\big|^{2}
		\big|
			\sigma(t_{k},X(t_{k}))
		\big|^{2}
	\rd s
\Big]
\\&\quad \quad\quad+
\sum_{k=0}^{N-1}
\Delta t_{k}
\e
\Big[
	\int_{t_{k}}^{t_{k+1}}
		\big|
			\nabla_{0}\mathscr{Y}^{\pi}(t_{k},s)
			-
			\nabla_{0}\mathscr{Y}^{\pi}(t_{k},t_{k})
		\big|^{2}
		\big|
			(\nabla X(t_{k}))^{-1}
		\big|^{2}
		\big|
			\sigma(t_{k},X(t_{k}))
		\big|^{2}
	\rd s
\Big]
\Big\}.
\end{align*}
By using the equality \eqref{eq_YZ_0}, together with the estimates \eqref{Lem_SDE_1_0}, \eqref{Lem_SDE_0_3}, \eqref{Lem_Lp_1_0} and \eqref{Lem_Lp_1_1}, we have
\begin{align*}
\sum_{k=0}^{N-1}
\Delta t_{k}
I_{k,k}^{(3)}
&\leq
C|\pi|
\Big\{
	\e
	\Big[
		\int_{0}^{T}
			\big|
				\nabla_{t}\mathscr{Y}^{\pi}(\tau(t),t)
			\big|^{2}
		\rd t
		\sup_{s \in [0,T]}
		\big|
			(\nabla X(s))^{-1}
		\big|^{2}
		\Big(
			1
			+
			\sup_{s \in [0,T]}
			\big|
				X(s)
			\big|^{2}
		\Big)
	\Big]
	\\&\quad\quad\quad\quad
	+
	\e
	\Big[
		\int_{0}^{T}
			\big|
				\nabla_{0}\mathscr{Y}^{\pi}(\tau(t),t)
			\big|^{2}
		\rd t
		\sup_{s \in [0,T]}
		\big|
			(\nabla X(s))^{-1}
		\big|^{2}
		\Big(
			1
			+
			\sup_{s \in [0,T]}
			\big|
				X(s)
			\big|^{2}
		\Big)
	\Big]
	\\&\quad\quad\quad\quad
	+
	\sum_{k=0}^{N-1}
	\Delta t_{k}
	\e
	\Big[
		\int_{t_{k}}^{t_{k+1}}
			\big|
				\nabla_{0}\mathscr{Z}^{\pi}(t_{k},s)
			\big|^{2}
		\rd s
		\sup_{s \in [0,T]}
		\big|
			(\nabla X(s))^{-1}
		\big|^{2}
		\Big(
			1
			+
			\sup_{s \in [0,T]}
			\big|
				X(s)
			\big|^{2}
		\Big)
	\Big]
\Big\}
\\&\leq
C|\pi|
\Big\{
	\e
	\Big[
		\Big(
			\int_{0}^{T}
				\big|
					\nabla_{t}\mathscr{Y}^{\pi}(\tau(t),t)
				\big|^{2}
			\rd t
		\Big)^{3}
	\Big]^{1/3}
	+
	\e
	\Big[
		\Big(
			\int_{0}^{T}
				\big|
					\nabla_{0}\mathscr{Y}^{\pi}(\tau(t),t)
				\big|^{2}
			\rd t
		\Big)^{3}
	\Big]^{1/3}
	\\*&\quad\quad\quad\quad\quad\quad\quad
	+
	\e
	\Big[
		\Big(
			\int_{0}^{T}
				\int_{0}^{T}
					\big|
						\nabla_{0}\mathscr{Z}^{\pi}(\tau(t),s)
					\big|^{2}
				\rd s
			\rd t
		\Big)^{3}
	\Big]^{1/3}
\Big\}
\\*&\quad \times
\e\Big[
	\sup_{s \in [0,T]}
	\big|
		(\nabla X(s))^{-1}
	\big|^{6}
\Big]^{1/3}
\Big(
	1+
	\e\Big[
		\sup_{s \in [0,T]}
		\big|X(s)\big|^{6}
	\Big]^{1/3}
\Big)
\\&\leq
C(1+|x|^{2})|\pi|.
\end{align*}

\emph{The sum of $I_{k,\ell}^{(3)}$ for $k<\ell$.}
By using the estimates \eqref{Lem_SDE_1_0}, \eqref{Lem_SDE_0_3} and \eqref{Lem_Lp_1_3}, we have
\begin{align*}
&\sum_{k=0}^{N-2}
\Delta t_{k}\!
\sum_{\ell=k+1}^{N-1}\!
I_{k,\ell}^{(3)}
=
\sum_{k=0}^{N-2}
\Delta t_{k}
\sum_{\ell=k+1}^{N-1}
\e
\Big[
	\int_{t_{\ell}}^{t_{\ell+1}}
		\big|
			\nabla_{s}\mathscr{Y}^{\pi}(t_{k},s)
			-
			\nabla_{t_{\ell}}\mathscr{Y}^{\pi}(t_{k},t_{\ell})
		\big|^{2}
		\big|
			(\nabla X(t_{\ell}))^{-1}
		\big|^{2}
		\big|
			\sigma(t_{\ell},X(t_{\ell}))
		\big|^{2}
	\rd s
\Big]
\\&\leq
3
\Big\{
\sum_{k=0}^{N-2}
\Delta t_{k}
\sum_{\ell=k+1}^{N-1}
\e
\Big[
	\int_{t_{\ell}}^{t_{\ell+1}}
		\big|
			\nabla_{s}\mathscr{Y}^{\pi}(t_{k},s)
			-
			\nabla_{t_{k+1}}\mathscr{Y}^{\pi}(t_{k},s)
		\big|^{2}
		\big|
			(\nabla X(t_{\ell}))^{-1}
		\big|^{2}
		\big|
			\sigma(t_{\ell},X(t_{\ell}))
		\big|^{2}
	\rd s
\Big]
\\&\quad\quad+
\sum_{k=0}^{N-2}
\Delta t_{k}
\sum_{\ell=k+1}^{N-1}
\Delta t_{\ell}
\e
\Big[
	\big|
		\nabla_{t_{\ell}}\mathscr{Y}^{\pi}(t_{k},t_{\ell})
		-
		\nabla_{t_{k+1}}\mathscr{Y}^{\pi}(t_{k},t_{\ell})
	\big|^{2}
	\big|
		(\nabla X(t_{\ell}))^{-1}
	\big|^{2}
	\big|
		\sigma(t_{\ell},X(t_{\ell}))
	\big|^{2}
\Big]
\\&\quad\quad+
\sum_{k=0}^{N-2}
\Delta t_{k}
\sum_{\ell=k+1}^{N-1}
\e
\Big[
	\int_{t_{\ell}}^{t_{\ell+1}}
		\big|
			\nabla_{t_{k+1}}\mathscr{Y}^{\pi}(t_{k},s)
			-
			\nabla_{t_{k+1}}\mathscr{Y}^{\pi}(t_{k},t_{\ell})
		\big|^{2}
		\big|
			(\nabla X(t_{\ell}))^{-1}
		\big|^{2}
		\big|
			\sigma(t_{\ell},X(t_{\ell}))
		\big|^{2}
	\rd s
\Big]
\Big\}
\\&\leq
C
\sum_{k=0}^{N-2}
\Delta t_{k}
\sum_{\ell=k+1}^{N-1}
\Delta t_{\ell}
\sup_{\theta \in [t_{\ell},t_{\ell+1}]}
\e
\Big[
	\sup_{s \in [t_{\ell},T]}
	\big|
		\nabla_{\theta}\mathscr{Y}^{\pi}(t_{k},s)
		-
		\nabla_{t_{k+1}}\mathscr{Y}^{\pi}(t_{k},s)
	\big|^{6}
\Big]^{1/3}
\\&\quad\quad\quad\quad\quad\quad\quad\quad 
\times
\e\Big[
	\sup_{s \in [0,T]}
	\big|
		(\nabla X(s))^{-1}
	\big|^{6}
\Big]^{1/3}
\Big(
	1+
	\e\Big[
		\sup_{s \in [0,T]}
		\big|X(s)\big|^{6}
	\Big]^{1/3}
\Big)
\\&\quad+
3
|\pi|
\sum_{k=0}^{N-2}
\Delta t_{k}
\sum_{\ell=k+1}^{N-1}
\sup_{s \in [t_{\ell},t_{\ell+1}]}
\e
\Big[
	\big|
		\nabla_{t_{k+1}}\mathscr{Y}^{\pi}(t_{k},s)
		-
		\nabla_{t_{k+1}}\mathscr{Y}^{\pi}(t_{k},t_{\ell})
	\big|^{2}
	\big|
		(\nabla X(t_{\ell}))^{-1}
	\big|^{2}
	\big|
		\sigma(t_{\ell},X(t_{\ell}))
	\big|^{2}
\Big]
\\&\leq
C(1+|x|^{2})|\pi|
\\&\quad+
3
|\pi|
\sum_{k=0}^{N-2}
\Delta t_{k}
\sum_{\ell=k+1}^{N-1}
\sup_{s \in [t_{\ell},t_{\ell+1}]}
\e
\Big[
	\big|
		\nabla_{t_{k+1}}\mathscr{Y}^{\pi}(t_{k},s)
		-
		\nabla_{t_{k+1}}\mathscr{Y}^{\pi}(t_{k},t_{\ell})
	\big|^{2}
	\big|
		(\nabla X(t_{\ell}))^{-1}
	\big|^{2}
	\big|
		\sigma(t_{\ell},X(t_{\ell}))
	\big|^{2}
\Big].
\end{align*}
We estimate the last sum in the above inequalities.
By using Lemma \ref{Lem_Lp_0}, together with the estimates \eqref{Lem_SDE_1_0}, \eqref{Lem_SDE_0_3}, \eqref{Lem_Lp_1_0} and \eqref{Lem_Lp_1_2}, we have
\begin{align*}
&
\sum_{k=0}^{N-2}
\Delta t_{k}
\sum_{\ell=k+1}^{N-1}
\sup_{s \in [t_{\ell},t_{\ell+1}]}
\e
\Big[
	\big|
		\nabla_{t_{k+1}}\mathscr{Y}^{\pi}(t_{k},s)
		-
		\nabla_{t_{k+1}}\mathscr{Y}^{\pi}(t_{k},t_{\ell})
	\big|^{2}
	\big|
		(\nabla X(t_{\ell}))^{-1}
	\big|^{2}
	\big|
		\sigma(t_{\ell},X(t_{\ell}))
	\big|^{2}
\Big]
\\*&=
\sum_{k=0}^{N-2}
\Delta t_{k}
\sum_{\ell=k+1}^{N-1}
\sup_{s \in [t_{\ell},t_{\ell+1}]}
\e
\Big[
	\Big|
		-
		\int_{t_{\ell}}^{s}
			\Big\{
				g_{x_{2}}(t_{k},r)
				\nabla X(r)
				+
				g_{y}(t_{k},r)
				\nabla_{t_{k+1}} \mathscr{Y}^{\pi}(t_{\ell},r)
				\\&\hspace{1cm}
				+
				\sum_{\alpha=1}^{d}
				g_{z_{1}^{(\alpha)}}(t_{k},r)
				\nabla_{t_{k+1}} \mathscr{Z}^{\pi,(\alpha)}(t_{k},r)
			\Big\}
		\rd r
		+
		\int_{t_{\ell}}^{s}
			\nabla_{t_{k+1}} \mathscr{Z}^{\pi}(t_{k},r)
		\rd W(r)
	\Big|^{2}
	\big|
		(\nabla X(t_{\ell}))^{-1}
	\big|^{2}
	\big|
		\sigma(t_{\ell},X(t_{\ell}))
	\big|^{2}
\Big]
\\&=
\sum_{k=0}^{N-2}
\Delta t_{k}
\sum_{\ell=k+1}^{N-1}
\sup_{s \in [t_{\ell},t_{\ell+1}]}
\e
\Big[
	\Big|
		-
		\int_{t_{\ell}}^{s}
			\Big\{
				g_{x_{2}}(t_{k},r)
				\nabla X(r)
				+
				g_{y}(t_{k},r)
				\nabla_{0} \mathscr{Y}^{\pi}(t_{\ell},r)
				\\&\hspace{1cm}
				+
				\sum_{\alpha=1}^{d}
				g_{z_{1}^{(\alpha)}}(t_{k},r)
				\nabla_{t_{k+1}} \mathscr{Z}^{\pi,(\alpha)}(t_{k},r)
			\Big\}
		\rd r
		+
		\int_{t_{\ell}}^{s}
			\nabla_{t_{k+1}} \mathscr{Z}^{\pi}(t_{k},r)
		\rd W(r)
	\Big|^{2}
	\big|
		(\nabla X(t_{\ell}))^{-1}
	\big|^{2}
	\big|
		\sigma(t_{\ell},X(t_{\ell}))
	\big|^{2}
\Big]
\\&\leq
C
\sum_{k=0}^{N-2}
\Delta t_{k}
\sum_{\ell=k+1}^{N-1}
\e
\Big[
	\Big\{
		(\Delta t_{\ell})^{2}
		\sup_{s \in [0,T]}
		\big|
			\nabla X(s)
		\big|^{2}
		+
		\Delta t_{\ell}
		\int_{t_{\ell}}^{t_{\ell+1}}
			\big|
				\nabla_{0} \mathscr{Y}^{\pi}(t_{\ell},r)
			\big|^{2}
		\rd r
		\\&\quad\quad\quad\quad\quad\quad\quad\quad\quad\quad\quad
		+
		(1+\Delta t_{\ell})
		\int_{t_{\ell}}^{t_{\ell+1}}
			\big|
				\nabla_{t_{k+1}} \mathscr{Z}^{\pi}(t_{k},r)
			\big|^{2}
		\rd r
	\Big\}
	\sup_{s \in [0,T]}
	\big|
		(\nabla X(s))^{-1}
	\big|^{2}
	\Big(
		1
		+
		\sup_{s \in [0,T]}
		\big|
			X(s)
		\big|^{2}
	\Big)
\Big]
\\&\leq
C
\e\Big[
	\Big\{
		|\pi|
		\sup_{s \in [0,T]}
		\big|
			\nabla X(s)
		\big|^{2}
		+
		|\pi|
		\int_{0}^{T}
			\big|
				\nabla_{0} \mathscr{Y}^{\pi}(\tau(t),t)
			\big|^{2}
		\rd t
		+
		\sum_{k=0}^{ N-2}
		\Delta t_{k}
		\int_{t_{k+1}}^{T}
			\big|
				\nabla_{t_{k+1}} \mathscr{Z}^{\pi}(t_{k},s)
			\big|^{2}
		\rd s
	\Big\}
	\\&\quad\quad\quad\quad\quad
	\times
	\sup_{s \in [0,T]}
	\big|
		(\nabla X(s))^{-1}
	\big|^{2}
	\Big(
		1
		+
		\sup_{s \in [0,T]}
		\big|
			X(s)
		\big|^{2}
	\Big)
\Big]
\\&\leq
C
\Big\{
	|\pi|
	\e\Big[
		\sup_{s \in [0,T]}
		\big|
			\nabla X(s)
		\big|^{6}
	\Big]^{1/3}
	+
	|\pi|
	\e\Big[
		\Big(
			\int_{0}^{T}
				\big|
					\nabla_{0} \mathscr{Y}^{\pi}(\tau(t),t)
				\big|^{2}
			\rd t
		\Big)^{3}
	\Big]^{1/3}
	\\&\quad\quad\quad\quad
	+
	\max_{k=0,\ldots,N-2}
	\sup_{\theta \in [t_{k+1},T]}
	\e\Big[
		\Big(
			\int_{0}^{T}
				\big|
					\nabla_{\theta} \mathscr{Z}^{\pi}(t_{k},s)
				\big|^{2}
			\rd s
		\Big)^{3}
	\Big]^{1/3}
\Big\}
\\&\quad\quad\quad
\times
\e\Big[
	\sup_{s \in [0,T]}
	\big|
		(\nabla X(s))^{-1}
	\big|^{6}
\Big]^{1/3}
\Big(
	1+
	\e\Big[
		\sup_{s \in [0,T]}
		\big|X(s)\big|^{6}
	\Big]^{1/3}
\Big)
\\&\leq
C(1+|x|^{2}).
\end{align*}
Thus, we get
\begin{align*}
\sum_{k=0}^{N-2}
\Delta t_{k}
\sum_{\ell=k+1}^{N-1}
I_{k,\ell}^{(3)}
&\leq
C(1+|x|^{2})|\pi|.
\end{align*}

Consequently, it holds that
\begin{align}\label{Thm_L2_0_8}
\sum_{k=0}^{N-1}
\Delta t_{k}
\sum_{\ell=0}^{N-1}
I_{k,\ell}^{(3)}
&\leq
C(1+|x|^{2})|\pi|.
\end{align}
From \eqref{Thm_L2_0_5}, \eqref{Thm_L2_0_6}, \eqref{Thm_L2_0_7} and \eqref{Thm_L2_0_8}, we obtain \eqref{Thm_L2_0_0}, and we finish the proof.
\end{proof}

We provide proofs of our main results.

\begin{proof}[Proof of Theorem~\ref{Thm_1}]
We observe that, under the assumptions $(\mathrm{H}_{\psi,g})$ and $(\mathrm{H}_{b,\sigma})$, the coefficients $(\Psi,G)$ defined by
\begin{align}\label{proof_1}
\begin{split}
&\Psi(t)
=
\psi(t,X(t),X(T)),~t \in [0,T],\\
&G
(t,s,y,z_{1},z_{2})
=
g
(
	t,
	s,
	X(t),
	X(s),
	y,
	z_{1},
	z_{2}
),~
(t,s,y,z_{1},z_{2}) \in \Delta[0,T] \times \real^{m} \times \real^{m \times d} \times \real^{m \times d},
\end{split}
\end{align}
satisfy $(\mathrm{H}_{\Psi,G})$ with $M$ replaced by $C\sqrt{1+|x|^2}$ and $\rho_{\Psi,G}(t)$ replaced by $\sqrt{t}+\rho_{\psi,g}(t)$.
Thus, by Corollary~\ref{Cor_0}, for any $\pi=\{t_0,t_1,\ldots,t_N\}\in\Pi[0,T]$, it holds that
\begin{align*}
	\mathcal{E}(Y;\pi)+\mathcal{E}(Z;\pi)
	&\leq
	C\Bigl\{\bigl(1+|x|^2\bigr)\bigl(|\pi|+\rho_{\psi,g}(|\pi|)^2\bigr)+\sum^{N-1}_{k=0}\Delta t_k\e\Bigl[\int^{t_{k+1}}_{t_k}\big|\mathscr{Z}^\pi(t_k,s)\big|^2\rd s\Bigr]\\*
	&\hspace{1cm}+\sum^{N-1}_{k=0}\Delta t_k\sum^{N-1}_{\ell=0}\e\Bigl[\int^{t_{\ell+1}}_{t_\ell}\big|\mathscr{Z}^\pi(t_k,s)-\overline{\mathscr{Z}}^\pi(t_k,t_\ell)\big|^2\rd s\Bigr]\Bigr\}.
\end{align*}
Therefore, by Theorem~\ref{Thm_L2_0}, we get the assertions.
\end{proof}

\begin{proof}[Proof of Theorem~\ref{Thm_2}]
Under the assumptions $(\mathrm{H}_{\psi,g})'$ and $(\mathrm{H}_{b,\sigma})'$, $(\Psi,G)$ defined by \eqref{proof_1}, together with $(\Psi^\pi,G^\pi)$ defined by
\begin{align*}
&\Psi^{\pi}(t_{k})
=
\psi(t_{k},X^{\pi}(t_{k}),X^{\pi}(t_{N})),~
k=0,\ldots,N-1,
\\
&G^{\pi}(t_{k},t_{\ell},y,z_{1},z_{2})
=
g(t_{k},t_{\ell},X^{\pi}(t_{k}),X^{\pi}(t_{\ell}), y, z_{1}, z_{2}),\\
&
\hspace{2cm}
(y,z_{1},z_{2}) \in \real^{m} \times \real^{m \times d} \times \real^{m \times d},~
k=0,\ldots,N-1,~\ell=k,\ldots,N-1,
\end{align*}
satisfy $(\mathrm{H}_{\Psi,G})'$ with $M$ replaced by $C\sqrt{1+|x|^2}$.
Thus, by Proposition~\ref{Prop_EM_1}, for any $\pi=\{t_0,t_1,\ldots,t_N\}\in\Pi[0,T]$ with $|\pi|\leq\delta$, it holds that
\begin{align*}
&\sum^{N-1}_{k=0}
\e\Bigl[\int^{t_{k+1}}_{t_k}\big|Y(t)-Y^\pi(t_k,t_k)\big|^2\rd t
+
\sum^{N-1}_{k=0}\sum^{N-1}_{\ell=0}
\e\Bigl[\int^{t_{k+1}}_{t_k}\int^{t_{\ell+1}}_{t_\ell}\big|Z(t,s)-Z^\pi(t_k,t_\ell)\big|^2\rd s\rd t\Bigr]\\
&\leq C\Bigl\{\bigl(1+|x|^2\bigr)|\pi|
+\sum^{N-1}_{k=0}\Delta t_k
\e\Bigl[\int^{t_{k+1}}_{t_k}\big|\mathscr{Z}^\pi(t_k,s)\big|^2\rd s\Bigr]\\
&\hspace{2cm}+\sum^{N-1}_{k=0}\Delta t_k\sum^{N-1}_{\ell=0}\e\Bigl[\int^{t_{\ell+1}}_{t_\ell}\big|\mathscr{Z}^\pi(t_k,s)-\overline{\mathscr{Z}}^\pi(t_k,t_\ell)\big|^2\rd s\Bigr]\Bigr\}.
\end{align*}
Noting that $\rho_\sigma(t)=L\sqrt{t}$, by Theorem~\ref{Thm_L2_0}, we get the assertion.
\end{proof}

%\appendix

\begin{appendix}
\section{Appendix: $L^{p}$-a priori estimates}\label{sec_Appendix}

In this appendix, we provide a proof of Theorem \ref{Lem_Lp}.
First, we show the following lemma.

\begin{lemm}\label{lemma_A}
Assume that $(\cH,\cG)$ satisfies the conditions (\rnum{1}), (\rnum{2}) and (\rnum{3}) in $\mathrm{(H}_{\Psi,G}\mathrm{)}$ with $(\Psi,G)=(\cH,\cG)$, and that $\cH$ and $\cG$ are bounded. Suppose that a triplet
\begin{equation*}
	(\cY(\cdot,\cdot),\cZ(\cdot,\cdot),\zeta(\cdot,\cdot))\in L^2(0,T;L^2_\mathbb{F}(\Omega;C([0,T];\real^m)))\times L^2(0,T;L^2_\mathbb{F}(0,T;\real^{m\times d}))\times L^2(0,T;L^2_\mathbb{F}(0,T;\real^{m\times d}))
\end{equation*}
satisfies
\begin{equation}\label{A_*}
\begin{split}
	&\cY(t,s)=\cH(t)+\int^T_s\cG(t,r,\cY(r,r),\cZ(t,r),\zeta(t,r))\1_{[t,T]}(r)\rd r-\int^T_s\cZ(t,r)\rd W(r),\\
	&\hspace{5cm}s\in[0,T],\ \text{a.e.}\ t\in[0,T], \text{a.s.},
\end{split}
\end{equation}
and
\begin{equation}\label{A_**}
	\int^s_0\big|\zeta(t,s)\big|^2\rd t\leq\int^s_0\big|\cZ(s,t)\big|^2\rd t,\ \text{a.e.}\ s\in[0,T],\ \text{a.s.}
\end{equation}
Then for any $p\geq2$, it holds that
\begin{equation}\label{A_***}
\begin{split}
	&\e\Bigl[\sup_{s\in[0,T]}\Bigl(\int^T_0\big|\cY(t,s)\big|^2\rd t\Bigr)^{p/2}+\Bigl(\int^T_0\big|\cY(t,t)\big|^2\rd t\Bigr)^{p/2}+\Bigl(\int^T_0\int^T_0\big|\cZ(t,s)\big|^2\rd s\rd t\Bigr)^{p/2}\Bigr]\\
	&\leq C_p\e\Bigl[\Bigl(\int^T_0\big|\cH(t)\big|^2\rd t\Bigr)^{p/2}+\Bigl(\int^T_0\int^T_t\big|\cG(t,s,0,0,0)\big|^2\rd s\rd t\Bigr)^{p/2}\Bigr].
\end{split}
\end{equation}
\end{lemm}

\begin{proof}
For simplicity of notation, we let $m=d=1$. The proof can be easily generalized to the multi-dimensional case. We fix $p\geq2$ and prove the estimate \eqref{A_***}.

By taking conditional expectations on both sides of \eqref{A_*}, we get
\begin{equation*}
	\cY(t,s)=\e_s\Bigl[\cH(t)+\int^T_s\cG(t,r,\cY(r,r),\cZ(t,r),\zeta(t,r))\1_{[t,T]}(r)\rd r\Bigr],\ s\in[0,T],\ \text{a.e.}\ t\in[0,T],\ \text{a.s.}
\end{equation*}
Since $\cH$ and $\cG$ are bounded, we see that $\cY(\cdot,\cdot)$ is bounded. Again by \eqref{A_*} and the Burkholder--Davis--Gundy inequality, we see that
\begin{equation*}
	\underset{t\in[0,T]}{\mathrm{ess\,sup}}\,\e\Bigl[\Bigl(\int^T_0\big|\cZ(t,s)\big|^2\rd s\Bigr)^{p/2}\Bigr]<\infty.
\end{equation*}

We investigate some key dynamics. Fix an arbitrary constant $\lambda>0$. Noting \eqref{A_*}, by using It\^{o}'s formula for $s\mapsto e^{\lambda s}|\cY(t,s)|^2$ on $[t,T]$ for each fixed $t\in[0,T]$, we get
\begin{equation}\label{A_1}
\begin{split}
	e^{\lambda s}\big|\cY(t,s)\big|^2
	&=e^{\lambda T}\big|\cH(t)\big|^2\\
	&\quad+\int^T_s\bigl\{-\lambda e^{\lambda r}\big|\cY(t,r)\big|^2+2e^{\lambda r}\cY(t,r)\cG(t,r,\cY(r,r),\cZ(t,r),\zeta(t,r))-e^{\lambda r}\big|\cZ(t,r)\big|^2\bigr\}\rd r\\
	&\quad-\int^T_s2e^{\lambda r}\cY(t,r)\cZ(t,r)\rd W(r),\ s\in[t,T],\ \text{a.e.}\ t\in[0,T].
\end{split}
\end{equation}
In particular, we have
\begin{equation}\label{A_2}
\begin{split}
	e^{\lambda t}\big|\cY(t,t)\big|^2
	&=e^{\lambda T}\big|\cH(t)\big|^2\\
	&\quad+\int^T_t\bigl\{-\lambda e^{\lambda s}\big|\cY(t,s)\big|^2+2e^{\lambda s}\cY(t,s)\cG(t,s,\cY(s,s),\cZ(t,s),\zeta(t,s))-e^{\lambda s}\big|\cZ(t,s)\big|^2\bigr\}\rd s\\
	&\quad-\int^T_t2e^{\lambda s}\cY(t,s)\cZ(t,s)\rd W(s),\ \text{a.e.}\ t\in[0,T].
\end{split}
\end{equation}
Also, by using It\^o's formula for $s \mapsto |\cY(t,s)|^{2}$ on $[0,t]$ for each fixed $t \in [0,T]$, we have
\begin{equation}\label{A_3}
	\big|\cY(t,s)\big|^2
	=\big|\cY(t,t)\big|^2
	-
	\int^t_s\big|\cZ(t,r)\big|^2\rd r
	-
	\int^t_s2\cY(t,r)\cZ(t,r)\rd W(r),\ s\in[0,t],\ \text{a.e.}\ t\in[0,T],
\end{equation}
and thus
\begin{equation}\label{A_Z}
	\int^t_0\big|\cZ(t,s)\big|^2\rd s\leq\big|\cY(t,t)\big|^2-\int^t_02\cY(t,s)\cZ(t,s)\rd W(s),\ \text{a.e.}\ t\in[0,T].
\end{equation}
Combining \eqref{A_1}, \eqref{A_2} and \eqref{A_3}, we obtain
\begin{equation}\label{A_4}
\begin{split}
	e^{\lambda(t\vee s)}\big|\cY(t,s)\big|^2&=e^{\lambda T}\big|\cH(t)\big|^2\\
	&\quad+\int^T_s\bigl\{-\lambda e^{\lambda r}\big|\cY(t,r)\big|^2+2e^{\lambda r}\cY(t,r)\cG(t,r,\cY(r,r),\cZ(t,r),\zeta(t,r))\bigr\}\1_{[t,T]}(r)\rd r\\
	&\quad-\int^T_se^{\lambda(t\vee r)}\big|\cZ(t,r)\big|^2\rd r\\
	&\quad-\int^T_s2e^{\lambda(t\vee r)}\cY(t,r)\cZ(t,r)\rd W(r),\ s\in[0,T],\ \text{a.e.}\ t\in[0,T].
\end{split}
\end{equation}
By integrating both sides of \eqref{A_2} and \eqref{A_4} with respect to $t\in[0,T]$, and by using (stochastic) Fubini's theorem (cf.\ Theorem 4.A in \cite{BeMi80}), we get
\begin{equation}\label{A_5}
\begin{split}
	\int^T_0e^{\lambda t}\big|\cY(t,t)\big|^2\rd t&=e^{\lambda T}\int^T_0\big|\cH(t)\big|^2\rd t\\
	&\quad+\int^T_0\int^s_0
		\bigl\{
			-\lambda e^{\lambda s}\big|\cY(t,s)\big|^2+2e^{\lambda s}\cY(t,s)\cG(t,s,\cY(s,s),\cZ(t,s),\zeta(t,s))\\
		&\hspace{2.5cm}-e^{\lambda s}\big|\cZ(t,s)\big|^2
		\bigr\}\rd t\rd s\\
	&\quad-\int^T_0\int^s_02e^{\lambda s}\cY(t,s)\cZ(t,s)\rd t\rd W(s)
\end{split}
\end{equation}
and
\begin{align*}%\label{A_6}
	\int^T_0e^{\lambda(t\vee s)}\big|\cY(t,s)\big|^2\rd t&=e^{\lambda T}\int^T_0\big|\cH(t)\big|^2\rd t\\*
	&\quad+\int^T_s\int^r_0\bigl\{-\lambda e^{\lambda r}\big|\cY(t,r)\big|^2+2e^{\lambda r}\cY(t,r)\cG(t,r,\cY(r,r),\cZ(t,r),\zeta(t,r))\bigr\}\rd t\rd r\\*
	&\quad-\int^T_s\int^T_0e^{\lambda(t\vee r)}\big|\cZ(t,r)\big|^2\rd t\rd r\\
	&\quad-\int^T_s\int^T_02e^{\lambda(t\vee r)}\cY(t,r)\cZ(t,r)\rd t\rd W(r),\ s\in[0,T].
\end{align*}
For any $\varepsilon>0$, define $\widetilde{\cY}^{\lambda, \varepsilon}(s):=\int_{0}^{T} e^{\lambda (t\vee s)}\big|\cY(t,s)\big|^{2}\rd t+\varepsilon$, $s \in [0,T]$.
We note that the process $\widetilde{\cY}^{\lambda, \varepsilon}(\cdot)$ is an It\^{o} process with the terminal condition
$\widetilde{\cY}^{\lambda, \varepsilon}(T)=e^{\lambda T}\int^T_0|\cH(t)|^2\rd t+\varepsilon$.
Thus, by using It\^{o}'s formula for $s\mapsto (\widetilde{\cY}^{\lambda, \varepsilon}(s))^{p/2}$, we obtain
\begin{align*}
	\bigl(\widetilde{\cY}^{\lambda, \varepsilon}(s)\bigr)^{p/2}
	&=
	\Bigl( e^{\lambda T}\int^T_0\big|\cH(t)\big|^2\rd t+\varepsilon\Bigr)^{p/2}\\
	&\quad+
	\int^T_s\frac{p}{2}\bigl(\widetilde{\cY}^{\lambda, \varepsilon}(r)\bigr)^{p/2-1}
	\int^r_0\bigl\{-\lambda e^{\lambda r}\big|\cY(t,r)\big|^2+2e^{\lambda r}\cY(t,r)\cG(t,r,\cY(r,r),\cZ(t,r),\zeta(t,r))\bigr\}\rd t\rd r\\
	&\quad-\int^T_s\frac{p}{2}\bigl(\widetilde{\cY}^{\lambda, \varepsilon}(r)\bigr)^{p/2-1}\int^T_0e^{\lambda(t\vee r)}\big|\cZ(t,r)\big|^2\rd t\rd r\\
	&\quad-\int^T_s\frac{p}{4}\Bigl(\frac{p}{2}-1\Bigr)\bigl(\widetilde{\cY}^{\lambda, \varepsilon}(r)\bigr)^{p/2-2}\Bigl|\int^T_02e^{\lambda(t\vee r)}\cY(t,r)\cZ(t,r)\rd t\Bigr|^2\rd r\\
	&\quad-\int^T_s\frac{p}{2}\bigl(\widetilde{\cY}^{\lambda, \varepsilon}(r)\bigr)^{p/2-1}\int^T_02e^{\lambda(t\vee r)}\cY(t,r)\cZ(t,r)\rd t\rd W(r),~
	s\in[0,T],
\end{align*}
and thus
\begin{equation}\label{A_7}
\begin{split}
	&\bigl(\widetilde{\cY}^{\lambda, \varepsilon}(s)\bigr)^{p/2}
	+
	\frac{p\lambda}{2}\int^T_s\bigl(\widetilde{\cY}^{\lambda, \varepsilon}(r)\bigr)^{p/2-1}e^{\lambda r}\int^r_0\big|\cY(t,r)\big|^2\rd t\rd r
	+\frac{p}{2}\int^T_s\bigl(\widetilde{\cY}^{\lambda, \varepsilon}(r)\bigr)^{p/2-1}\int^T_0e^{\lambda(t\vee r)}\big|\cZ(t,r)\big|^2\rd t\rd r\\
	&+p\Bigl(\frac{p}{2}-1\Bigr)\int^T_s\bigl(\widetilde{\cY}^{\lambda, \varepsilon}(r)\bigr)^{p/2-2}\Bigl|\int^T_0e^{\lambda(t\vee r)}\cY(t,r)\cZ(t,r)\rd t\Bigr|^2\rd r\\
	&=\Bigl(e^{\lambda T}\int^T_0\big|\cH(t)\big|^2\rd t+\varepsilon\Bigr)^{p/2}
	+\frac{p}{2}\int^T_s\bigl(\widetilde{\cY}^{\lambda, \varepsilon}(r)\bigr)^{p/2-1}e^{\lambda r}\int^r_02\cY(t,r)\cG(t,r,\cY(r,r),\cZ(t,r),\zeta(t,r))\rd t\rd r\\
	&\quad-p\int^T_s\bigl(\widetilde{\cY}^{\lambda, \varepsilon}(r)\bigr)^{p/2-1}\int^T_0e^{\lambda(t\vee r)}\cY(t,r)\cZ(t,r)\rd t\rd W(r),\ s\in[0,T].
\end{split}
\end{equation}
By using the above dynamics, we prove \emph{Claim 1} and \emph{Claim 2} below. In the following, we use the notation $\cG_0(t,s):=\cG(t,s,0,0,0)$.

\paragraph{}
\emph{Claim 1}. For any constant $\lambda\geq1+2L^{2}(2T+3)$, it holds that
\begin{align*}
	&\e\Bigl[\Bigl(\int^T_0e^{\lambda t}\big|\cY(t,t)\big|^2\rd t\Bigr)^{p/2}+\Bigl(\int^T_0\int^T_0e^{\lambda(t\vee s)}\big|\cZ(t,s)\big|^2\rd s\rd t\Bigr)^{p/2}\Bigr]\\
	&\leq C_p\e\Bigl[e^{p\lambda T/2}\Bigl(\int^T_0\big|\cH(t)\big|^2 \rd t\Bigr)^{p/2}+\Bigl(\int^T_0\int^T_te^{\lambda s}\big|\cG_0(t,s)\big|^2\rd s\rd t\Bigr)^{p/2}+\sup_{s\in[0,T]}\Bigl(\int^T_0e^{\lambda(t\vee s)}\big|\cY(t,s)\big|^2\rd t\Bigr)^{p/2}\Bigr].
\end{align*}
\emph{Proof of Claim 1}. Fix arbitrary constants $\mu,\lambda>0$. By the Lipschitz continuity of $\cG$ and Young's inequality, it holds that
\begin{equation}\label{A_Young}
\begin{split}
	&2\cY(t,s)\cG(t,s,\cY(s,s),\cZ(t,s),\zeta(t,s))\\
	&\leq\bigl(1+2\mu L^{2}+2L^{2}\bigr)\big|\cY(t,s)\big|^2+\big|\cG_0(t,s)\big|^2+\frac{1}{\mu}\big|\cY(s,s)\big|^2+\frac{1}{2}\big|\cZ(t,s)\big|^2+\frac{1}{\mu}\big|\zeta(t,s)\big|^2
\end{split}
\end{equation}
for any $(t,s)\in\Delta[0,T]$. Thus, by \eqref{A_5}, for any $\mu,\lambda>0$, it holds that
\begin{align*}
	\int^T_0e^{\lambda t}\big|\cY(t,t)\big|^2\rd t &\leq e^{\lambda T}\int^T_0\big|\cH(t)\big|^2\rd t\\*
	&\quad+\int^T_0e^{\lambda s}\Bigl\{\bigl(1+2\mu L^{2}+2L^{2}-\lambda\bigr)\int^s_0\big|\cY(t,s)\big|^2\rd t-\frac{1}{2}\int^s_0\big|\cZ(t,s)\big|^2\rd t\\*
	&\hspace{2.5cm}+\int^s_0\big|\cG_0(t,s)\big|^2\rd t+\frac{T}{\mu}\big|\cY(s,s)\big|^2+\frac{1}{\mu}\int^s_0\big|\zeta(t,s)\big|^2\rd t\Bigr\}\rd s\\
	&\quad-\int^T_0\int^s_02e^{\lambda s}\cY(t,s)\cZ(t,s)\rd t\rd W(s).
\end{align*}
By the assumption \eqref{A_**}, together with \eqref{A_Z}, we have
\begin{equation*}
	\int^s_0\big|\zeta(t,s)\big|^2\rd t\leq\int^s_0\big|\cZ(s,t)\big|^2\rd t\leq \big|\cY(s,s)\big|^2-\int^s_02\cY(s,t)\cZ(s,t)\rd W(t),\ \text{a.e.}\ s\in[0,T].
\end{equation*}
Hence, by the stochastic Fubini's theorem, we obtain
\begin{align*}
	\int^T_0e^{\lambda t}\big|\cY(t,t)\big|^2\rd t &\leq e^{\lambda T}\int^T_0\big|\cH(t)\big|^2\rd t\\
	&\quad+\int^T_0e^{\lambda s}\Bigl\{\bigl(1+2\mu L^{2}+2L^{2}-\lambda\bigr)\int^s_0\big|\cY(t,s)\big|^2\rd t-\frac{1}{2}\int^s_0\big|\cZ(t,s)\big|^2\rd t\\
	&\hspace{2.5cm}+\int^s_0\big|\cG_0(t,s)\big|^2\rd t+\frac{T+1}{\mu}\big|\cY(s,s)\big|^2\Bigr\}\rd s\\
	&\quad-\int^T_0\Bigl\{2\int^s_0 e^{\lambda s} \cY(t,s)\cZ(t,s)\rd t+\frac{2}{\mu}\int^T_se^{\lambda t}\cY(t,s)\cZ(t,s)\rd t\Bigr\}\rd W(s).
\end{align*}
We set $\mu=2(T+1)$. Then for any $\lambda\geq1+2\mu L^{2}+2L^{2}=1+2L^{2}(2T+3)$, we have
\begin{align*}
	&\int^T_0e^{\lambda t}\big|\cY(t,t)\big|^2\rd t+\int^T_0\int^T_te^{\lambda s}\big|\cZ(t,s)\big|^2\rd s\rd t\\
	&\leq2e^{\lambda T}\int^T_0\big|\cH(t)\big|^2\rd t+2\int^T_0\int^T_te^{\lambda s}\big|\cG_0(t,s)\big|^2\rd s\rd t\\
	&\hspace{0.5cm}-\int^T_0\Bigl\{4\int^s_0 e^{\lambda s}\cY(t,s)\cZ(t,s)\rd t+\frac{2}{T+1}\int^T_se^{\lambda t}\cY(t,s)\cZ(t,s)\rd t\Bigr\}\rd W(s).
\end{align*}
Furthermore, again by \eqref{A_Z} and the stochastic Fubini's theorem,
\begin{align*}
	&\int^T_0e^{\lambda t}\big|\cY(t,t)\big|^2\rd t+\int^T_0\int^T_0e^{\lambda(t\vee s)}\big|\cZ(t,s)\big|^2\rd s\rd t\\
	&=\int^T_0e^{\lambda t}\big|\cY(t,t)\big|^2\rd t+\int^T_0\int^T_te^{\lambda s}\big|\cZ(t,s)\big|^2\rd s\rd t+\int^T_0e^{\lambda t}\int^t_0\big|\cZ(t,s)\big|^2\rd s\rd t\\
	&\leq2\int^T_0e^{\lambda t}\big|\cY(t,t)\big|^2\rd t+\int^T_0\int^T_te^{\lambda s}\big|\cZ(t,s)\big|^2\rd s\rd t-\int^T_0e^{\lambda t}\int^t_02\cY(t,s)\cZ(t,s)\rd W(s)\rd t\\
	&\leq4e^{\lambda T}\int^T_0\big|\cH(t)\big|^2\rd t+4\int^T_0\int^T_te^{\lambda s}\big|\cG_0(t,s)\big|^2\rd s\rd t\\
	&\hspace{0.5cm}-\int^T_0\Bigl\{8\int^s_0 e^{\lambda s} \cY(t,s)\cZ(t,s)\rd t+\Bigl(\frac{4}{T+1}+2\Bigr)\int^T_se^{\lambda t}\cY(t,s)\cZ(t,s)\rd t\Bigr\}\rd W(s).
\end{align*}
By using the Burkholder--Davis--Gundy inequality, H\"{o}lder's inequality and Young's inequality, we see that
\begin{align*}
	&\e\Bigl[\Bigl(\int^T_0e^{\lambda t}\big|\cY(t,t)\big|^2\rd t\Bigr)^{p/2}+\Bigl(\int^T_0\int^T_0e^{\lambda(t\vee s)}\big|\cZ(t,s)\big|^2\rd s\rd t\Bigr)^{p/2}\Bigr]\\*
	&\leq C_p\e\Bigl[e^{p\lambda T/2}\Bigl(\int^T_0\big|\cH(t)\big|^2\rd t\Bigr)^{p/2}+\Bigl(\int^T_0\int^T_te^{\lambda s}\big|\cG_0(t,s)\big|^2\rd s\rd t\Bigr)^{p/2}\\
	&\hspace{2cm}+\Bigl(\int^T_0\Bigl(\int^T_0e^{\lambda(t\vee s)}\big|\cY(t,s)\big|\big|\cZ(t,s)\big|\rd t\Bigr)^2\rd s\Bigr)^{p/4}\Bigr]\\
	&\leq C_p\e\Bigl[e^{p\lambda T/2}\Bigl(\int^T_0\big|\cH(t)\big|^2\rd t\Bigr)^{p/2}+\Bigl(\int^T_0\int^T_te^{\lambda s}\big|\cG_0(t,s)\big|^2\rd s\rd t\Bigr)^{p/2}\\
	&\hspace{2cm}+\Bigl(\int^T_0\int^T_0e^{\lambda(t\vee s)}\big|\cY(t,s)\big|^2\rd t\int^T_0e^{\lambda(t\vee s)}\big|\cZ(t,s)\big|^2\rd t\rd s\Bigr)^{p/4}\Bigr]\\
	&\leq C_p\e\Bigl[e^{p\lambda T/2}\Bigl(\int^T_0\big|\cH(t)\big|^2\rd t\Bigr)^{p/2}+\Bigl(\int^T_0\int^T_te^{\lambda s}\big|\cG_0(t,s)\big|^2\rd s\rd t\Bigr)^{p/2}\\
	&\hspace{2cm}+\sup_{s\in[0,T]}\Bigl(\int^T_0e^{\lambda(t\vee s)}\big|\cY(t,s)\big|^2\rd t\Bigr)^{p/4}\Bigl(\int^T_0\int^T_0e^{\lambda(t\vee s)}\big|\cZ(t,s)\big|^2\rd s\rd t\Bigr)^{p/4}\Bigr]\\
	&\leq C_p\e\Bigl[e^{p\lambda T/2}\Bigl(\int^T_0\big|\cH(t)\big|^2\rd t\Bigr)^{p/2}+\Bigl(\int^T_0\int^T_te^{\lambda s}\big|\cG_0(t,s)\big|^2\rd s\rd t\Bigr)^{p/2}+\sup_{s\in[0,T]}\Bigl(\int^T_0e^{\lambda(t\vee s)}\big|\cY(t,s)\big|^2\rd t\Bigr)^{p/2}\Bigr]\\
	&\hspace{0.5cm}+\frac{1}{2}\e\Bigl[\Bigl(\int^T_0\int^T_0e^{\lambda(t\vee s)}\big|\cZ(t,s)\big|^2\rd s\rd t\Bigr)^{p/2}\Bigr].
\end{align*}
Noting that
\begin{equation*}
	\e\Bigl[\Bigl(\int^T_0\int^T_0e^{\lambda(t\vee s)}\big|\cZ(t,s)\big|^2\rd s\rd t\Bigr)^{p/2}\Bigr]\leq T^{p/2}e^{p\lambda T/2}\,\underset{t\in[0,T]}{\mathrm{ess\,sup}}\,\e\Bigl[\Bigl(\int^T_0\big|\cZ(t,s)\big|^2\rd s\Bigr)^{p/2}\Bigr]<\infty,
\end{equation*}
we get the assertion of \emph{Claim 1}.

\paragraph{}
\emph{Claim 2}. For any constants $\mu,\lambda>0$ with $\lambda\geq1+2\mu L^{2}+2L^{2}$, it holds that
\begin{align*}
	&\e\Bigl[\sup_{s\in[0,T]}\Bigl(\int^T_0e^{\lambda(t\vee s)}\big|\cY(t,s)\big|^2\rd t\Bigr)^{p/2}\Bigr]\\
	&\leq C_p\e\Bigl[e^{p\lambda T/2}\Bigl(\int^T_0\big|\cH(t)\big|^2\rd t\Bigr)^{p/2}+\Bigl(\int^T_0\int^T_te^{\lambda s}\big|\cG_0(t,s)\big|^2\rd s\rd t\Bigr)^{p/2}\\
	&\hspace{2cm}+\frac{1}{\mu^{p/2}}\Bigl(\int^T_0e^{\lambda t}\big|\cY(t,t)\big|^2\rd t\Bigr)^{p/2}+\frac{1}{\mu^{p/2}}\Bigl(\int^T_0\int^T_0e^{\lambda(t\vee s)}\big|\cZ(t,s)\big|^2\rd s\rd t\Bigr)^{p/2}\Bigr].
\end{align*}

\emph{Proof of Claim 2}. Fix arbitrary constants $\mu,\lambda>0$. By \eqref{A_7} and \eqref{A_Young}, for each $s\in[0,T]$, we have
\begin{align*}
	&\bigl(\widetilde{\cY}^{\lambda, \varepsilon}(s)\bigr)^{p/2}+\frac{p\lambda}{2}\int^T_s\bigl(\widetilde{\cY}^{\lambda, \varepsilon}(r)\bigr)^{p/2-1}e^{\lambda r}\int^r_0\big|\cY(t,r)\big|^2\rd t\rd r
	\\&
	+\frac{p}{2}\int^T_s\bigl(\widetilde{\cY}^{\lambda, \varepsilon}(r)\bigr)^{p/2-1}\int^T_0e^{\lambda(t\vee r)}\big|\cZ(t,r)\big|^2\rd t\rd r\\*
	&+p\Bigl(\frac{p}{2}-1\Bigr)\int^T_s\bigl(\widetilde{\cY}^{\lambda,\varepsilon}(r)\bigr)^{p/2-2}\Bigl|\int^T_0e^{\lambda(t\vee r)}\cY(t,r)\cZ(t,r)\rd t\Bigr|^2\rd r\\
	&\leq \Bigl(e^{\lambda T}\int^T_0\big|\cH(t)\big|^2\rd t+\varepsilon\Bigr)^{p/2}\\
	&\hspace{0.5cm}+\frac{p}{2}\bigl(1+2\mu L^{2}+2L^{2}\bigr)\int^T_s\bigl(\widetilde{\cY}^{\lambda, \varepsilon}(r)\bigr)^{p/2-1}e^{\lambda r}\int^r_0\big|\cY(t,r)\big|^2\rd t\rd r\\
	&\hspace{0.5cm}+\frac{p}{4}\int^T_s\bigl(\widetilde{\cY}^{\lambda, \varepsilon}(r)\bigr)^{p/2-1}e^{\lambda r}\int^r_0\big|\cZ(t,r)\big|^2\rd t\rd r\\
	&\hspace{0.5cm}+
	\frac{p}{2}\int^T_s
		\bigl(\widetilde{\cY}^{\lambda, \varepsilon}(r)\bigr)^{p/2-1}e^{\lambda r}
		\Big\{
			\int^r_0\big|\cG_0(t,r)\big|^2\rd t
			+
			\frac{T}{\mu}
			\big|\cY(r,r)\big|^2
			+
			\frac{1}{\mu}
			\int^r_0\big|\zeta(t,r)\big|^2\rd t
		\Big\}
	\rd r
	\\
	&\hspace{0.5cm}-p\int^T_s\bigl(\widetilde{\cY}^{\lambda, \varepsilon}(r)\bigr)^{p/2-1}\int^T_0e^{\lambda(t\vee r)}\cY(t,r)\cZ(t,r)\rd t\rd W(r).
\end{align*}
Noting the assumption \eqref{A_**}, when $\lambda\geq1+2\mu L^{2}+2L^{2}$, we have, for each $s\in[0,T]$,
\begin{equation}\label{A_8}
\begin{split}
	&\bigl(\widetilde{\cY}^{\lambda, \varepsilon}(s)\bigr)^{p/2}
	+
	p\Bigl(\frac{p}{2}-1\Bigr)\int^T_s\bigl(\widetilde{\cY}^{\lambda, \varepsilon}(r)\bigr)^{p/2-2}\Bigl|\int^T_0e^{\lambda(t\vee r)}\cY(t,r)\cZ(t,r)\rd t\Bigr|^2\rd r\\
	&\leq \Bigl(e^{\lambda T}\int^T_0\big|\cH(t)\big|^2\rd t+\varepsilon\Bigr)^{p/2}
	\\
	&\quad+
	\frac{p}{2}\int^T_s
		\bigl(\widetilde{\cY}^{\lambda, \varepsilon}(r)\bigr)^{p/2-1}e^{\lambda r}
		\Big\{
			\int^r_0\big|\cG_0(t,r)\big|^2\rd t
			+
			\frac{T}{\mu}
			\big|\cY(r,r)\big|^2
			+
			\frac{1}{\mu}
			\int^r_0\big|\cZ(r,t)\big|^2\rd t
		\Big\}
	\rd r
	\\&\hspace{0.5cm}-p\int^T_s\bigl(\widetilde{\cY}^{\lambda, \varepsilon}(r)\bigr)^{p/2-1}\int^T_0e^{\lambda(t\vee r)}\cY(t,r)\cZ(t,r)\rd t\rd W(r).
\end{split}
\end{equation}
Thus, by the Burkholder--Davis--Gundy inequality, we get
\begin{align*}
	&\e\Bigl[\sup_{s\in[0,T]}\bigl(\widetilde{\cY}^{\lambda, \varepsilon}(s)\bigr)^{p/2}\Bigr]\\
	&\leq C_p\e\Bigl[\Bigl(e^{\lambda T}\int^T_0\big|\cH(t)\big|^2\rd t+\varepsilon\Bigr)^{p/2}
	\\
	&\hspace{1.5cm}+
	\int^T_0
		\bigl(\widetilde{\cY}^{\lambda, \varepsilon}(s)\bigr)^{p/2-1}e^{\lambda s}
		\Big\{
			\int^s_0\big|\cG_0(t,s)\big|^2\rd t
			+
			\frac{1}{\mu}
			\big|\cY(s,s)\big|^2
			+
			\frac{1}{\mu}
			\int^s_0\big|\cZ(s,t)\big|^2\rd t
		\Big\}
	\rd s
	\\
	&\hspace{1.5cm}+\Bigl\{\int^T_0\bigl(\widetilde{\cY}^{\lambda, \varepsilon}(s)\bigr)^{p-2}\Bigl|\int^T_0e^{\lambda(t\vee s)}\cY(t,s)\cZ(t,s)\rd t\Bigr|^2\rd s\Bigr\}^{1/2}\Bigr]\\
	&\leq C_p\e\Bigl[\Bigl(e^{\lambda T}\int^T_0\big|\cH(t)\big|^2\rd t+\varepsilon\Bigr)^{p/2}\\
	&\hspace{1.5cm}+\sup_{s\in[0,T]}\bigl(\widetilde{\cY}^{\lambda, \varepsilon}(s)\bigr)^{p/2-1}
	\Bigl\{\int^T_0\int^T_te^{\lambda s}\big|\cG_0(t,s)\big|^2\rd s\rd t+\frac{1}{\mu}\int^T_0e^{\lambda t}\big|\cY(t,t)\big|^2\rd t\\&\hspace{6cm}+\frac{1}{\mu}\int^T_0\int^T_0e^{\lambda(t\vee s)}\big|\cZ(t,s)\big|^2\rd s\rd t\Bigr\}\\
	&\hspace{1.5cm}+\sup_{s\in[0,T]}\bigl(\widetilde{\cY}^{\lambda, \varepsilon}(s)\bigr)^{p/4}\Bigl\{\int^T_0\bigl(\widetilde{\cY}^{\lambda, \varepsilon}(s)\bigr)^{p/2-2}\Bigl|\int^T_0e^{\lambda(t\vee s)}\cY(t,s)\cZ(t,s)\rd t\Bigr|^2\rd s\Bigr\}^{1/2}\Bigr].
\end{align*}
Furthermore, by Young's inequality,
\begin{align*}
	\e\Bigl[\sup_{s\in[0,T]}\bigl(\widetilde{\cY}^{\lambda, \varepsilon}(s)\bigr)^{p/2}\Bigr]
	&\leq C_p\e\Bigl[\Bigl(e^{\lambda T}\int^T_0\big|\cH(t)\big|^2\rd t+\varepsilon\Bigr)^{p/2}+\Bigl(\int^T_0\int^T_te^{\lambda s}\big|\cG_0(t,s)\big|^2\rd s\rd t\Bigr)^{p/2}\\*
	&\hspace{1.5cm}+\frac{1}{\mu^{p/2}}\Bigl(\int^T_0e^{\lambda t}\big|\cY(t,t)\big|^2\rd t\Bigr)^{p/2}+\frac{1}{\mu^{p/2}}\Bigl(\int^T_0\int^T_0e^{\lambda(t\vee s)}\big|\cZ(t,s)\big|^2\rd s\rd t\Bigr)^{p/2}\\*
	&\hspace{1.5cm}+\int^T_0\bigl(\widetilde{\cY}^{\lambda, \varepsilon}(s)\bigr)^{p/2-2}\Bigl|\int^T_0e^{\lambda(t\vee s)}\cY(t,s)\cZ(t,s)\rd t\Bigr|^2\rd s\Bigr]\\*
	&\hspace{0.5cm}+\frac{1}{2}\e\Bigl[\sup_{s\in[0,T]}\bigl(\widetilde{\cY}^{\lambda, \varepsilon}(s)\bigr)^{p/2}\Bigr],
\end{align*}
and thus, it holds that
\begin{equation}\label{A_9}
\begin{split}
	\e\Bigl[\sup_{s\in[0,T]}\bigl(\widetilde{\cY}^{\lambda, \varepsilon}(s)\bigr)^{p/2}\Bigr]
	&\leq C_p\e\Bigl[\Bigl(e^{\lambda T}\int^T_0\big|\cH(t)\big|^2\rd t+\varepsilon\Bigr)^{p/2}+\Bigl(\int^T_0\int^T_te^{\lambda s}\big|\cG_0(t,s)\big|^2\rd s\rd t\Bigr)^{p/2}\\
	&\hspace{1.5cm}+\frac{1}{\mu^{p/2}}\Bigl(\int^T_0e^{\lambda t}\big|\cY(t,t)\big|^2\rd t\Bigr)^{p/2}+\frac{1}{\mu^{p/2}}\Bigl(\int^T_0\int^T_0e^{\lambda(t\vee s)}\big|\cZ(t,s)\big|^2\rd s\rd t\Bigr)^{p/2}\\
	&\hspace{1.5cm}+\int^T_0\bigl(\widetilde{\cY}^{\lambda, \varepsilon}(s)\bigr)^{p/2-2}\Bigl|\int^T_0e^{\lambda(t\vee s)}\cY(t,s)\cZ(t,s)\rd t\Bigr|^2\rd s\Bigr].
\end{split}
\end{equation}
On the other hand, noting that $\cY(\cdot,\cdot)$ is bounded and
\begin{align*}
	\e\Bigl[\int^T_0\Bigl(\int^T_0\big|\cZ(t,s)\big|\rd t\Bigr)^2\rd s\Bigr]\leq T^2\,\underset{t\in[0,T]}{\mathrm{ess\,sup}}\,\e\Bigl[\int^T_0\big|\cZ(t,s)\big|^2\rd s\Bigr]<\infty,
\end{align*}
by letting $s=0$ in \eqref{A_8} and taking expectations on both sides, we have
\begin{align*}
	&\e\Bigl[\int^T_0\bigl(\widetilde{\cY}^{\lambda, \varepsilon}(s)\bigr)^{p/2-2}\Bigl|\int^T_0e^{\lambda(t\vee s)}\cY(t,s)\cZ(t,s)\rd t\Bigr|^2\rd s\Bigr]\\*
	&\leq C_p\e\Bigl[\Bigl(e^{\lambda T}\int^T_0\big|\cH(t)\big|^2\rd t+\varepsilon\Bigr)^{p/2}
	\\
	&\hspace{1.5cm}+
	\int^T_0\bigl(\widetilde{\cY}^{\lambda, \varepsilon}(s)\bigr)^{p/2-1}e^{\lambda s}
	\Big\{
		\int^s_0\big|\cG_0(t,s)\big|^2\rd t
		+
		\frac{1}{\mu}
		\big|\cY(s,s)\big|^2
		+
		\frac{1}{\mu}
		\int^s_0\big|\cZ(s,t)\big|^2\rd t
	\Big\}
	\rd s
	\Big],
\end{align*}
and hence
\begin{equation}\label{A_10}
\begin{split}
	&\e\Bigl[\int^T_0\bigl(\widetilde{\cY}^{\lambda, \varepsilon}(s)\bigr)^{p/2-2}\Bigl|\int^T_0e^{\lambda(t\vee s)}\cY(t,s)\cZ(t,s)\rd t\Bigr|^2\rd s\Bigr]\\
	&\leq C_p\e\Bigl[\Bigl(e^{\lambda T}\int^T_0\big|\cH(t)\big|^2\rd t+\varepsilon\Bigr)^{p/2}\\
	&\hspace{1.5cm}+\sup_{s\in[0,T]}\bigl(\widetilde{\cY}^{\lambda, \varepsilon}(s)\bigr)^{p/2-1}
	\Bigl\{\int^T_0\int^T_te^{\lambda s}\big|\cG_0(t,s)\big|^2\rd s\rd t+\frac{1}{\mu}\int^T_0e^{\lambda t}\big|\cY(t,t)\big|^2\rd t
	\\&\hspace{6cm}+\frac{1}{\mu}\int^T_0\int^T_0e^{\lambda(t\vee s)}\big|\cZ(t,s)\big|^2\rd s\rd t\Bigr\}\Bigr].
\end{split}
\end{equation}
By \eqref{A_9} and \eqref{A_10}, together with Young's inequality, we can show that
\begin{align*}
	\e\Bigl[\sup_{s\in[0,T]}\bigl(\widetilde{\cY}^{\lambda, \varepsilon}(s)\bigr)^{p/2}\Bigr]
	&\leq C_p\e\Bigl[\Bigl(e^{\lambda T}\int^T_0\big|\cH(t)\big|^2\rd t+\varepsilon\Bigr)^{p/2}+\Bigl(\int^T_0\int^T_te^{\lambda s}\big|\cG_0(t,s)\big|^2\rd s\rd t\Bigr)^{p/2}\\*
	&\hspace{1.5cm}+\frac{1}{\mu^{p/2}}\Bigl(\int^T_0e^{\lambda t}\big|\cY(t,t)\big|^2\rd t\Bigr)^{p/2}+\frac{1}{\mu^{p/2}}\Bigl(\int^T_0\int^T_0e^{\lambda(t\vee s)}\big|\cZ(t,s)\big|^2\rd s\rd t\Bigr)^{p/2}\Bigr].
\end{align*}
Recall that $\widetilde{\cY}^{\lambda, \varepsilon}(s)=\int_{0}^{T} e^{\lambda (t\vee s)}|\cY(t,s)|^{2}\rd t+\varepsilon$, $s \in [0,T]$, and that $\cY(\cdot,\cdot)$ and $\cH(\cdot)$ are bounded.
Thus, by letting $\varepsilon \downarrow 0$ and using the dominated convergence theorem, we see that the assertion of \emph{Claim 2} holds.

\paragraph{}
By \emph{Claim 1} and \emph{Claim 2}, for a sufficiently large $\lambda>0$ (which depends only on $p$, $L$ and $T$), it holds that
\begin{align*}
	&\e\Bigl[\sup_{s\in[0,T]}\Bigl(\int^T_0e^{\lambda(t\vee s)}\big|\cY(t,s)\big|^2\rd t\Bigr)^{p/2}+\Bigl(\int^T_0e^{\lambda t}\big|\cY(t,t)\big|^2\rd t\Bigr)^{p/2}+\Bigl(\int^T_0\int^T_0e^{\lambda(t\vee s)}\big|\cZ(t,s)\big|^2\rd s\rd t\Bigr)^{p/2}\Bigr]\\
	&\leq C_p\e\Bigl[e^{p\lambda T/2}\Bigl(\int^T_0\big|\cH(t)\big|^2\rd t\Bigr)^{p/2}+\Bigl(\int^T_0\int^T_te^{\lambda s}\big|\cG_0(t,s)\big|^2\rd s\rd t\Bigr)^{p/2}\Bigr].
\end{align*}
Therefore, the desired estimate \eqref{A_***} holds, and we finish the proof of the lemma.
\end{proof}

Now we are ready to prove the $L^p$-a priori estimates of the solutions of Type-\Rnum{2} BSVIEs and BSDE systems for $p\geq2$.

\begin{proof}[Proof of Theorem~\ref{Lem_Lp}]
For simplicity of notation, we let $m=d=1$. The proof can be easily generalized to the multi-dimensional case.

First, we prove (\rnum{1}). For each $n\in\n$, define
\begin{equation*}
	\Psi_{n}(t):=(-n)\vee\bigl(\Psi(t)\wedge n\bigr)\ \text{and}\ G_n(t,s,y,z_1,z_2):=(-n)\vee\bigl(G(t,s,y,z_1,z_2)\wedge n\bigr).
\end{equation*}
Then $(\Psi_n,G_n)$ satisfies the conditions (\rnum{1}), (\rnum{2}) and (\rnum{3}) in $\mathrm{(H}_{\Psi,G}\mathrm{)}$ with the common constant $L$, and $\Psi_n$ and $G_n$ are bounded. Denote by $(Y_n(\cdot),Z_n(\cdot,\cdot))\in L^2_\mathbb{F}(0,T;\real)\times L^2(0,T;L^2_\mathbb{F}(0,T;\real))$ the adapted M-solution of Type-\Rnum{2} BSVIE \eqref{BSVIE_II_1} corresponding to the coefficients $(\Psi,G)=(\Psi_n,G_n)$, and define
\begin{equation*}
	Y_n(t,s):=\e_s\Bigl[\Psi_n(t)+\int^T_sG_n(t,r,Y_n(r),Z_n(t,r),Z_n(r,t))\1_{[t,T]}(r)\rd r\Bigr],\ (t,s)\in[0,T]^2.
\end{equation*}
Then $Y_n(\cdot,\cdot)\in L^2(0,T;L^2_\mathbb{F}(\Omega;C([0,T];\real)))$, and $Y_n(t,t)=Y_n(t)$, a.e.\ $t\in[0,T]$, a.s.
Furthermore, by the same arguments as in the proof of Proposition \ref{Prop_BSDE_0}, we see that the following holds:
\begin{align*}
	&Y_n(t,s)=\Psi_n(t)+\int^T_sG_n(t,r,Y_n(r,r),Z_n(t,r),Z_n(r,t))\1_{[t,T]}(r)\rd r-\int^T_sZ_n(t,r)\rd W(r),\\
	&\hspace{5cm}s\in[0,T],\ \text{a.e.}\ t\in[0,T],\ \text{a.s.}
\end{align*}
Noting that the process $s\mapsto Z_n(s,t)\1_{[t,T]}(s)$ is adapted for each $t\in[0,T]$, we apply Lemma~\ref{lemma_A} with
\begin{equation*}
	\cH=\Psi_n,\ \cG=G_n,\ \cY(t,s)=Y_n(t,s),\ \cZ(t,s)=Z_n(t,s),\ \zeta(t,s)=Z_n(s,t)\1_{[t,T]}(s).
\end{equation*}
Then we obtain
\begin{align*}
	&\e\Bigl[\sup_{s\in[0,T]}\Bigl(\int^T_0\big|Y_n(t,s)\big|^2\rd t\Bigr)^{p/2}+\Bigl(\int^T_0\big|Y_n(t)\big|^2\rd t\Bigr)^{p/2}+\Bigl(\int^T_0\int^T_0\big|Z_n(t,s)\big|^2\rd s\rd t\Bigr)^{p/2}\Bigr]\\
	&\leq C_p\e\Bigl[\Bigl(\int^T_0\big|\Psi_n(t)\big|^2\rd t\Bigr)^{p/2}+\Bigl(\int^T_0\int^T_t\big|G_n(t,s,0,0,0)\big|^2\rd s\rd t\Bigr)^{p/2}\Bigr]\\
	&\leq C_p\e\Bigl[\Bigl(\int^T_0\big|\Psi(t)\big|^2\rd t\Bigr)^{p/2}+\Bigl(\int^T_0\int^T_t\big|G(t,s,0,0,0)\big|^2\rd s\rd t\Bigr)^{p/2}\Bigr].
\end{align*}
We emphasize that the constant $C_p>0$ does not depend on $n\in\n$. By the $L^2$-stability estimate for Type-\Rnum{2} BSVIEs (cf.\ Theorem~3.7 in \cite{Yo08}), we see that
\begin{equation*}
	\lim_{n\to\infty}(Y_n(\cdot),Z_n(\cdot,\cdot))=(Y(\cdot),Z(\cdot,\cdot))\ \text{in}\ L^2_\mathbb{F}(0,T;\real)\times L^2(0,T;L^2_\mathbb{F}(0,T;\real)).
\end{equation*}
Furthermore, we have
\begin{align*}
	&\e\Bigl[\sup_{s\in[0,T]}\int^T_0\big|Y_n(t,s)-Y(t,s)\big|^2\rd t\Bigr]\\
	&\leq\int^T_0\e\Bigl[\sup_{s\in[0,T]}\e_s\Bigl[\big|\Psi_n(t)-\Psi(t)\big|\\
	&\hspace{3cm}+\int^T_t\bigl|G_n(t,r,Y_n(r),Z_n(t,r),Z_n(r,t))-G(t,r,Y(r),Z(t,r),Z(r,t))\bigr|\rd r\Bigr]^2\Bigr]\rd t\\
	&\leq4\int^T_0\e\Bigl[\Bigl\{\big|\Psi_n(t)-\Psi(t)\big|\\
	&\hspace{2cm}+\int^T_t\bigl|G_n(t,r,Y_n(r),Z_n(t,r),Z_n(r,t))-G(t,r,Y(r),Z(t,r),Z(r,t))\bigr|\rd r\Bigr\}^2\Bigr]\rd t\\
	&\leq C\e\Bigl[\int^T_0\big|\Psi_n(t)-\Psi(t)\big|^2\rd t+\int^T_0\int^T_t\bigl|G_n(t,s,Y(s),Z(t,s),Z(s,t))-G(t,s,Y(s),Z(t,s),Z(s,t))\bigr|^2\rd s\rd t\\
	&\hspace{2cm}+\int^T_0\big|Y_n(t)-Y(t)\big|^2\rd t+\int^T_0\int^T_0\big|Z_n(t,s)-Z(t,s)\big|^2\rd s\rd t\Bigr],
\end{align*}
where in the second inequality we used Doob's martingale inequality. By the dominated convergence theorem, the last term in the above inequalities tends to zero as $n\to\infty$. Thus, there exists a subsequence $\{n_i\}_{i\in\n}\subset\n$ such that
\begin{align*}
	&\lim_{i\to\infty}\int^T_0\big|Y_{n_i}(t)\big|^2\rd t=\int^T_0\big|Y(t)\big|^2\rd t,\ \text{a.s.},\\
	&\lim_{i\to\infty}\int^T_0\int^T_0|Z_{n_i}(t,s)|^2\rd s\rd t=\int^T_0\int^T_0\big|Z(t,s)\big|^2\rd s\rd t,\ \text{a.s.},\\*
	&\lim_{i\to\infty}\sup_{s\in[0,T]}\int^T_0\big|Y_{n_i}(t,s)\big|^2\rd t=\sup_{s\in[0,T]}\int^T_0\big|Y(t,s)\big|^2\rd t,\ \text{a.s.}
\end{align*}
Therefore, by Fatou's lemma, we get
\begin{align*}
	&\e\Bigl[\sup_{s\in[0,T]}\Bigl(\int^T_0\big|Y(t,s)\big|^2\rd t\Bigr)^{p/2}+\Bigl(\int^T_0\big|Y(t)\big|^2\rd t\Bigr)^{p/2}+\Bigl(\int^T_0\int^T_0\big|Z(t,s)\big|^2\rd s\rd t\Bigr)^{p/2}\Bigr]\\*
	&\leq\liminf_{i\to\infty}\e\Bigl[\sup_{s\in[0,T]}\Bigl(\int^T_0\big|Y_{n_i}(t,s)\big|^2\rd t\Bigr)^{p/2}+\Bigl(\int^T_0\big|Y_{n_i}(t)\big|^2\rd t\Bigr)^{p/2}+\Bigl(\int^T_0\int^T_0\big|Z_{n_i}(t,s)\big|^2\rd s\rd t\Bigr)^{p/2}\Bigr]\\
	&\leq C_p\e\Bigl[\Bigl(\int^T_0|\Psi(t)|^2\rd t\Bigr)^{p/2}+\Bigl(\int^T_0\int^T_t\big|G(t,s,0,0,0)\big|^2\rd s\rd t\Bigr)^{p/2}\Bigr].
\end{align*}
This completes the proof of the assertion (\rnum{1}).

Next, we prove (\rnum{2}).
For each $n \in \n$, define $(\Psi_n,G_n)$ as above.
Denote by $\{(\mathscr{Y}^{\pi,\theta}_n(t_k,\cdot),\mathscr{Z}^{\pi,\theta}_n(t_k,\cdot))\}^{N-1}_{k=0}$ the solution of BSDE system \eqref{BSDE_sys_0} corresponding to $\pi=\{t_0,t_1,\ldots,t_N\}\in\Pi[0,T]$, $\theta\in[0,T]$ and the coefficients $(\Psi,G)=(\Psi_n,G_n)$. We apply Lemma~\ref{lemma_A} with
\begin{align*}
	&\cH(t)=\Psi_n(\tau(t)),~
	\cG(t,s,y,z_1,z_2)=G_n(\tau(t),s,y,z_1,z_2)\1_{[\tau^*(t),T)}(s),\\
	&\cY(t,s)=\mathscr{Y}^{\pi,\theta}_n(\tau(t),s),~
	\cZ(t,s)=\mathscr{Z}^{\pi,\theta}_n(\tau(t),s),~
	\zeta(t,s)=\mathcal{I}^{\pi,\theta}[\mathscr{Z}^{\pi,\theta}_n(\tau(s),\cdot)](\tau(t))\1_{[\tau^*(t),T)}(s).
\end{align*}
We see that the triplet $(\cY(\cdot,\cdot),\cZ(\cdot,\cdot),\zeta(\cdot,\cdot))$ satisfies the equation \eqref{A_*}.
Also, for each $\ell=1,\ldots,N-1$ and $s\in[t_\ell,t_{\ell+1})$, we have
\begin{align*}
	\int^s_0\big|\zeta(t,s)\big|^2\rd t&=\int^s_0\bigl|\mathcal{I}^{\pi,\theta}[\mathscr{Z}^{\pi,\theta}_n(\tau(s),\cdot)](\tau(t))\1_{[\tau^*(t),T)}(s)\bigr|^2\rd t=\int^{\tau(s)}_0\bigl|\mathcal{I}^{\pi,\theta}[\mathscr{Z}^{\pi,\theta}_n(\tau(s),\cdot)](\tau(t))\bigr|^2\rd t\\*
	&=\sum^{\ell-1}_{k=0}\Delta t_k\bigl|\mathcal{I}^{\pi,\theta}[\mathscr{Z}^{\pi,\theta}_n(t_\ell,\cdot)](t_k)\bigr|^2=\sum^{\ell-1}_{k=0}\Delta t_k\Bigl|\frac{1}{\Delta t_k}\int^{t_{k+1}\vee\theta}_{t_k\vee\theta}\mathscr{Z}^{\pi,\theta}_n(t_\ell,t)\rd t\Bigr|^2\\
	&\leq\sum^{\ell-1}_{k=0}\int^{t_{k+1}}_{t_k}\big|\mathscr{Z}^{\pi,\theta}_n(t_\ell,t)\big|^2\rd t=\int^{\tau(s)}_0|\mathscr{Z}^{\pi,\theta}_n(\tau(s),t)|^2\rd t\\
	&\leq\int^s_0\big|\cZ(s,t)\big|^2\rd t.
\end{align*}
Thus, the inequality~\eqref{A_**} holds. By Lemma~\ref{lemma_A}, we have
\begin{align*}
	&\e\Bigl[\sup_{s\in[0,T]}\Bigl(\int^T_0\big|\mathscr{Y}^{\pi,\theta}_n(\tau(t),s)\big|^2\rd t\Bigr)^{p/2}+\Bigl(\int^T_0\big|\mathscr{Y}^{\pi,\theta}_n(\tau(t),t)\big|^2\rd t\Bigr)^{p/2}+\Bigl(\int^T_0\int^T_0\big|\mathscr{Z}^{\pi,\theta}_n(\tau(t),s)\big|^2\rd s\rd t\Bigr)^{p/2}\Bigr]\\*
	&\leq C_p\e\Bigl[\Bigl(\int^T_0\big|\Psi_n(\tau(t))\big|^2\rd t\Bigr)^{p/2}+\Bigl(\int^T_0\int^T_{\tau^*(t)}\big|G_n(\tau(t),s,0,0,0)\big|^2\rd s\rd t\Bigr)^{p/2}\Bigr]\\
	&\leq C_p\e\Bigl[\Bigl(\int^T_0\big|\Psi(\tau(t))\big|^2\rd t\Bigr)^{p/2}+\Bigl(\int^T_0\int^T_{\tau^*(t)}\big|G(\tau(t),s,0,0,0)\big|^2\rd s\rd t\Bigr)^{p/2}\Bigr].
\end{align*}
By using the $L^2$-stability estimate for BSDE systems (see Lemma~\ref{Lem_BSDE_sys_0}), together with Fatou's lemma, we can show that
\begin{align*}
	&\e\Bigl[\sup_{s\in[0,T]}\Bigl(\int^T_0\big|\mathscr{Y}^{\pi,\theta}(\tau(t),s)\big|^2\rd t\Bigr)^{p/2}+\Bigl(\int^T_0\big|\mathscr{Y}^{\pi,\theta}(\tau(t),t)\big|^2\rd t\Bigr)^{p/2}+\Bigl(\int^T_0\int^T_0\big|\mathscr{Z}^{\pi,\theta}(\tau(t),s)\big|^2\rd s\rd t\Bigr)^{p/2}\Bigr]\\
	&\leq C_p\e\Bigl[\Bigl(\int^T_0\big|\Psi(\tau(t))\big|^2\rd t\Bigr)^{p/2}+\Bigl(\int^T_0\int^T_{\tau^*(t)}\big|G(\tau(t),s,0,0,0)\big|^2\rd s\rd t\Bigr)^{p/2}\Bigr].
\end{align*}
This completes the proof of the assertion~(\rnum{2}).
\end{proof}

\end{appendix}

\section*{Acknowledgments}
%The authors would like to thank the editor and the referees for their constructive comments and suggestions.
The first author was supported by JSPS KAKENHI Grant Number JP18J20973.
The second author was supported by JSPS KAKENHI Grant Number 19K14552.

%\bibliography{reference}

\end{document}